\documentclass[12pt]{amsart}
\usepackage{amssymb}
\usepackage{amsbsy}
\usepackage{amscd}
\usepackage[mathscr]{eucal}
\usepackage{verbatim}
\usepackage{version}
\usepackage[all]{xy}
\allowdisplaybreaks
\usepackage{pstricks}
\usepackage{pst-all}

\newenvironment{LG}{{\bf LG}. \footnotesize \red}{}

\makeatletter
\def\cal{\mathcal}

\def\frak{\mathfrak}

\newenvironment{pf}{\proof[\proofname]}{\endproof}
\newenvironment{pf*}[1]{\proof[#1]}{\endproof}
\makeatother

\makeatletter
\@ifclasslater{amsart}{1999/11/24}{}{
\renewcommand*\subjclass[2][1991]{%
  \def\@subjclass{#2}%
  \@ifundefined{subjclassname@#1}{%
    \ClassWarning{\@classname}{Unknown edition (#1) of Mathematics
      Subject Classification; using '1991'.}%
  }{%
    \@xp\let\@xp\subjclassname\csname subjclassname@#1\endcsname
  }%
}
\renewcommand{\subjclassname}{%
  \textup{1991} Mathematics Subject Classification}
\@xp\let\csname subjclassname@1991\endcsname \subjclassname
\@namedef{subjclassname@2000}{%
  \textup{2000} Mathematics Subject Classification}
}
\makeatother
\newtheorem{Theorem}[equation]{Theorem}
\newtheorem{Corollary}[equation]{Corollary}
\newtheorem{Lemma}[equation]{Lemma}
\newtheorem{Proposition}[equation]{Proposition}

\theoremstyle{definition}
\newtheorem{Definition}[equation]{Definition}

\newtheorem{Notation}[equation]{Notation}

\newtheorem{Conjecture}[equation]{Conjecture}

\theoremstyle{remark}
\newtheorem{Remark}[equation]{Remark}
\newtheorem*{Claim}{Claim}

\numberwithin{equation}{section}

\newcommand{\thmref}[1]{Theorem~\ref{#1}}
\newcommand{\secref}[1]{\S\ref{#1}}
\newcommand{\lemref}[1]{Lemma~\ref{#1}}
\newcommand{\propref}[1]{Proposition~\ref{#1}}
\newcommand{\corref}[1]{Corollary~\ref{#1}}

\newcommand{\defref}[1]{Definition~\ref{#1}}
\newcommand{\remref}[1]{Remark~\ref{#1}}
\newcommand{\conref}[1]{Conjecture~\ref{#1}}
\newcommand{\C}{{\Bbb C}}
\newcommand{\Z}{{\Bbb Z}}
\newcommand{\Q}{{\Bbb Q}}
\newcommand{\R}{{\Bbb R}}

\newcommand{\la}{{\frak \lambda}}

\newcommand{\Tor}{\underline{Tor}}
\newcommand{\D}{{\mathcal D}}
\newcommand{\ra}{\rightarrow}

\newcommand{\Pic}{\operatorname{Pic}}

\newcommand{\rk}{\mathop{{\rm rk}}}

\newcommand{\Coeff}{\mathop{\text{\rm Coeff}}}
\def\oo{{\cal O}}

\def\P{{\mathbb P}}
\def\<{\langle}
\def\>{\rangle}

\def\F{{\mathcal F}}
\def\PP{{\mathcal P}}
\def\G{{\mathcal G}}
\def\E{{\mathcal E}}
\def\cc{{\mathcal C}}

\def\WT{{\widetilde \theta}}
\def\eps{{\epsilon}}

\def\th{\theta}

\makeatletter
\newcommand{\vechatom}{
    {\Vec{\omega}}
    \,\smash[b]{\hbox{\lower2\ex@\hbox{$\m@th\hat{\null}$}}}
}
\makeatother

\def\R{{\mathbb R}}
\def\RR{{\mathcal R}}
\def\Q{{\mathbb Q}}

\def\P{{\mathbb P}}
\def\Z{{\mathbb Z}}
\def\C{{\mathbb C}}
\def\H{{\mathcal H}}
\def\oo{{\mathcal O}}

\begin{document}

\title[Verlinde-type formulas for  rational surfaces]
{Verlinde-type formulas for rational surfaces}
\author{Lothar G\"ottsche}
\address{International Centre for Theoretical Physics, Strada Costiera 11, 
34014 Trieste, Italy}
\email{gottsche@ictp.it}

\begin{abstract} For a projective algebraic surface $X$, with an ample line bundle $\omega$, let $M_\omega^X(c_1,d)$ be the moduli space of rank $2$, $\omega$-semistable torsion free sheaves $E$ with $c_1(E)=c_1$ and $4c_2(E)-c_1^2=d$. For   line bundles $L$ on $X$, let $\mu(L)$ be the corresponding determinant line bundles  on $M_H^X(c_1,d)$.
The $K$-theoretic Donaldson invariants are the holomorphic Euler characteristics $\chi(M_\omega^X(c_1,d),\mu(L))$.
In this paper we develop an algorithm which in principle determines all their generating functions for the projective plane, its blowup in finitely many points, and also for $\P^1\times\P^1$.
Among others, we apply this algorithm to compute the generating functions of the $\chi(M_H^{\P^2}(0,d),\mu (nH))$  and $\chi(M_H^{\P^2}(H,d),\mu (nH))$ for $n\le 11$, for $H$ the hyperplane class on $\P^2$.
We give  some conjectures about the general structure of  these generating functions and interpret them in terms of  Le Potier's strange duality conjecture.
\end{abstract}

\maketitle
\tableofcontents

\section{Introduction}
In this paper let $(X,\omega)$ be a pair of a rational surface $X$ and an ample line bundle $\omega$.
We consider the moduli spaces 
$M^X_\omega(c_1,d)$ of $\omega$-semistable torsion-free coherent sheaves of rank $2$ on $X$ with Chern classes $c_1\in H^2(X,\Z)$ and $c_2$ such that $d=4c_2-c_1^2$.
Associated to  a  line bundle $L$ on $X$ there is a determinant bundle $\mu(L)\in\Pic(M^X_\omega(c_1,d))$. If $L$ is ample, then $\mu(L)$ is nef and big
on $M^X_\omega(c_1,d)$, and a suitable power induces the map from $M^X_\omega(c_1,d)$ to the corresponding Uhlenbeck compactification.

If one considers instead of a rational surface $X$ a curve $C$, the spaces of sections of the corresponding determinant bundles are the spaces of conformal blocks, and 
their dimensions are given by the celebrated Verlinde formula. In  \cite{Zag} many reformulations of this formula are given.  In particular
[Thm.~1.(vi)]\cite{Zag}  expresses the generating function on a fixed curve as a rational function. 
In this paper we  study the generating functions of the holomorphic Euler characteristics $\chi(M^X_\omega(c_1,d),\mu(L))$, and show that they are given as rational functions.
Let 
$$\chi_{c_1}^{X,\omega}(L):=\sum_{d>0}\chi(M^X_\omega(c_1,d),\mu(L))\Lambda^d$$
(In case $c_1=0$ the coefficient of $\Lambda^4$ is slightly different, furthermore in case $\omega$ lies on a wall (see below), here instead of $\chi(M^X_\omega(c_1,d),\mu(L))$
we use the average over the chambers adjacent to $\omega$.) We can view the spaces of sections $H^0(M^X_\omega(c_1,d),\mu(L))$ as analogues to the spaces of conformal blocks. In most cases we will consider (see \propref{highvan} below), the higher cohomology groups of the determinant bundle 
$\mu(L)$ vanish. Thus  our formulas for the $\chi_{c_1}^{X,\omega}(L)$  are analogues of the Verlinde formula for rational surfaces.

\begin{Notation}
For two Laurent series $P(\Lambda)=\sum_{n} a_n\Lambda^n,Q(\Lambda)=\sum_{n} b_n\Lambda^n\in \Q[\Lambda^{-1}][[\Lambda]]$ we write $P(\Lambda)\equiv Q(\Lambda)$ if 
there is an $n_0\in \Z$ with $a_n=b_n$ for all $n\ge n_0$.
\end{Notation}

\begin{Theorem}
\label{rationalal}
Let $X$ be $\P^2$, $\P^1\times \P^1$ or a blowup of $\P^2$ in $n$  points. 
Let $c_1\in H^2(X,\Z)$, $L\in Pic(X)$. 
There is a polynomial $P^{X}_{c_1,L}(\Lambda)\in \Lambda^{-c_1^2}\Q[\Lambda^{\pm 4}]$ and $l^{X}_{c_1,L}\in \Z_{\ge 0}$, such that 
$$\chi_{c_1}^{X,\omega}(L)\equiv\frac{P^X_{c_1,L}(\Lambda)}{(1-\Lambda^4)^{l^X_{c_1,L}}}.$$
Here $\omega$ is an ample line bundle on $X$ with $\<\omega,K_X\><0$. In case $X$ is the blowup of $\P^2$ in $n$ points we assume furthermore that $\omega=H-a_1E_1-\ldots-a_nE_n$, with $|a_i|<\frac{1}{\sqrt{n}}$ for all $i$.
Note that $P^{X}_{c_1,L}(\Lambda)$,  $l^{X}_{c_1,L}$ are independent of $\omega$ (subject to the conditions above).
In particular  for any other ample line bundle $\omega'$ on $X$ satisfying the conditions for $\omega$ above,  we have 
$\chi_{c_1}^{X,\omega'}(L)-\chi_{c_1}^{X,\omega}(L)\in \Z[\Lambda]$.
\end{Theorem}


We will see that there is an algorithm for determining the generating functions $\chi_{c_1}^{X,\omega}(L)$ of \thmref{rationalal}. 
Let now $H$ be the hyperplane bundle on $\P^2$.
We apply the algorithm above to determine the generating functions of the 
$\chi(M^{\P^2}_H(0,d),\mu(nH))$ and the $\chi(M^{\P^2}_H(H,d),\mu(nH))$ for $n\le 11$. These were determined before (and strange duality proven) for $c_1=0$ and $n=1,2$ in \cite{Abe}, and for all $c_1$ for $n=1,2,3$ in \cite{GY}.
We get the following result.
Put 
{\small \begin{align*}&p_1(t)=p_2(t)=1,\ p_3(t)=1+t^2,p_4(t)=1+6t^2+t^3, 
 p_5(t)=1+21t^2+20t^3+21t^4+t^6, \\
&p_6(t)=1+56t^2+147t^3+378t^4+266t^5+148t^6+27t^7+t^8,\\
&p_7(t)=1+126t^2+690t^3+3435t^4+7182t^5+9900t^6 +7182t^7+3435t^8+690t^9+126t^{10}+t^{12},\\
&p_8(t)=1+252t^2+2475t^3+21165t^4+91608t^5+261768t^6+462384t^7+
549120t^8+417065t^9\\
&\ +210333t^{10}+66168t^{11}
+13222t^{12}+1515t^{13}+75t^{14}+t^{15},\\
&p_9(t)=1 + 462t^2 + 7392t^3 + 100359t^4 + 764484t^5+ 3918420t^6 + 13349556t^7 + 31750136t^8 \\ 
&\ + 52917800t^9 + 62818236t^{10} 
+ 52917800t^{11}+ 31750136t^{12} + 13349556t^{13} + 3918420t^{14}\\
&\ + 764484t^{15}+ 100359t^{16}+7392t^{17}+ 462t^{18}+t^{20},\\
&p_{10}(t)=1+ 792t^2 + 19305t^3 + 393018t^4 + 4788696t^5 + 39997980t^6 + 231274614t^7  + 961535355t^8 \\
&\ + 2922381518t^9 + 6600312300t^{10} + 11171504661t^{11} + 14267039676t^{12} + 13775826120t^{13}  \\
&\ + 10059442536t^{14} + 5532629189t^{15} +2277448635t^{16}+693594726t^{17} 
+ 154033780t^{18} +24383106t^{19}\\
&\ + 2669778t^{20}+192588t^{21}
+ 8196t^{22}+ 165t^{23}+t^{24}.\\
&p_{11}(t)= 1 + 1287t^2 + 45474t^3 + 1328901t^4 + 24287340t^5 + 309119723t^6
 + 2795330694t^7 \\ 
 &\ + 18571137585t^8 + 9253037887 6t^9 + 351841388847t^{10} + 1033686093846t^{11}+ 2369046974245t^{12}\\
&\   + 4264149851544t^{13} + 6056384937603t^{14} + 6805690336900t^{15}+ 6056384937603t^{16}+ 4264149851544t^{17} \\
&\  + 2369046974245t^{18} + 1033686093846t^{19} + 351841388847t^{20} + 92530378876t^{21} + 18571137585t^{22} \\
&\ + 2795 330694t^{23}+ 309119723t^{24}+ 24287340t^{25}+ 1328901t^{26}+ 45474t^{27}+ 1287t^{28}+t^{30}.
  \end{align*}}
  For $1\le n\le  11$ we put $P_n(\Lambda):=p_n(\Lambda^4)$, $Q_n(\Lambda)=\Lambda^{n^2-1}P_n(\frac{1}{\Lambda})$.
It is easy to see that for $n$ odd, $P_n$ is symmetric, i.e.  $P_n(\Lambda)=Q_n(\Lambda)$.

\begin{Theorem} \label{mainp2} For $1\le n\le 11$ we have
\begin{enumerate}
\item
$$1+\binom{n+2}{2}\Lambda^4+\sum_{d>4}\chi(M^{\P^2}_H(0,d),\mu(nH))\Lambda^d=\frac{P_n(\Lambda)}{(1-\Lambda^4)^{\binom{n+2}{2}}}.$$
\item if $n$ is even, then
$$\sum_{d>0}\chi(M^{\P^2}_H(H,d),\mu(nH))\Lambda^d=\frac{Q_n(\Lambda)}{(1-\Lambda^4)^{\binom{n+2}{2}}}.$$
\end{enumerate}
\end{Theorem}


We see that for $n\le 11$, the generating functions $\chi^{\P^2,H}_H(nh)$, $\chi^{\P^2,H}_0(nH)$
have a number of interesting features, which we conjecture to hold for all $n>0$.

\begin{Conjecture}\label{p2con}
For all $n>0$ there are polynomials $p_n(t)\in \Z[t]$ such the following holds. We put $P_n(\Lambda)=p_n(\Lambda^4)$, $Q_n(\Lambda)=\Lambda^{n^2-1}P_n(\frac{1}{\Lambda})$.
\begin{enumerate}
\item $$1+\binom{n+2}{2}\Lambda^4+\sum_{d>4}\chi(M^{\P^2}_H(0,d),\mu(nH))\Lambda^d=\frac{P_n(\Lambda)}{(1-\Lambda^4)^{\binom{n+2}{2}}}.$$
\item If $n$ is odd, then
$P_n(\Lambda)=Q_n(\Lambda)$, if $n$ is even then
$$\sum_{d>0}\chi(M^{\P^2}_H(H,d),\mu(nH))\Lambda^d=\frac{Q_n(\Lambda)}{(1-\Lambda^4)^{\binom{n+2}{2}}}.$$
\item $p_n(1)=2^{\binom{n-1}{2}}$.
\item For $i$ odd and $i\le n-3$ we have 
$$\Coeff_{x^i}\big[e^{-(n^2-1)x/2}P_n(e^x)]=\Coeff_{x^i}\big[e^{-(n^2-1)x/2}Q_n(e^x)]=0.$$
\item The degree of $p_{n}(t)$ is the largest integer strictly smaller than $n^2/4$.
\end{enumerate}
\end{Conjecture}

On $\P^1\times \P^1$ we get the following results. Let $F$ and $G$ be the classes of the fibres of the projections to the two factors. Let 
{\small \begin{align*}
q^0_1&:=1,\ q^0_{2}:=1+t^2,\ q^0_3=1+10t^2+4t^3+t^4,q^0_4:=1+46t^2 + 104t^3 + 210t^4 + 104t^5 + 46t^6 + t^8,\\ q^0_5&:=1 + 146t^2 + 940t^3 + 5107t^4 + 12372t^5 + 19284t^6+ 16280t^7  + 8547t^8 + 2452t^9 + 386t^{10} + 20t^{11} + t^{12},\\
q^0_6&:=1 + 371t^2 + 5152t^3 + 58556t^4 + 361376t^5 + 1469392t^6 + 3859616t^7 + 6878976t^8 + 8287552t^9 \\&+ 6878976t^{10} + 3859616t^{11} + 1469392t^{12} + 361376t^{13} + 58556t^{14} + 5152t^{15} + 371t^{16} + t^{18},\\
q^0_7&=1+812t^2+ 20840t^3 + 431370t^4+5335368t^5+ 44794932t^6+ 259164216t^7+ 1070840447t^8\\
&+ 3214402272t^9+ 7125238944t^{10}+ 11769293328t^{11}
+ 14581659884t^{12}  + 13577211024t^{13}\\
&+9496341984t^{14} + 4966846032t^{15} + 1928398719t^{16} + 548923040t^{17}+ 112654644t^{18}+ 16232904t^{19}\\
&+ 1584906t^{20}+ 97448t^{21}+ 3564t^{22} + 56t^{23}+t^{24},\\
q^{F+G}_{2}&=t^{\frac{1}{2}}(1+t),q^{F+G}_{4}=t^{\frac{1}{2}}(1 + 10t + 84t^2 + 161t^3 + 161t^4+ 84t^5+ 10t^6 + t^7),\\
q^{F+G}_{6}&=t^{\frac{1}{2}}(1+ 35t+ 1296t^2+ 18670t^3 + 154966t^4 + 770266t^5 + 2504382t^6+ 5405972t^7 + 7921628t^8 \\&+ 7921628t^{9} + 5405972t^{10}+ 2504382t^{11} + 770266t^{12}+ 154966t^{13} + 18670t^{14} + 1296t^{15}+ 35t^{16} + t^{17}),\\
q^F_{2}&:=2t,\ q^F_4:=3t + 43t^2 + 105t^3 + 210t^4 + 105t^5 + 43t^6 + 3t^7,\\ 
q^F_{6}&:=4t + 274t^2 + 5520t^3 + 57022t^4 + 366052t^5 + 1460922t^6 + 3873184t^7 + 6855798t^8 + 8316880t^9  \\&+ 6855798t^{10}+ 3873184t^{11} + 1460922t^{12} + 366052t^{13} + 57022t^{14} + 5520t^{15} + 274t^{16} + 4t^{17}.\\
\end{align*}}
For $n$ odd put $q^{F+G}_n(t)=t^{d^2/2}q^0_n(t^{-1})$. 
Then we get
\begin{Theorem}\label{P11gen}
\begin{enumerate}
\item  $\displaystyle{\sum_{d>4} \chi(M^{\P^1\times \P^1}_{F+G}(0,d),\mu(nF+nG))\Lambda^d=\frac{q^0_n(\Lambda^4)}{(1-\Lambda^4)^{(n+1)^2}}-1-(n^2+2n+1)\Lambda^4}$ for $1\le n\le 7$.
\item  $\displaystyle{\sum_{d>0} \chi(M^{\P^1\times \P^1}_{F+G}(F+G,d),\mu(nF+nG))\Lambda^d=\frac{q^{F+G}_n(\Lambda^4)}{(1-\Lambda^4)^{(n+1)^2}}-\Lambda^2}$ for $1\le n\le 7$.
\item $\displaystyle{\sum_{d>0} \chi(M^{\P^1\times \P^1}_{F+G}(F,d),\mu(nF+nG))\Lambda^d=\frac{q^F_n(\Lambda^4)}{(1-\Lambda^4)^{(n+1)^2}}}$ for $n=2,4,6$.
\end{enumerate}
\end{Theorem}
\begin{Remark}
\begin{enumerate}
\item For $d$ even and $c_1=0$, $F$, $F+G$ we have  $q^{c_1}_n(t)=t^{n^2/2}q^{c_1}_n(t^{-1})$.
\item For all $1\le n\le 7$ we have $q^0_n(1)=q^{F+G}_n(1)=2^{(n-1)^2}$, and if $d$ is also even $q^F_n(1)=2^{(n-1)^2}$.
\item For all $1\le n \le 7$ and all $i$ odd with $i\le n-2$ we have $\Coeff_{x^i}\big[e^{-n^2x/4}q^0_{d}(e^x)]=\Coeff_{x^i}\big[e^{-n^2x/4}q^{F+G}_{n}(e^x)]=0$.
\end{enumerate}
\end{Remark}

The results on $\P^2$, $\P^1\times \P^1$ as well as the computations for other rational surfaces lead to a general conjecture.
For a line bundle $L$ on a rational surface $X$ we denote $\chi(L)=L(L-K_X)/2+1$ the holomorphic Euler characteristic and 
$g(L)=L(L+K_X)/2+1$ the genus of a smooth curve in the linear system $|L|$.

\begin{Conjecture}\label{ratconj}
Let $X$ be a rational surface and let $\omega$ be ample on $X$ with $\<\omega, K_X\><0$.   Let $L$ be a sufficiently ample line bundle on $X$. 
Then we have the following.
\begin{enumerate}
\item There is a polynomial $P^X_{c_1,L}(\Lambda)\in \Lambda^{-c_1^2}\Z_{\ge 0}[\Lambda^{\pm 4}]$, such  that
$$\sum_{d\ge 0} \chi(M^{X}_\omega(c_1,d),\mu(L))\Lambda^d\equiv \frac{P^X_{c_1,L}(\Lambda)}{(1-\Lambda^4)^{\chi(L)}}.$$
\item We have $P^X_{c_1,L}(1)=2^{g(L)}$.
\item We have the "duality" 
$$P^X_{c_1,L}(\Lambda)=\Lambda^{L^2+8-K_X^2}P^X_{L+K_X-c_1,L}(\frac{1}{\Lambda}).$$
\item If $i$ is odd, and  $L$ is sufficiently ample with respect to $i$, then 
$$\Coeff_{x^i}\big[e^{-\frac{1}{2}(L^2+8-K_X^2)x}P^X_{c_1,L}(e^x)\big]=0.$$
In the case of $(\P^2,dH)$ and $(\P^1\times \P^1,dF+dG)$ sufficiently ample with respect to $i$  means that $L+K_X$ is $i$-very ample.
\end{enumerate}
\end{Conjecture}

\begin{Remark} \label{unique}
The polynomial $P^X_{c_1,L}(\Lambda)$ is not well defined. We can write $P^X_{c_1,L}(\Lambda)=\Lambda^{-c_1^2}p^X_{c_1,L}(\Lambda^4)$, and the polynomial
$p^X_{c_1,L}(t)$ is well defined only up to adding a Laurent polynomial in $t$ divisible by $(1-t)^{\chi(L)}$. On the other hand, if $L$ is sufficiently ample with respect to $c_1,X$, we conjecture that we can choose $p^X_{c_1,L}(t)$ with
$\deg(p^X_{c_1,L}(t))<\chi(L)$ (i.e. the difference in degree of the highest order and lowest order term in $p^X_{c_1,L}(t)$ is smaller than $\chi(L)$). Assuming this, $p^X_{c_1,L}(t)$ and thus $P^X_{c_1,L}(\Lambda)$ are uniquely determined.
\end{Remark}

\begin{Remark}\label{chipol}
Part (1) of \conref{ratconj} requires a condition of sufficient ampleness (see \thmref{rpoly}). On the other hand it appears that a modified version of the conjecture holds in  larger generality, i.e. 
$\chi^{\omega}_{X,c_1}(L)\equiv \frac{P^X_{c_1,L}(\Lambda)}{(1-\Lambda^4)^{\chi(L)}}.$
with $P^X_{c_1,L}(\Lambda)\in \Lambda^{-c_1^2}\Q[\Lambda^{\pm 4}]$, and
\begin{enumerate}
\item  $P^X_{c_1,L}(1)=2^{g(L)}$,
\item 
$\chi^{\omega}_{X,c_1}(L)\equiv (-1)^{\chi(L)}\Lambda^{-(L-K_X)^2+4} \cdot \chi^\omega_{X,L+K_X-c_1}(L)|_{\Lambda=\frac{1}{\Lambda}}.$
\end{enumerate}
\end{Remark}

{\bf Approach.}
This paper is built on  \cite{GY}, and both papers are built on \cite{GNY}. In \cite{GNY} the wallcrossing terms for the $K$-theoretic Donaldson invariants are 
determined in terms of modular forms, based on the solution of the Nekrasov conjecture for the $K$-theoretic partition function (see \cite{Nek}, \cite{NO}, \cite{NY1},\cite{NY2},\cite{NY3}).
, and both \cite{GY} and this paper sum up the wallcrossing terms to get closed formulas for the generating functions. 
The main new inputs are the systematic use  of the generating function  the "$K$-theoretic Donaldson invariants 
with point class" $\chi^{X,\omega}_{c_1}(L,P^r)$, and the blowup formulas.
We introduce in an ad hoc way  $\chi^{X,\omega}_{c_1}(L,P^r):=\frac{1}{\Lambda^r}\chi^{\widehat X,\omega}_{c_1+E}(L-E)$, where $\widetilde X$ is the blowup of $X$ in $r$ general points and $E$ is the sum of the exceptional divisors (but note that these are invariants on $X$, depending on an ample class $\omega$ on $X$). These invariants satisfy a wallcrossing formula which is very similar to that of the standard $K$-theoretic Donaldson invariants $\chi^{X,\omega}_{c_1}(L)$.
We prove blowup formulas that compute all the generating formulas of $K$-theoretic Donaldson invariants on any blowup of $X$ in terms of the $\chi^{X,\omega}_{c_1}(L,P^r)$.
On the other hand we also prove blowdown formulas, which compute all the generating functions of the $K$-theoretic Donaldson invariants with point class $\chi^{X,\omega}_{c_1}(M,P^r)$ in terms of a {\it very small part} of those on the blowup $\widehat X$. 
Then, generalizing the methods of \cite{GY}, we compute this small part in the case $\widehat X$ is the blowup of $\P^2$ in a point. Thus, using the blowdown formulas, we determine the generating functions of the $K$ theoretic Donaldson invariants with point class of $\P^2$, and thus, by using the blowup formula again, of all blowups of $\P^2$. Finally, as the blowup of $\P^1\times \P^1$ in a point is equal to the blowup of 
$\P^2$ in two points, we apply the blowdown formulas again to determine generating functions for $\P^1\times\P^1$.
These methods give an algorithm, which in principle computes all the generating functions mentioned above. The algorithm proves the rationality of the 
generating functions, and is carried for many $X$ and $L$ to obtain the  explicit generating functions $\chi^{X,\omega}_{c_1}(L)$.

\section{Background material}\label{sec:background}

In this whole paper $X$ will be a simply connected nonsingular projective rational surface over $\C$. 
Usually $X$ will be $\P^2$, $\P^1\times \P^1$ or the blowup of $\P^2$ in finitely many points.

We will fix some notation that we want to use during this whole paper.

\begin{Notation} \label{R} 
\begin{enumerate}
\item For a class $\alpha,\beta\in H^2(X,\Q)$, we denote $\<\alpha,\beta\>$ their intersection product.
For $\beta\in H^2(X)$ we also write $\beta^2$ instead of $\<\beta,\beta\>$.
\item For a line bundle $L$ on $X$ we denote its first Chern class by the same letter.
\item If $\widehat X$ is the blowup of $X$ in a point or in a finite set of points, and $L\in Pic(X)$, we denote its pullback to $\widehat X$ by the same letter.
The same holds for classes $\alpha\in H^2(X,\R)$.
\item We denote $
 \RR:=\Q[[q^2\Lambda^2,q^4]]$.
 \item Let $\Q[t_1,\ldots,t_k]_n$ be the set of polynomials in $t_1,\ldots,t_k$ of degree $n$ and $\Q[t_1,\ldots,t_k]_{\le n}$ the polynomials in $t_1,\ldots,t_n$
 of degree at most $n$.
\item Let $\omega$ be an ample divisor on $X$. 
For $r\ge 0$, $c_1\in Pic(X)$, $c_2\in H^4(X,\Z)$ let
$M_\omega^X(r,c_1,c_2)$  the moduli space of $\omega$-semistable rank $r$ sheaves on $X$ with $c_1(E)=c_1$, $c_2(E)=c_2$. 
Let $M^X_\omega(r, c_1,c_2)_s$ be the open 
subset of stable sheaves. 
We will   write $M^X_\omega(c_1,d)$ with $d:=4c_2-c_1^2$ instead of  $M^X_\omega(2,c_1,c_2)$.
 \end{enumerate}
 \end{Notation}

\subsection{Determinant line bundles}\label{sec:detbun}
We briefly review the determinant line bundles on the moduli space
\cite{DN},\cite{LP1}, \cite{LP2}, for more details we refer to \cite[Chap.~8]{HL}. We mostly follow \cite[Sec.~1.1,1.2]{GNY}.

For a Noetherian scheme $Y$ we denote by $K(Y)$ and $K^0(Y)$ the Grothendieck groups of coherent sheaves and locally free sheaves on $Y$ respectively.
If $Y$ is nonsingular and quasiprojective, then $K(Y)=K^0(Y)$.
If we want to distinguish a sheaf $\F$ and its class in $K(Y)$, we
denote the latter by $[\F]$. The product $[\F].[\G]:=\sum_i (-1)^i[\Tor_i(\F,\G)]$ makes
$K^0(Y)$ into a commutative ring and $K(Y)$ into a $K^0(Y)$ module.
For a proper morphism $f\colon Y_1\to Y_2$ we have the pushforward
homomorphism
\(f_!\colon K(Y_1)\to K(Y_2);  [\F] \mapsto\sum_i (-1)^i [R^if_*\F].\)
For any morphism $f\colon Y_1\to Y_2$ we have the pullback
homomorphism \(   f^*\colon K^0(Y_2)\to K^0(Y_1) \)
given by 
\(  [\F] \mapsto[f^*\F] \) for a locally free sheaf $\F$ on $Y_2$.
Let $\E$ be a flat family of coherent sheaves of class $c$ on $X$ parametrized by a scheme $S$, then $\E\in K^0(X\times S)$.
Let $p:X\times S\to S$, $q:X\times S\to X$ be the projections.
Define $\lambda_\E:K(X)\to \Pic(S)$ as the composition of the following homomorphisms:
\begin{equation}\label{dlb}
\xymatrix@C=0.3cm{
  K(X)=K^0(X) \ar[rr]^{~~q^{*}} && K^0(X\times S) \ar[rr]^{.[\E]} && K^0(X\times
S) \ar[rr]^{~~~p_{!}} && K^0(S)\ar[rr]^{det^{-1}} &&
\Pic(S),}\end{equation}
 By Proposition 2.1.10 in \cite{HL} $p_{!}([\F])\in K^0(S)$ for $\F$ $S$-flat .
We have the following facts.
\begin{enumerate}
\item $\lambda_\E$ is a homomorphism, i.e. $\lambda_\E(v_1+v_2)=\lambda_\E(v_1)\otimes \lambda_{\E}(v_2)$.
\item If $\mu\in \Pic(S)$ is a line bundle, then $\lambda_{\E\otimes p^*\mu}(v)=
\lambda_{\E}(v)\otimes \mu^{\chi(c\otimes v)}$.
\item $\lambda_\E$ is compatible with base change: if
$\phi:S'\to S$ is a morphism, then $\lambda_{\phi^*\E}(v)=\phi^*\lambda_{\E}(v)$.
\end{enumerate}
Define $K_c:=c^\perp=\big\{v\in K(X)\bigm|
\chi(v\otimes c)=0\big\}$,~
and $K_{c,\omega}:=c^\perp\cap\{1,h,h^2\}^{\perp\perp}$, where $h=[\oo_\omega]$.  Then we have a well-defined morphism $\lambda\colon K_c\to \Pic(M_\omega^X(c)^s)$,  and $\lambda\colon K_{c,\omega}\to \Pic(M_\omega^X(c))$ satisfying the following properties:
 \begin{enumerate}
 \item The $\lambda$ commute with the  inclusions $K_{c,\omega}\subset K_c$ and $\Pic(M_\omega^X(c))\subset \Pic(M_\omega^X(c)^s)$.
 \item If $\E$ is a flat family of semistable sheaves  on $X$ of class $c$ parametrized by $S$,  then  we have
$\phi_{\E}^*(\lambda(v))=\lambda_{\E}(v)$ for all $v\in K_{c,\omega}$ with $\phi_\E:S\ra M^X_\omega(c)$ the classifying morphism.
 \item If $\E$ is a flat family of stable sheaves, the statement of (2) holds with  $K_{c,\omega}$, $M^X_\omega(c)$ replaced by  $K_{c}$, $M^X_\omega(c)^s$.
 \end{enumerate}
 

Since $X$ is a simply connected surface, both the moduli space $M^X_\omega(c)$ and the determinant line bundle $\lambda(c^*)$ only depend on the images of $c$ and $c^*$ in $K(X)_{num}.$  Here $K(X)_{num}$ is the Grothendieck group modulo numerical equivalence.  We say that $u,v\in K(X)$ are numerically equivalent if $u-v$ is in the radical of the quadratic form $(u,v)\mapsto \chi(X,u\otimes v)\equiv \chi(u\otimes v)$

We call $H$ {\it general} with respect to
$c$ if all the strictly semistable sheaves in $M_H^X(c)$
are strictly semistable with respect to all ample divisors on $X$ in a neighbourhood of $H$ \

Often $\lambda\colon K_{c,\omega}\to \Pic(M_\omega^X(c))$ can be extended.  For instance let $c=(2,c_1,c_2)$, then $\lambda(v(L))$ is well-defined over $M^X_\omega(c)$ if $\<L,\xi\>=0$ for all $\xi$ a class of type $(c_1,d)$ (see \secref{walls}) with $\<\omega,\xi\>=0$.  This can be seen  from the construction of $\lambda(v(L))$ (e.g. see the proof of Theorem 8.1.5 in \cite{HL}).

\subsection{Walls}\label{walls}
Denote by $\cc\subset H^2(X,\R)$ the ample cone of $X$. 
Then $\cc$ has a chamber structure:
For a class $\xi\in H^2(X,\Z)\setminus \{0\}$ let $W^\xi:=\big\{ x\in \cc\bigm| \< x,\xi\>=0\big\}$.
Assume $W^\xi\ne \emptyset $. Let $c_1\in \Pic(X)$, $d\in \Z$ congruent to $-c_1^2$ modulo 4.
Then we call $\xi$ a {\it class of type} $(c_1,d)$ and call 
$W^\xi$ a {\it wall of type} $(c_1,d)$ if the following conditions hold
\begin{enumerate}
\item 
$\xi+c_1$ is divisible by $2$ in $H^2(X,\Z)$,
\item $d+\xi^2\ge 0$.
\end{enumerate}
We call $\xi$ a {\it class of type} $(c_1)$, if $\xi+c_1$ is divisible by $2$ in $H^2(X,\Z)$.
We say that $\omega\in \cc$ lies on the wall $W^\xi$ if $\omega\in W^\xi$.
The {\it chambers of type} $(c_1,d)$ are the connected components of the complement
of the walls of type $(c_1,d)$ in $\cc$.
Then $M_\omega^X(c_1,d)$ depends only on the chamber of type $(c_1,d)$ of $\omega$.
Let $c\in K(X)$ be the class of $\F\in M_\omega^X(c_1,d)$. It is easy to see 
that $\omega$ is general with respect to $c$ if and only if $\omega$ does not lie on a wall of 
type $(c_1,d)$.

\subsection{$K$-theoretic Donaldson invariants}\label{backDong}
We write $M^X_\omega(c_1,d)$ for $M^X_\omega(2,c_1,c_2)$ with $d=4c_2-c_1^2$.
Let $v\in K_c$, where $c$ is the class of a coherent rank $2$ sheaf 
with Chern classes $c_1,c_2$. 
Let $L$ be a line bundle on $X$ and assume that $\<L,c_1\>$ is even.
Then for $c$ the class of a rank $2$ coherent sheaf with Chern classes 
$c_1,c_2$, we put 
\begin{equation}\label{eq:uL} v(L):=(1-L^{-1})+\<\frac{L}{2},L+K_X+c_1\>[\oo_x]\in K_c.\end{equation}
Note that $v(L)$ is independent of $c_2$. 
Assume that $\omega$ is general with respect to $(2,c_1,c_2)$. Then we denote
$\mu(L):=
\lambda(v(L))\in \Pic(M^X_\omega(c_1,d))$.
The {\it $K$-theoretic Donaldson invariant\/} of $X$, with respect to $L,c_1,d,\omega$ is 
$\chi(M^X_\omega(c_1,d),\mathcal O(\mu(L)))$.

We recall the following blowup relation for the  $K$-theoretic Donaldson invariants from  \cite[Sec.1.4]{GNY}.  
Let $(X,\omega)$ be a polarized rational surface. Let $\widehat X$ be the
blowup of $X$ in a point and $E$ the exceptional divisor.
In the
following we always denote a class in $H^*(X,\Z)$ and its pullback by
the same letter. 
Let $Q$ be an open subset of a suitable quot-scheme
such that $M^X_\omega(c_1,d)=Q/GL(N)$. Assume that $Q$ is smooth \textup(e.g.\ $\langle -K_X,\omega\rangle>0$\textup).
We choose $\epsilon>0$ sufficiently small so that  $\omega-\epsilon E$ is ample on $\widehat X$ and  there is no class $\xi$ of type $(c_1,d)$ or of type $(c_1+E,d+1)$ on $\widehat X$ with $\<\xi, \omega\><0<\<\xi, (\omega-\epsilon E)\>.$
In case $c_1=0$ assume $d>4$.

\begin{Lemma}  \label{blowsimple}
We have 
\begin{align*}
\chi({M}^{\widehat X}_{\omega-\epsilon E}(c_1,d),\mu(L))&
=\chi(M^X_\omega(c_1,d),\mu(L)), \\
 \chi({M}^{\widehat X}_{\omega-\epsilon E}(c_1+E,d+1),\mu(L))&
=\chi(M^X_\omega(c_1,d),\mu(L))
\end{align*}
for any line bundle  $L$ on $X$ such that $\<L, c_1\>$ is even and 
$\<L,\xi\>=0$ for $\xi$ any class of 
type $(c_1,d)$ on $\widehat X$ with $\<\omega,\xi\>=0$. 
\end{Lemma}


Following \cite{GY}, we introduce the generating function of the $K$-theoretic Donaldson invariants.
\begin{Definition} \label{KdonGen} Let $c_1\in H^2(X,\Z)$. Let $\omega$ be ample on $X$ not on a wall of type $(c_1)$.
\begin{enumerate}
\item 
If $c_1\not \in 2 H^2(X,\Z)$, let
\begin{equation}\label{eq:Kdon}
\begin{split}
\chi_{c_1}^{X,\omega}(L)&:=\sum_{d>0}
\chi(M^X_\omega(c_1,d),\mathcal O(\mu(L)))\Lambda^d, 
\end{split}
\end{equation}
\item In case $c_1=0$ let 
$\widehat X$ be the blowup of $X$ in a point. Let $E$ be the exceptional divisor. Let $\epsilon>0$ be sufficiently small so that there is no class
$\xi$ of class $(E,d+1)$ on $\widehat X$ with  $\<\xi, \omega\> <0 <\<\xi, (\omega-\epsilon E)\>$. 
We put 
\begin{equation}\label{eq:Kdon0}
\begin{split}
\chi_{0}^{X,\omega}(L)&:=\sum_{d>4}
\chi(M^X_\omega(0,d),\mathcal O(\mu(L)))\Lambda^d+\Big(\chi(M^{\widehat X}_{\omega-\epsilon E}(E,5) ,\mu(L))+LK_X-\frac{K_X^2+L^2}{2}-1\Big)\Lambda^4.
\end{split}
\end{equation}
\end{enumerate}
\end{Definition}



\begin{Remark}\label{rem:canonical}(\cite[Rem.~1.9]{GNY})
If $H$ is a general polarization, then
$\mu(2K_X)$ is a line bundle on $M^X_H(c)$
which coincides with the dualizing sheaf
on the locus of stable sheaves $M_H^X(c)^s$.
If $\dim (M_H^X(c) \setminus M_H^X(c)^s) \leq \dim M_H^X(c)-2$,
then $\omega_{M_H^X(c)}=\mu(2K_X)$.
\end{Remark}

Under rather general assumptions the higher cohomology of  $\mu(L)$ vanishes. 
The following follows from \cite[Prop.2.9]{GY} and its proof, which is based on  \cite[Sec.1.4]{GNY}.
\begin{Proposition} \label{highvan}
Fix $c_1,d$. Let $\omega$ be an ample line bundle on $X$ which is general with respect to $c_1,d$, and satisfies $\<-K_X,\omega\>>0$.
Let $L$ be a nef line bundle on $X$ such that $L-2K_X$ is ample.
If $c_1$ is not divisible by $2$ in $H^2(X,\Z)$ or $d>8$, we have 
$H^i(M_\omega^X(c_1,d),\mu(L))=0$ for all $i>0$, in particular 
$$\dim H^0(M_\omega^X(c_1,d),\mu(L))=\chi(M_\omega^X(c_1,d),\mu(L)).$$
\end{Proposition}

\section{Strange duality}\label{strd}
\subsection{Review of strange duality}
We briefly review the strange duality conjecture from for surfaces from \cite{LPst}.
The strange duality conjecture was formulated for $X$ a smooth curve in the 1990s (see \cite{Bea} and \cite{Donagi}) and in this case been proved around 2007 (see \cite{Bel1}, \cite{MO1}). For $X$ a surface, there  is a formulation for some special due to Le Potier (see \cite{LPst} or \cite{Da2}). 
Let $c,c^*\in K(X)_{num}$ with $c\in K_{c^*}$.
Let $H$ be an ample line bundle on $X$ which is
both $c$-general and $c^*$-general.
Write $\D_{c,c^*}:=\lambda(c^*)\in \Pic(M^X_H(c))$,
$\D_{c^*,c}:=\lambda(c)\in \Pic(M^X_H(c^*))$.
Assume that  all $H$-semistable sheaves $\F$ on $X$ of class $c$ and
all $H$-semistable sheaves $\G$ on $X$ of class $c^*$ satisfy
\begin{enumerate}
\item $\Tor_i(\F,\G)=0$ for all $i\ge 1$,
\item $H^2(X,\F\otimes \G)=0$.
\end{enumerate}
Both conditions are automatically satisfied if
$c$ is not of dimension $0$ and  $c^*$ is of dimension $1$
(see \cite[p.9]{LPst}). 

Put $\D:=\D_{c,c^*}\boxtimes \D_{c^*,c}\in \Pic(M^X_H(c)\times M^X_H(c^*))$.
In \cite[Prop.~9]{LPst} a canonical section $\sigma_{c,c^*}$ of $\D$ is constructed, whose zero set is supported on
$$\mathscr{D}:=\big\{([\F],[\G])\in   M^X_H(c)\times M^X_H(c^*)\bigm| H^0(X,\F\otimes \G)\ne 0\big\}.$$
The  element $\sigma_{c,c^*}$ of $H^0(M^X_H(c),\D_{c,c^*})\otimes H^0(M^X_H(c^*),\D_{c^*,c})$, gives  a linear map
\begin{equation}
\label{SDmap}
SD_{c,c^*}:H^0(M^X_H(c),\D_{c,c^*})^\vee \to H^0(M^X_H(c^*),\D_{c^*,c}),
\end{equation}
 called the {\it strange duality map}.
Le Potier's strange duality conjecture is then the following.
\begin{Conjecture}\label{sdcon} Under the above assumptions  $SD_{c,c^*}$ is an isomorphism.
\end{Conjecture}

It seems natural to believe that under more general assumptions than \conref{sdcon} we have the  {\it numerical version of strange duality}
$\chi(M^X_H(c),\D_{c,c^*})^\vee = \chi(M^X_H(c^*),\D_{c^*,c})$.

\subsection{Interpretation of the main results and conjectures in view of strange duality}
In this subsection let 
 $c=(2,c_1,c_2)$ and $c^*=(0,L,\chi=\<\frac{L}2,c_1\>)$, so that $\D_{c,c^*}=\mu(L)$. The moduli space $M^X_H(c^*)$ is a moduli space of pure dimension $1$ sheaves.
 It has a natural projection $\pi:=\pi^{L,c_1}:M^X_H(c^*)\to |L|$, whose fibre over a smooth curve $C$ in $|L|$ is the Jacobian of line bundles degree $\< \frac{L}{2},c_1+K_X+L\>$ on $C$.
 In particular $c^*$ is independent of $c_2$.
 In case $c_1=0$ the fibre of  $\pi^{L,0}$ over the class of a nonsingular curve $C$ is the Jacobian $J_{g(C)-1}(C)$ of degree $g(C)-1=\frac{1}{2}\deg(K_C)$ line bundles on $C$.
In this case  we denote by $\Theta:=\lambda([\oo_X])\in \Pic(M^X_H(c^*)$. The divisor of its restriction to a fibre $J_{g(C)-1}(C)$ is the classical theta divisor of 
degree $g(C)-1$ divisors on $C$ with a section.

Let again $c_1$ be general and let $\oo_X(c_1)$ be the line bundle with first Chern class $c_1$; we denote $\Theta_{2,c_1}:=\lambda([\oo_X\oplus \oo_X(c_1)])\in\Pic(M^X_H(c^*))$.
We also denote $\eta:=\lambda(\oo_{x})\in \Pic(M^X_H(c^*))$, for $x$ a general point of $X$. It is standard that 
$\eta=\pi^*(\oo_{|L|}(1))$, with $\oo_{|L|}(1)$ the hyperplane bundle on $|L|$.
Thus  we see that $\D_{c^*,c}=\lambda(c)=\Theta_{2,c_1}\otimes \pi^*(\oo_{|L|}(c_2))$; in particular
in case $c_1=0$ we have $\D_{c^*,c}=\lambda(c)=\Theta^{\otimes 2}\otimes \pi^*(\oo_{|L|}(c_2))$.

 We use Le Potier's strange duality conjecture and the results and conjectures from the introduction to make conjectures about the pushforwards
 $\pi^{L,c_1}_*(\Theta_{2,c_1})$, $\pi^{L,c_1}_!(\Theta_{2,c_1})$.
For a Laurent polynomial $f(t):=\sum_{n} a_n t^n\in \Z[t^{-1},t]$ we put $f(\oo_{|L|}(-1)):=\bigoplus_n \oo_{|L|}(-n)^{\oplus a_n}.$

\begin{Conjecture}\label{splitconj} \begin{enumerate}
\item If $L$ is sufficiently ample on $X$, then, defining $p^{X,c_1}_L$ as in \remref{unique}, then
$\pi^{L,c_1}_{*}(\lambda(\Theta_{2,c_1}))\otimes \oo_{|L|}=p^{X,c_1}_L(\oo_{|L|}(-1))$ and $R^i\pi^{L,c_1}_*(\lambda(\Theta_{2,c_1}))=0$ for $i>0$.
In particular $\pi_{*}(\lambda_{c^*}(\Theta_{2,c_1}))$ splits as a direct sum of line bundles on 
$|L|$. (Note that his implies that $p^{X,c_1}_L$ is a polynomial with nonnegative coefficients, as conjectured in  \conref{ratconj}(1)).
\item In particular in the case $X=\P^2$, and $d>0$, we get, with the polynomials $p_d(t)$ from \conref{p2con}, that
$$\pi^{dH,0}_{*}(\lambda(\Theta^2))=p_d(\oo_{|dH|}(-1)), \ \pi^{2dH,H}_{*}(\lambda(\Theta_{2,H}))=p_{2d}(\oo_{|2dH|}(1))\otimes\oo_{|2dH|}(-d^2).$$
\item Under more general assumptions on  $L$ on $X$, we expect that there is a choice  of $P^{X,c_1}_L(\Lambda)=\Lambda^{-c_1^2}p^{X,c_1}_L(\Lambda^4)$,
such that
$\pi^{L,c_1}_{!}(\lambda(\Theta_{2,c_1}))=p^{X,c_1}_L(\oo_{|L|}(-1))$
\end{enumerate}
\end{Conjecture}
\begin{Remark}
\begin{enumerate}
\item Assuming part (2) of \conref{splitconj}, \thmref{mainp2} determines $\pi^{dH,0}_{*}(\lambda(\Theta^2))$, $\pi^{dH,H}_{*}(\lambda(\Theta_{2,H}))$ as direct sum of line bundles for $d\le 11$.
\item For $X=\P^1\times\P^1$, assuming part (1) of \conref{splitconj},\thmref{P11gen} gives, with the notation from there, for $d\le 7$ that
\begin{align*}
\pi^{d(F+G),0}_{*}(\lambda(\Theta^2))&=q^0_d(\oo_{|d(F+G)|}(-1)),\\ 
\pi^{d(F+G),F}_{*}(\lambda(\Theta_{2,F}))&=q^F_d(\oo_{|d(F+G)|}(-1)),\\
\pi^{d(F+G),F+G}_{*}(\lambda(\Theta_{2,F+G}))&=(t^{1/2}q^{F+G}_d(t))|_{t=(\oo_{|d(F+G)|}(-1))}.
\end{align*}
\item
In \cite{GY} some further generating functions for the $K$-theoretic Donaldson invariants of $X=\P^1\times\P^1$ or $X=\widetilde \P^2$ are computed. From 
the results there we expect
\begin{align*}
\pi^{nF+2G,0}_*(\Theta^{2})&=(\oo_{|nF+2G|}\oplus \oo_{|nF+2G|}(-1))^{\otimes n}_{ev},\\
\pi^{nF+2G,F}_*(\Theta_{2,F})&=(\oo_{|nF+2G|}\oplus \oo_{|nF+2G|}(-1))^{\otimes n}_{odd},
\end{align*}
where $(\bullet)_{ev}$ and $(\bullet)_{odd}$ denotes respectively the part consisting only of even powers $\oo(-2d)$ or odd powers $\oo(-2d-1)$.
In particular this would give
$$\pi^{nF+2G,0}_*(\Theta^{2})\oplus
\pi^{nF+2G,F}_*(\Theta_{2,F})=(\oo_{|nF+2G|}\oplus \oo_{|nF+2G|}(-1))^{\otimes n}.$$
\end{enumerate}
\end{Remark}

\begin{Remark} We briefly motivate the above conjectures.
Assuming strange duality \conref{sdcon}, we have, using also the projection formula,
\begin{align*}
H^0(M_\omega^X(c_1,d),\chi(L))^\vee&=H^0(M^X_\omega(c^*), \lambda(c))=H^0(|L|,  \pi^{L,c_1}_*(\lambda(c))\\
&=H^0(|L|,  \pi^{L,c_1}_*(\lambda(\Theta_{2,c_1}))\otimes \oo_{|L|}(c_2)),
\end{align*}
and similarly, assuming the numerical version of strange duality above,
$$\chi(M_\omega^X(c_1,d),\chi(L))=\chi( \lambda(c))=\chi(M^X_\omega(c^*),\pi_!(\lambda(c))=\chi(\pi^{L,c_1}_!(\lambda(\Theta_{2,c_1}))\otimes \oo_{|L|}(c_2)).$$
We assume $H^i(X,L)=0$ for $i>0$, thus $\dim(|L|)=\chi(L)-1$,
 then for $0\le l\le \dim(|L|)$, and $n\ge 0$, we have 
 $$\sum_{n\ge 0} \chi(|L|,\oo_{|L|}(-l+n))t^n=\frac{t^l}{(1-t)^{\chi(L)}}.$$
Thus, assuming the numerical part of the strange duality conjecture  and part (3) of \conref{splitconj}, we would get
\begin{align*}
\chi^{X,\omega}_{c_1}(L)&\equiv\Lambda^{-c_1^2} \sum_{n\ge 0}\chi\big(|L|,\pi^{L,c_1}_!(\Theta_{2,c_1})\otimes \oo_{|L|}(n)\big)\Lambda^{4n}\\
&\equiv\Lambda^{-c_1^2} \sum_{n\ge 0}\chi\big(|L|, p^{X,c_1}_L(\oo_{|L|}(-1))\otimes \oo_{|L|}(n)\big)\Lambda^{4n}\\
&=\Lambda^{-c_1^2}\frac{p^{X,c_1}_L(\Lambda^4)}{(1-\Lambda^4)^{\chi(L)}}=\frac{P^{X,c_1}_L(\Lambda)}{(1-\Lambda^4)^{\chi(L)}}
\end{align*}
Assuming the strange duality conjecture and part (1) of \conref{splitconj}, we would get the same statement with the left hand side replaced by 
$\sum_{n\ge 0} H^0(M^X_H(2,c_1,n), \mu(L)) t^{n}$. In other words \conref{splitconj} explains the generating functions of \thmref{mainp2}, \thmref{P11gen} and \conref{ratconj}(1).
\end{Remark}

\begin{Remark}
Assuming \conref{splitconj} and the strange duality conjecture, we see that $\rk(\pi_!(\Theta_{2,c_1}))=p^{X,c_1}_L(1)$.
As mentioned above, the fibre over $\pi^{L,c_1}:M^X_H(c^*)\to |L|$ over the point corresponding to  a smooth curve $C$ in $|L|$ is the Jacobian $J_d(C)$ of line bundles degree $d=\< \frac{L}{2},c_1+K_X+L\>$ on $C$, and
we see that $\Theta_{2,c_1}$ is a polarisation of type $(2,\ldots,2)$. Thus by the Riemann-Roch theorem we have $\chi(J_d(C),\Theta_{2,c_1}|_{J_d(C)})=2^{g(C)}$.
Thus \conref{splitconj} implies that $\pi_!(\Theta_{2,c_1})$ has rank $2^{g(C)}$, therefore, assuming the strange duality conjecture, it implies $p^{X,c_1}_L(1)=2^{g(C)}$, as predicted in 
\conref{ratconj} and seen e.g. in \thmref{mainp2} and \thmref{P11gen} for many $L$ in the case of $\P^2$ and $\P^1\times\P^1$.
\end{Remark}

\begin{Remark} Let $L$ again be sufficiently ample on $X$. 
Assuming the strange duality conjecture \conref{sdcon} and part (1) of \conref{splitconj} we get that part (3) of \conref{ratconj} 
gives the conjectural duality 
$$\pi^{L,c_1}_*(\Theta_{2,c_1})=(\pi^{L,L+K_X-c_1}_*(\Theta_{2,L+K_X-c_1}))^\vee\otimes \oo_{|L|}(-\<L,L+K_X)\>/2-\<c_1,c_1-K_X)\>/2+\<L,c_1\>/2-2)).$$
In particular in case $c_1=0$, 
$$\pi^{L,0}_*(\Theta^{\otimes 2})=(\pi^{L,L+K_X}_*(\Theta_{2,L+K_X}))^\vee\otimes \oo_{|L|}(-\<L,L+K_X)\>/2-2).$$
In the case of $X=\P^2$ we should have for $d>0$ that
\begin{align*}
\pi^{2dH,0}_*(\Theta^{\otimes 2})&=(\pi^{2dH,H}_*(\Theta_{2,H}))^\vee\otimes \oo_{|2dH|}(-d^2),\\
\pi^{(2d+1)H,0}_*(\Theta^{\otimes 2})&=(\pi^{(2d+1)H,0}_*(\Theta^{\otimes 2}))^\vee\otimes \oo_{|(2d+1)H|}(-d(d+1)).
\end{align*}
Similarly we conjecture for $X=\P^1\times\P^1$ e.g. that for $d>0$
$$\pi^{2d(F+G),0}_*(\Theta^{\otimes 2})=(\pi^{2d(F+G),0}_*(\Theta^{\otimes 2}))^\vee\otimes \oo_{|2d(F+G)|}(-2d^2).$$
\end{Remark}

\section{Wallcrossing formula}
\subsection{Theta functions and modular forms}\label{thetamod}
We start by reviewing results and notations from \cite{GNY}, \cite[Sec.~3.1]{GY}.
 For $\tau\in \H=\big\{\tau\in \C\bigm| \Im(\tau)>0\big\}$ put $q=e^{\pi i\tau/4}$ and for  $h\in \C$ put $y=e^{h/2}$. Note that the notation is not standard.
Recall the $4$ Jacobi theta functions:
\begin{equation}
\begin{split}\label{theta}
\theta_1(h)&:=\sum_{n\in \Z} i^{2n-1} q^{(2n+1)^2} y^{2n+1}=-iq(y-y^{-1})\prod_{n>0}(1-q^{8n})(1-q^{8n}y^2)(1-q^{8n}y^{-2}),\\
\theta_2(h):&=\sum_{n\in \Z} q^{(2n+1)^2} y^{2n+1}=-q(y+y^{-1})\prod_{n>0}(1-q^{8n})(1+q^{8n}y^2)(1+q^{8n}y^{-2}),\\
\theta_3(h)&:=\sum_{n\in \Z} q^{(2n)^2} y^{2n},\qquad
\theta_4(h):=\sum_{n\in \Z} i^{2n}q^{(2n)^2} y^{2n}.
\end{split}
\end{equation}
We usually do not write the argument $\tau$. The conventions are essentially the same as in 
\cite{WW} and in \cite{Ak}, where the $\theta_i$ for $i\le 3$ are denoted $\vartheta_i$ and $\theta_4$ is  denoted $\vartheta_0$.
Denote 
\begin{equation}\label{thetatilde}
\begin{split}\theta_i&:=\theta_i(0), \quad 
\widetilde\theta_i(h):=\frac{\theta_i(h)}{\theta_i}, \quad i=2,3,4;\qquad \widetilde\theta_1(h):=\frac{\theta_1(h)}{\theta_4},\\
u&:=-\frac{\th_2^2}{\th_3^2}-\frac{\th_3^2}{\th_2^2}=-\frac{1}{4}q^{-2} - 5q^2 + \frac{31}{2}q^6 - 54q^{10}+O(q^{14}),
\end{split}
\end{equation}
  and two Jacobi functions, i.e. Jacobi forms of weight and index $0$,
  $\Lambda:=\frac{\theta_1(h)}{\theta_4(h)}$, $M:=2\frac{\widetilde \theta_2(h)\widetilde \theta_3(h)}{\widetilde \theta_4(h)^2}$, which satisfy the relation 
  \begin{equation}\label{MuL}
  M=2\sqrt{1+u\Lambda^2+\Lambda^4},\end{equation}
   and the formulas
 \begin{equation}\label{dLdh}\frac{\partial\Lambda}{\partial h}=\frac{\theta_2\theta_3}{4i}M,
\quad 
 h=\frac{2i}{\theta_2\theta_3}\int_{0}^\Lambda\frac{dx}{\sqrt{1+ux^2+x^4}}.
  \end{equation}

 In \cite[Sec.~3.1]{GY} it is shown that  $h\in iq^{-1}\Lambda\RR.$
    A function $F(\tau,h)$ can via formula \eqref{dLdh} also be viewed as a function 
 of $\tau$ and $\Lambda$.
  In this case, viewing $\tau$ and $\Lambda$ as the independent variables we define 
 $F' :=\frac{4}{\pi i} \frac{\partial F}{\partial \tau}=q\frac{\partial F}{\partial q},
 \quad F^*:=\Lambda\frac{\partial F}{\partial \Lambda},$ and get 
\begin{equation}
\label{hstar}
h^*=\frac{4i\Lambda}{\theta_2\theta_3 M}, \quad
u'=\frac{2\theta_4^8}{\theta_2^2\theta_3^2}.
\end{equation}

For future use we record the following standard formulas for the behaviour of the theta functions under translation.
\begin{align}\label{T4trans}
\theta_4(h+2\pi i)&=\theta_4(h),\quad \theta_4(h+ 2\pi i \tau)=-q^{-4}y^{-2}\theta_4(h),\quad\theta_4(h+ \pi i \tau)= i q^{-1}y^{-1}\theta_1(h)\\
\label{T1trans}
\theta_1(h+2\pi i)&=-\theta_1(h),\quad \theta_1(h+ 2\pi i\tau)=-q^{-4}y^{-2}\theta_1(h),\quad \theta_1(h+ \pi i\tau)= i q^{-1}y^{-1}\theta_4(h),\\
\label{T2trans}\theta_2(h+\pi i \tau)&=q^{-1}y^{-1}\theta_3(h),\quad
\theta_3(h+\pi i \tau)=q^{-1}y^{-1}\theta_2(h).\end{align}
(see e.g. \cite[Table VIII, p.~202]{Ak}).

\begin{Lemma}\label{thetaadd} Let $a$, $b\in \Z$. Then
\begin{enumerate}
\item $\theta_4(h)=(-1)^bq^{4b^2}y^{2b}\theta_4(h+2\pi i b\tau)$, $\quad \theta_4(h+2\pi i a)=\theta_4(z)$,
\item $\theta_1(h)=(-1)^bq^{4b^2}y^{2b}\theta_1(h+2\pi i b\tau)$,$\quad \theta_1(h+2\pi i a)=(-1)^a\theta_1(z)$,
\item $\theta_4(h)=e^{\pi i (b-\frac{1}{2})}q^{(2b+1)^2}y^{2b+1}\theta_1(h+2\pi i (b+\frac{1}{2})\tau)$,
\item $\theta_1(h)=e^{\pi i (b-\frac{1}{2})}q^{(2b+1)^2}y^{2b+1}\theta_4(h+2\pi i (b+\frac{1}{2})\tau)$.
\end{enumerate}
\end{Lemma}
\begin{proof}
All these formulas follow by straightforward induction from \eqref{T4trans} and \eqref{T1trans}.
As an illustration we check   (1) and (3). The formula $\theta_4(h+ 2\pi i \tau)=-q^{-4}y^{-2}\theta_4(h)$ gives by induction
\begin{align*}\theta_4(h+2\pi i b\tau)&=-q^{-4}e^{-(h+2\pi i(b-1)\tau)}\theta_4(h+2\pi i (b-1)\tau)\\
&=
-q^{-8b+4}y^{-2}(-1)^{-(b-1)}q^{-4(b-1)^2}y^{-(2b-2)}\theta_4(h)=(-1)^{-b}q^{-(2b)^2}y^{-2b}\theta_4(h),\end{align*}
and (1) follows.
Similarly
\begin{align*}\theta_4(h+2\pi i (b+1/2)\tau)&=iq^{-1}e^{-h/2-\pi i b\tau}\theta_1(h+2\pi ib\tau)=iq^{-4b-1}y^{-1}(-1)^{-b}q^{-(2b)^2}y^{-2b}\theta_1(h)\\
&=e^{-\pi i (b-\frac{1}{2})}q^{-(2b+1)^2}y^{-(2b+1)}\theta_1(h),\end{align*} and (3) follows.
\end{proof}

\subsection{Wallcrossing formula}\label{wallcro}
Now we review the wallcrossing formula from \cite{GNY}, \cite{GY}, and generalize it slightly.
Let $\sigma(X)$ be the signature of $X$.

\begin{Definition}\label{wallcrossterm}
Let $r\ge 0$, let $\xi\in H^2(X,\Z)$ with $\xi^2< 0$. Let $L$ be a line bundle on $X$. We put
$$\Delta_\xi^X(L,P^r):=2 i^{\<\xi, K_X\>} \Lambda^2 q^{-\xi^2}
y^{\<\xi,(L-K_X)\>}\widetilde\theta_4(h)^{(L-K_X)^2}\theta_4^{\sigma(X)}u'h^*M^r,$$
and put $\Delta_\xi^X(L):=\Delta_\xi^X(L,P^0)$.
 By the results of the previous section it can be developed as a power series  
$$\Delta_\xi^X(L,P^r)=\sum_{d\ge 0} f_d(\tau)\Lambda^d\in \C((q))[[\Lambda]],$$
whose coefficients $f_d(\tau)$ are Laurent series in $q$. 
If $\<\xi,L\>\equiv r \mod 2$,
the {\it  wallcrossing term} is defined as
$$\delta_{\xi}^X(L,P^r):=\sum_{d\ge 0} \delta_{\xi,d}^X(L,P^r)\Lambda^d\in \Q[[\Lambda]], $$
with 
$$\delta_{\xi,d}^X(L,P^r)=\Coeff_{q^0}[f_d(\tau)].$$
Again we write $\delta_{\xi,d}^X(L):=\delta_{\xi,d}^X(L,P^0)$ and $\delta_{\xi}^X(L):=\delta_{\xi}^X(L,P^0)$.

The wallcrossing terms $\delta_{\xi}^X(L):=\delta_{\xi,d}^X(L,P^0)$  were already introduced in \cite{GNY} and used in
\cite{GY}. As we will recall in a moment, they compute the change of the K-theoretic Donaldson invariants $\chi^{X,\omega}_{c_1}(L)$, when $\omega$ crosses a wall. 
Later we will introduce  K-theoretic Donaldson invariants with point class $\chi^{X,\omega}_{c_1}(L,P^r)$, whose  wallcrossing is computed by 
$\delta_{\xi}^X(L,P^r)$. Intuitively we want to  think of $r$ as the power of a $K$-theoretic point class $\cal P$.
\end{Definition}

\begin{Remark}\label{delb}
\begin{enumerate}
\item $\delta_{\xi,d}^X(L,P^r)=0$ unless $d\equiv -\xi^2\mod 4$.
\item 
In the definition of $\delta_\xi^X(L,P^r)$ we can replace
$\Delta_{\xi}^X(L,P^r)$ by
\begin{equation}
\begin{split}\label{Delbar}
&\overline \Delta_{\xi}^X(L,P^r):= \frac{1}{2}(\Delta_{\xi}^X(L,P^r)-\Delta_{-\xi}^X(L,P^r))\\
&\ =M^{r} i^{\<\xi, K_X\>} \Lambda^2 q^{-\xi^2}
\big(y^{\<\xi(L-K_X)\>}-(-1)^{\xi^2}y^{-\<\xi,(L-K_X)\>}\big)\widetilde\theta_4(h)^{(L-K_X)^2}\theta_4^{\sigma(X)}u'h^*.
\end{split}
\end{equation}
\end{enumerate}
\end{Remark}
\begin{proof}
(1) As $h\in \C[[q^{-1}\Lambda,q^4]]$, we also have $h^*, y,\widetilde \theta_4(h),M\in \C[[q^{-1}\Lambda,q^4]]$. Finally $u,u'\in q^{-2}\Q[[q^4]]$.
It follows that $\Delta_{\xi}^X(L,P^r)\in q^{-\xi^2}\C[[q^{-1}\Lambda,q^4]]
$. Writing $\Delta_{\xi}^X(L,P^r)=\sum_{d} f_{d,r}(\tau)\Lambda^d$, we see that $\Coeff_{q^0}[f_{d,r}(\tau)]=0$ unless $d\equiv -\xi^2\mod 4$.
(2) Note that $\widetilde \theta_4(h)$ is even in $\Lambda$ and $h^*$ is odd in 
$\Lambda$, thus $\overline \Delta_{\xi}^X(L,P^r)= \sum_{d\equiv -\xi^2 (2) } f_{d,r}(\tau)\Lambda^d$, and the claim follows by (1). 
\end{proof}

The main result of \cite{GNY} is the following (see also \cite{GY}).

\begin{Theorem}\label{wallcr}
Let $H_1$, $H_2$ be ample divisors on $X$, assume that $\<H_1,K_X\><0$, $\<H_2,K_X\><0$, and that $H_1$, $H_2$ do not lie on a wall 
of type $(c_1,d)$. Then 
\begin{align*}
\chi(M^X_{H_1}(c_1,d),\mu(L))-\chi(M^X_{H_2}(c_1,d),\mu(L))&=\sum_{\xi}\delta^X_{\xi,d}(L),
\end{align*}
where $\xi$ runs through all classes of type $(c_1,d)$
with $\<\xi, H_1\>>0 >\<\xi, H_2\>$. 
\item
\end{Theorem}
Note that the condition  $\<H_1,K_X\><0$, $\<H_2,K_X\><0$ implies that all the classes of type  $(c_1,d)$
with $\<\xi, H_1\>>0 >\<\xi ,H_2\>$ are good in the sense of \cite{GNY}, so the wallcrossing formula there applies.
Let $c_1\in H^2(X,\Z)$. Let $H_1,H_2$ be ample on $X$, assume they do not lie on a wall of type $(c_1)$. Then it follows that
$$\chi^{X,H_1}_{c_1}(L)-\chi^{X,H_2}_{c_1}(L)=\sum_{\xi}\delta^X_\xi(L),$$
where $\xi$ runs through all classes in $c_1+2H^2(X,\Z)$ with $\<\xi ,H_1\> >0>\<\xi, H_2\>$.

\subsection{Polynomiality and vanishing of the wallcrossing}
By definition the wallcrossing terms 
$\delta_\xi^X(L,P^r)$ are power series in $\Lambda$.  We now show that 
they are always polynomials, modifying the proof of \cite[Thm.~3.19]{GY}. 
We have seen above that $h\in iq^{-1}\Lambda\Q[[q^{-2}\Lambda^2,q^4]]$, and thus $y=e^{h/2}\in \Q[[iq^{-1}\Lambda,q^4]]$.
\begin{Lemma} \label{qpow}
(\cite[Lem.~3.18]{GY})
\begin{enumerate}
\item $\sinh(h/2)=y-y^{-1}\in iq^{-1}\Lambda\RR$, $\frac{1}{\sinh(h/2)}\in iq\Lambda^{-1}\RR$.
\item For all integers $n$ we have 
\begin{align*}
\sinh((2n+1)h/2)&\in i\Q[q^{-1}\Lambda]_{|2n+1|}\RR,\quad
\cosh(nh)\in  \Q[q^{-2}\Lambda^2]_{|n|} \RR,\\
\sinh(nh)h^*&\in \Q[q^{-2}\Lambda^2]_{|n|}\RR
\quad \cosh((2n+1)h/2)h^*\in i \Q[q^{-1}\Lambda]_{|2n+1|} \RR.
\end{align*}
\item $\widetilde \theta_4(h)\in \RR
$, with $\widetilde \theta_4(h)=1+q^2\Lambda^2+O(q^4)$.
\end{enumerate}
\end{Lemma}

\begin{Lemma}\label{vanwall}
Let $r\in \Z_{\ge 0}$, let $\xi\in H^2(X,\Z)$, and $L$ a line bundle on $X$  with $\xi^2<0$ and $\<\xi,L\>\equiv r\mod 2$.
\begin{enumerate}
\item 
$\delta_{\xi,d}^X(L,P^r)=0$ unless $-\xi^2\le d\le \xi^2+2|\<\xi,L-K_X\>|+2r+4$.
In particular $\delta_{\xi}^X(L,P^r)\in \Q[\Lambda]$.
\item $\delta_\xi^X(L,P^r)=0$ unless $-\xi^2\le |\<\xi,L-K_X\>|+r+2$. (Recall that 
by definition $\xi^2<0$).
\end{enumerate}
\end{Lemma}

\begin{proof}
Assume first that $r=2l$ is even. 
Let $N:=\<\xi,L-K_X\>.$ Then it is shown in the proof of \cite[Thm.~3.19]{GY} that 
$\overline \Delta_\xi^X(L)=q^{-\xi^2}\Q[q^{-1}\Lambda]_{\le |N|+2}\RR.$
On the other hand we note that $M^2=4(1+u\Lambda^2+\Lambda^4)\in \Q[q^{-2}\Lambda^2]_{\le 1}\RR.$
Putting this together we get 
$$\Delta_\xi^X(L,P^{r})\in q^{-\xi^2}\Q[q^{-1}\Lambda]_{\le |N|+r+2}\RR.$$

Now assume that $r=2l+1$ is odd.  If $N$ is even, then by the condition that $\<L,\xi\>$ is odd, we get $(-1)^{\xi^2}=-(-1)^{N}$, and therefore
$$\overline \Delta_\xi^X(L,P^r)=q^{-\xi^2}M^ri^{\<\xi,K_X\>}\Lambda^2\cosh(Nh/2)h^*
\widetilde \theta_4^{(L-K_X)^2}\theta_4^{\sigma(X)} u'.$$
By \eqref{hstar} we get $h^*M=\frac{4i\Lambda}{\theta_2\theta_3}\in i\Lambda q^{-1}\Q[q^4]$.
Thus by Lemma \ref{qpow} we get $\cosh(Nh/2)h^*M\in i\Q[q^{-1}\Lambda]_{\le |N|+1}\RR.$
Using also that $\<\xi, K_X\>\equiv \xi^2\equiv 1\mod 2$, that  $M^2\in \Q[q^{-2}\Lambda^2]_{\le 1}\RR$ and $\Lambda^2u'\in q^{-2}\Lambda^2\RR$, we get
again $\overline \Delta_\xi^X(L,P^r)\in q^{-\xi^2}\Q[q^{-1}\Lambda]_{\le |N|+r+2}\RR.$
Finally, if $N$ is odd, a similar argument shows that 
$$\overline \Delta_\xi^X(L,P^r)=q^{-\xi^2}M^ri^{\<\xi,K_X\>}\Lambda^2\sinh(Nh/2)h^*
\widetilde \theta_4^{(L-K_X)^2}\theta_4^{\sigma(X)} u'\in q^{-\xi^2}\Q[q^{-1}\Lambda]_{\le |N|+r+2}\RR.$$

Therefore we have in all cases that 
$\delta_{\xi,d}^X(L,P^r)= 0$ unless $-\xi^2-\min(d,2|N|+2r+4-d)\le 0$, i.e. 
unless 
$-\xi^2\le d\le \xi^2+2|N|+2r+4$.
In particular $\delta_{\xi}^X(L,P^r)=0$ unless 
$-\xi^2\le\xi^2+2|N|+2r+4$, i.e. unless $-\xi^2\le |N|+r+2$.
\end{proof}

\begin{Remark}\label{c1d}
We note that this implies that for $\xi$ a class of type $(c_1)$, 
$\delta_{\xi,d}^X(L)=0$ for all $L$ unless $\xi$ a class of type $(c_1,d)$.
\end{Remark}

\section{Indefinite theta functions, vanishing, invariants with point class}
We want to study the $K$-theoretic Donaldson invariants 
for polarizations on the boundary of the ample cone. 
Let  $F\in H^2(X,\Z)$ the class of an effective divisor with $F^2=0$ and such that $F$ is nef, i.e. 
$\<F,C\>\ge 0$ for any effective curve in $X$. Then $F$ is a limit of ample classes. 
Let $c_1\in H^2(X,\Z)$ such that $\<c_1,F\>$ is odd. 
Fix $d\in \Z$ with $d\equiv -c_1^2 \mod 4$. Let $\omega$ be ample on $X$. Then for $n>0$ sufficiently large $nF+ \omega$ is ample on $X$ and there is no wall $\xi$ of type $(c_1,d)$ with 
$\<\xi, (nF+ \omega)\>>0> \<\xi, F\>$.  Let $L\in \Pic(X)$ and $r\in\Z_{\ge 0}$ with $\<c_1,L\>$ even. 
Thus we define for $n$ sufficiently large
\begin{align*}M_F^X(c_1,d)&:=M_{nF+\omega}^X(c_1,d), \\ \chi(M_F^X(c_1,d),\mu(L))&:=\chi(M_{nF+\omega}^X(c_1,d),\mu(L)),\\
\chi^{X,F}_{c_1}(L)&:=\sum_{d\ge 0} \chi(M_F^X(c_1,d),\mu(L))\Lambda^d.
\end{align*}

We use the  following standard fact.
\begin{Remark}\label{vanbound}
Let $X$ be a simply connected algebraic surface, and let 
$\pi:X\to \P^1$ be a morphism whose general fibre is isomorphic to 
$\P^1$. Let $F\in H^2(X,\Z)$ be the class of a fibre. Then $F$ is nef. Assume that 
$\<c_1,F\>$ is odd. 
Then $M_F^X(c_1,d)=\emptyset$ for all $d$. 
Thus $\chi(M_F^X(c_1,d),\mu(L))=0$ for all $d\ge 0$.
Thus if $\omega$ ample on $X$ and  does not lie on a wall of type $(c_1)$, then 
$$\chi^{X,\omega}_{c_1}(L)=\sum_{\omega\xi>0>\xi F} \delta_{\xi}^X(L),$$
where the sum is over all classes $\xi$ of type $(c_1)$ with 
$\omega\xi>0>\xi F$.
\end{Remark}

\subsection{Theta functions for indefinite lattices} 


We briefly review a few facts about  theta functions for indefinite lattices of type $(r-1,1)$
introduced in \cite{GZ}. More can be found in \cite{GZ}, \cite{GY}. For us a {\it lattice} is a free $\Z$-module $\Gamma$ together with a quadratic form
$Q:\Gamma\to \frac{1}{2}\Z$, such that the associated bilinear form 
$x\cdot y:=Q(x+y)-Q(x)-Q(y)$ is nondegenerate and $\Z$-valued.
We denote the extension of the quadratic and bilinear form to 
$\Gamma_\R:=\Gamma\otimes_{\Z} \R$ and $\Gamma_\C:=\Gamma\otimes_{\Z} \C$ by the same letters.
 We will consider the case that $\Gamma$ is  $H^2(X,\Z)$ for a rational surface $X$ with the {\it negative} of the intersection form.Thus for $\alpha,\beta\in H^2(X,\Z)$ we have $Q(\alpha)=-\frac{\alpha^2}{2}$, $\alpha\cdot \beta=-\<\alpha,\beta\>$.
Now let $\Gamma$ be a lattice of rank $r$. Denote by $M_\Gamma$ the set of meromorphic 
maps $f:\Gamma_\C\times \H\to \C$.
For $A=\left(\begin{matrix} a&b\\c&d\end{matrix}\right)\in Sl(2,\Z)$,
we define  a map
$|_kA:M_\Gamma\to M_\Gamma$ by
$$f|_{k}A(x,\tau):=(c\tau+d)^{-k}\exp\left(-2\pi i\frac{cQ(x)}{c\tau+d}\right)
f\left(\frac{x}{c\tau+d},\frac{a\tau+b}{c\tau+d}\right).$$
Then $|_kA$ defines an action of $Sl(2,\Z)$ on $M_\Gamma$.
We denote 
\begin{align*}S_\Gamma&:=\big\{ f\in \Gamma\bigm| f \hbox{ primitive},  Q(f)=0,\  f\cdot h <0\big\},\
C_\Gamma:=\big\{ m\in \Gamma_\R\bigm| Q(m)<0, \ m\cdot h<0\big\}.
\end{align*}
For $f\in S_\Gamma$ put 
$D(f):=\big\{(\tau,x)\in \H\times \Gamma_\C\bigm| 0< \Im(f\cdot x)<\Im(\tau)/2\big\},$
and for $h\in C_\Gamma$ put $D(h)=\H\times \Gamma_\C$.
For $t\in \R$ denote $$\mu(t):=\begin{cases} 1& t\ge 0, \\0 & t<0.\end{cases}$$
Let $c,b\in \Gamma$. Let  $f,g\in S_\Gamma\cup C_\Gamma$.
Then for $(\tau,x)\in D(f)\cap D(g)$
define 
$$\Theta^{f,g}_{\Gamma,c,b}(\tau,x):=\sum_{\xi\in \Gamma+c/2}(\mu(\xi\cdot f)-\mu(\xi\cdot  g)) e^{2\pi i \tau Q(\xi)}e^{2\pi i \xi\cdot(x+b/2)}.$$
Let $T:=\left(\begin{matrix} 1 & 1\\0&1\end{matrix}\right)$, 
$S:=\left(\begin{matrix} 0 & -1\\1&0\end{matrix}\right)\in Sl(2,\Z)$.

\begin{Theorem}\label{thetaprop}
\begin{enumerate}
\item For $f,g\in S_\Gamma$ 
the function $\Theta^{f,g}_{X,c,b}(\tau,x)$ has an meromorphic continuation to 
$\H\times \Gamma_\C$.
\item For 
 $|\Im(f\cdot x)/\Im(\tau)|<1/2$ and $|\Im(g\cdot x)/\Im(\tau)|<1/2$ it has a Fourier development
\begin{align*}
&\Theta_{X,c,b}^{f,g}(x,\tau):=\frac{1}{1-e^{2\pi i f\cdot (x+b/2)}}
\sum_{\substack{\xi\cdot f=0\\ f\cdot g\le\xi\cdot g<0}}e^{2\pi i \tau Q(\xi)}e^{2\pi i \xi \cdot (x+b/2)}\\ &
-\frac{1}{1-e^{2\pi i g\cdot (x+b/2)}}\sum_{\substack{\xi\cdot g=0\\ f\cdot g \le \xi \cdot f<0 }}e^{2\pi i \tau Q(\xi)}e^{2\pi i \xi \cdot(x+b/2)}+
\sum_{\xi f>0>\xi g} e^{2\pi i \tau Q(\xi)}\big(e^{2\pi i \xi \cdot(x+b/2)}-
e^{-2\pi i \xi \cdot(x+b/2)}\big),
\end{align*}
where the sums are always over $\xi\in \Gamma+c/2$.
\item
\begin{equation*}
\label{thetajacobi}
\begin{split}
(\Theta_{X,c,b}^{f,g}\theta_{3}^{\sigma(\Gamma)})|_1S&=(-1)^{-b\cdot c/2} \Theta_{X,b,c}^{f,g}\theta_{3}^{\sigma(\Gamma)},\\
(\Theta_{X,c,b}^{f,g}\theta_{3}^{\sigma(\Gamma)})|_1T&=(-1)^{3Q(c)/2-c w/2} \Theta_{X,c,b-c+w}^{f,g}\theta_{4}^{\sigma(\Gamma)},\\
(\Theta_{X,c,b}^{f,g}\theta_{3}^{\sigma(\Gamma)})|_1T^2&=(-1)^{-Q(c)} 
\Theta_{X,c,b}^{f,g}\theta_{3}^{\sigma(\Gamma)},\\
(\Theta_{X,c,b}^{f,g}\theta_{3}^{\sigma(\Gamma)})|_1T^{-1}S&=
(-1)^{-Q(c)/2-c\cdot b/2}\Theta_{X,w-c+b,c}^{f,g}\theta_{2}^{\sigma(\Gamma)},
\end{split}\end{equation*}
where $w$ is a characteristic element of $\Gamma$.
\end{enumerate}
\end{Theorem}
\begin{Remark}
For $f,\ g,\ h\in C_\Gamma\cap S_\Gamma$ we have the cocycle condition.
$\Theta^{f,g}_{\Gamma,c,b}(\tau,x)+\Theta^{g,h}_{\Gamma,c,b}(\tau,x)=\Theta^{f,h}_{\Gamma,c,b}(\tau,x)$, which holds wherever all three terms are defined. 
\end{Remark}

In the following let $X$ be a rational algebraic surface. 
We can express the difference of the $K$-theoretic Donaldson invariants  for two different polarisations in terms of these indefinite theta functions. Here we take $\Gamma$ to be $H^2(X,\Z)$ with the {\it negative} of the intersection form, and we choose $K_X$ as the characteristic element in \thmref{thetaprop}(3).

\begin{Definition}
Let $F,G\in S_\Gamma\cup C_\Gamma$, let $c_1\in H^2(X,\Z)$.
We put 
\begin{align*}
\Psi^{F,G}_{X,c_1}(L;\Lambda,\tau)&:=
\Theta^{F,G}_{X,c_1,K_X}\Big(\frac{(L-K_X)h}{2\pi i},\tau\Big)   \Lambda^2 
\widetilde\theta_4(h)^{(L-K_X)^2}\theta_4^{\sigma(X)}u'h^*.
\end{align*}
\end{Definition}

\begin{Lemma}\label{thetawall}
Let $H_1,H_2$ be ample on $X$ with $\<H_1,K_X\><0$ and $\<H_2,K_X\><0$, and assume that they do not lie on a wall of type $(c_1)$. 
Then
\begin{enumerate}
\item $$\Psi^{H_2,H_1}_{X,c_1}(L;\Lambda,\tau)M^r=\sum_\xi \overline \Delta^X_\xi(L,P^r),$$ were $\xi$ runs through all classes on $X$ of type $(c_1)$ with 
$\<H_2,\xi\> >0> \<H_1, \xi \>$.
\item 
$\chi^{X,H_2}_{c_1}(L)-\chi^{X,H_1}_{c_1}(L)=\Coeff_{q^0} \big[\Psi^{H_2,H_1}_{X,c_1}(L;\Lambda,\tau)\big].
$
\end{enumerate}
\end{Lemma}
\begin{proof} (2) is proven in  \cite[Cor.~4.6]{GY}, where the assumptions is made that $-K_X$ is ample, but the proof only uses $\<H_1,K_X\><0$ and $\<H_2,K_X\><0$, because
this condition is sufficient for \thmref{wallcr}. The argument of \cite[Cor.~4.6]{GY} actually shows (1) in case $r=0$, but as $\overline \Delta^X_\xi(L,P^r)=\overline \Delta^X_\xi(L,P^0)M^r$, the case of 
general $r$ follows immediately.
\end{proof}

Following \cite{GY} we use \lemref{thetawall} to extend the generating function $\chi^{X,\omega}_{c_1}(L)$   to 
$\omega\in S_L\cup C_L$.

\begin{Definition}\label{chiext}
Let $\eta$ be ample on $X$ with $\<\eta, K_X\><0$,  and not on a wall of type $(c_1)$. Let $\omega\in S_X\cup C_X$. We put 
$$\chi^{X,\omega}_{c_1}(L):=\chi^{X,\eta}_{c_1}(L)+\Coeff_{q^0} \big[\Psi^{\omega,\eta}_{X,c_1}(L;\Lambda,\tau)\big].$$
By the cocycle condition the definition of $\chi^{X,\omega}_{c_1}(L)$  is independent of the choice of $\eta$.
Furthermore by  \corref{thetawall} this coincides with the previous definition in case $\omega$ is also ample, $\<\omega,K_X\><0$ and $\omega$ does not lie on a wall of type $(c_1)$. However  if  $\<\omega,K_X\>\ge 0$, it is very well possible that the coefficient of $\Lambda^d$ of $\chi^{X,\omega}_{c_1}(L)$ is different
from $\chi(M^X_\omega(c_1,d),\mu(L))$.
\end{Definition}

\begin{Remark}\label{difftheta}
Now let $H_1,H_2\in S_X\cup C_X$.
By the cocycle condition, we have
\begin{align*}
\chi^{X,H_2}_{c_1}(L)-\chi^{X,H_1}_{c_1}(L)=\Coeff_{q^0} \big[\Psi^{H_2,H_1}_{X,c_1}(L;\Lambda,\tau)\big].
\end{align*}
\end{Remark}


\begin{Proposition} \label{blowgen}
Let $X$ be a rational surface. Let $\omega\in C_X\cup S_X$. Let $c_1\in H^2(X,\Z)$. Let $\widehat X$ be the blowup of $X$ in a general point, and $E$ the exceptional divisor.
Let $L\in \Pic(X)$ with $\<L, c_1\>$ even. Then
\begin{enumerate}
\item $\chi^{\widehat X,\omega}_{c_1}(L)=\chi^{X,\omega}_{c_1}(L)$, 
\item $\chi^{\widehat X,\omega}_{c_1+E}(L)=\Lambda \chi^{X,\omega}_{c_1}(L)$.
\end{enumerate}
\end{Proposition}
\begin{proof}This is \cite[Prop.~4.9]{GY}, where the additional assumption is made that $-K_{\widehat X}$ is ample. The proof works without this assumption with very minor modifications. 
In the original proof the result is first proven for an $H_0\in C_X$ which does not lie on any wall of type $(c_1)$. We now have to assume in addition that $\<H_0,K_X\><0$.
The rest of the proof is unchanged. 
\end{proof}
In  \cite[Thm.~4.21]{GY} is is shown that if $X$ is a rational surface with $-K_X$ ample, then $\chi^{X,F}_{c_1}(L)=\chi^{X,G}_{c_1}(L)$  for all $F,G\in S_X$.
A modification of this proof shows the following.
\begin{Proposition}\label{basic}
Let $X$ be $\P^1\times\P^1$ or a blowup of $\P^2$ in finitely many points. Let $L\in \Pic(X)$, let $c_1\in H^2(X,\Z)$ with $\<c_1,L\>$ even. Let $F,G\in S_X$.
Assume that for all $W\in K_X+2H^2(X,\Z)$ with $\<F,W\>\le 0\le  \<G,W\>$, we have 
$W^2<K_X^2$. Then 
$\chi^{X,F}_{c_1}(L)=\chi^{X,G}_{c_1}(L)$.
\end{Proposition}
\begin{proof}
We know that $\chi^{X,F}_{c_1}(L)-\chi^{X,G}_{c_1}(L)=\Coeff_{q^0}\big[\Psi^{F,G}_{X,c_1}(L,\Lambda,\tau)\big]$, and 
in the proof of \cite[Thm.~4.21]{GY} it is shown that 
$$\Coeff_{q^0}\big[\Psi^{F,G}_{X,c_1}(L,\Lambda,\tau)\big]=-\frac{1}{4}\Coeff_{q^0}\big[\tau^{-2}\Psi^{F,G}_{X,c_1}(L,\Lambda,S\tau)\big]
-\frac{1}{4}i^{c_1^2+3}\Coeff_{q^0}\big[\tau^{-2}\Psi^{F,G}_{X,c_1}(L,i\Lambda,S\tau)\big].$$
Therefore it is enough to show that 
$\Coeff_{q^0}\big[\tau^{-2}\Psi^{F,G}_{X,c_1}(L,\Lambda,S\tau)\big]=0$.
Furthermore in the proof of  \cite[Thm.~4.21]{GY} we have seen that the three functions
$\widetilde u:=-\frac{\theta_3^4+\theta_4^4}{\theta_3^2\theta_4^2}$,
$$
\widetilde h=-\frac{2}{\theta_{4}\theta_{3}}.
\sum_{\substack{n\ge 0\\ n\ge k\ge0}}
\binom{-\frac{1}{2}}{n}\binom{n}{k}\frac{\widetilde u^k\Lambda^{4n-2k+1}}{4n-2k+1},\quad 
\widetilde G(\Lambda,\tau)=
\frac{(-1)^{\<c_1,K_X\>/2-\sigma(X)/4}\Lambda^3}{\theta_3^3\theta_4^3(1+\widetilde{u}\Lambda^2+\Lambda^4)} 
\widetilde\theta_{2}(\widetilde h)^{(L-K_X)^2}
$$ 
are regular at $q=0$, and furthermore that we can write
$$\tau^{-2}\Psi^{F,G}_{X,c_1}(L,\Lambda,S\tau)=\Theta_{X,K_X,c_1}^{F,G}\left(\frac{(L-K_X)\widetilde h}{2\pi  i }, \tau\right)\theta_{2}^{K_X^2}
\widetilde G(\Lambda,\tau).$$
(note that $\sigma(X)+8=K_X^2$ to compare with the formulas in the  proof of \cite[Thm.~4.19]{GY}).
As $\theta_{2}^{K_X^2}$ starts with $q^{K_X^2}$, specializing the formula of \thmref{thetaprop}(2) to  the case
$c=K_X$, $b=c_1$, $F=f$, $G=g$, we see that all the summands in $\Theta_{X,K_X,c_1}^{F,G}\left(\frac{(L-K_X)\widetilde h}{2\pi  i }, \tau\right)$ 
are of the form $q^{-W^2}J_W(\Lambda,\tau)$, where $J_W(\Lambda,\tau)$ is regular at $q=0$ and 
 $W\in K_X+2H^2(X,\Z)$ with $\<F,W\>\le 0\le  \<G,W\>$.
The claim follows.
\end{proof}

\begin{Corollary}\label{strucdiff}
Let $X=\P^1\times\P^1$, or let $X$ be the blowup of $\P^2$ in finitely many general points $p_1,\ldots,p_n$ with exceptional divisors $E_1,\ldots,E_n$. In case $X=\P^1\times\P^1$ let 
$F$ be the class of a fibre of the projection to one of the two factors;  otherwise let $F=H-E_i$ for some $i\in \{1,\ldots,n\}$. 
Let $c_1\in H^2(X,\Z)$ and let $L$ be a line bundle on $X$ with $\<L,c_1\>$ even.
Then 
\begin{enumerate}
\item $\chi^{X,F}_{c_1}(L)=0.$

\item Thus for all $\omega\in S_X\cup C_X$ we have 
$$\chi^{X,\omega}_{c_1}(L)=\Coeff_{q^0}\big[\Psi^{\omega,F}_{X,c_1}(L;\Lambda,\tau)\big].$$
\end{enumerate}
\end{Corollary}

\begin{proof}
(1) Let $\widehat X$ be the blowup of $X$ in a general point with exceptional divisor $E$. Then $\widehat X$ is the blowup of $\P^2$ in $n+1$ general general points, (with $n=1$ in case $X=\P^1\times\P^1$). We denote $E_1,\ldots,E_{n+1}$ the exceptional divisors, then we can assume that $F=H-E_1$. We put $G=H-E_{n+1}$. If $\<c_1, H\>$ is even, we put $\widehat c_1=c_1+E_{n+1}$, and 
if $\<c_1, H\>$ is even, we put $\widehat c_1=c_1$. Thus $\<\widehat c_1, G\>$ is odd and therefore by \remref{vanbound}  we get 
$\chi^{\widehat X, G}_{\widehat c_1}(L)=0$. 
By \propref{blowgen} we have $\chi^{X,F}_{c_1}(L)=\chi^{\widehat X,F}_{\widehat c_1}(L)$ or 
$\chi^{X,F}_{c_1}(L)=\frac{1}{\Lambda}\chi^{\widehat X,F}_{\widehat c_1}(L)$.
Therefore it is enough to show that $\chi^{\widehat X,F}_{\widehat c_1}(L)=\chi^{\widehat X,G}_{\widehat c_1}(L)$.
So by \propref{basic} we need to show that for all $W\in K_{\widehat X}+2H^2(\widehat X,\Z)$ with $\<F,W\>\le 0\le  \<G,W\>$, we have 
$W^2<K_{\widehat X}^2$. 
Let $W=kH+a_1E_1+\ldots +a_{n+1}E_{n+1}\in K_{\widehat X}+2H^2(\widehat X,\Z)$ with $\<F,W\>\le 0\le  \<G,W\>$.
Then $k,a_1,\ldots,a_{n+1}$ are odd integers, the condition 
$\<F,W\>\le 0$ gives that $k\le a_1$, and the condition $\<G,W\>\ge 0$ gives that $k\ge a_{n+1}$.
So either $k<0$ and $|a_1|\ge |k|$ or $k>0$, and $|a_{n+1}|\ge |k|$.
As all the $a_i$ are odd, this gives 
$$W^2=k^2-a_1^2-\ldots-a_{n+1}^2\le -n<8-n=K_X^2.$$
\end{proof}

\subsection{Invariants with point class}
We can now define $K$-theoretic Donaldson invariants with powers of the point class.
\begin{Corollary}
Let $X$ be the blowup of  $\P^2$ in general points $p_1,\ldots,p_r$, with exceptional divisors $E_1,\ldots, E_r$, 
Let $\overline X$ be the blowup of  $\P^2$ in general points $q_1,\ldots,q_r$, with exceptional divisors $\overline E_1,\ldots, \overline E_r$.
For a class $M=dH+a_1E_1+\ldots+a_rE_r\in H^2(X,\R)$ let $\overline M:=dH+a_1\overline E_1+\ldots+a_r\overline E_r\in H^2(\overline X,\R)$.
Then for all $L\in \Pic(X)$, $c_1\in H^2(X,\Z)$ with $\<L,c_1\>$ even, $\omega\in C_X\cup S_X$, we have
$\chi^{X,\omega}_{c_1}(L)=\chi^{\overline X,\overline \omega}_{\overline c_1}(\overline L)$.
\end{Corollary}
\begin{proof} 
Let  $F=H-E_1\in S_X$, then $\overline F=H-\overline E_1\in S_{\overline X}$, and thus
$\chi^{X,F}_{c_1}(L)=0=\chi^{\overline X,\overline F}_{\overline c_1}(\overline L)$.
The map sending $E_i$ to $\overline E_i$ for all $i$ is an isomorphism of lattices, thus 
$\Psi^{\omega,F}_{X,c_1}(L;\Lambda,\tau)=\Psi^{\overline \omega,\overline F}_{\overline X,\overline c_1}(\overline L;\Lambda,\tau)$.
Thus we get by \corref{strucdiff} that
$$\chi^{X,\omega}_{c_1}(L)=\Coeff_{q^0}\big[\Psi^{\omega,F}_{X,c_1}(L;\Lambda,\tau)\big]=\Coeff_{q^0}\big[\Psi^{\overline \omega,\overline F}_{\overline X,\overline c_1}(\overline L;\Lambda,\tau)\big]
=\chi^{\overline X,\overline \omega}_{\overline c_1}(\overline L).$$
\end{proof}

\begin{Definition}
Let $X$ be $\P^1\times\P^1$ or the blowup of $\P^2$ in finitely many general points. Let $\omega\in S_X\cup C_X$, $c_1\in H^2(X,\Z)$, $L\in Pic(X)$.
Let $X_r$ be the blowup of $X$ in $r$ general points, with exceptional divisors $E_1,\ldots,E_r$. Write $E:=E_1+\ldots+E_r$.
We put 
$$\chi^{X,\omega}_{c_1}(L,P^r):=\Lambda^{-r}\chi^{X_r,\omega}_{c_1+E}(L-E),
\quad
\chi^{X,\omega}_{c_1,d}(L,P^r):=\Coeff_{\Lambda^d}\big[\chi^{X,\omega}_{c_1}(L,P^r)\big].$$ 
We call the $\chi^{X,\omega}_{c_1,d}(L,P^r),\ \chi^{X,\omega}_{c_1}(L,P^r)$ the $K$-theoretic Donaldson invariants with point class.
More generally, if $F(\Lambda,P)=\sum_{i,j} a_{i,j} \Lambda^iP^j\in \Q[\Lambda,P]$ is a polynomial, we put
$$\chi^{X,\omega}_{c_1}(L,F(\Lambda,P)):=\sum_{i,j}a_{i,j} \Lambda^i \chi^{X,\omega}_{c_1}(L,P^j).$$
\end{Definition}

\begin{Remark}
There should be a $K$-theory class ${\mathcal P}$ on $M^X_\omega(c_1,d)$, such that 
$\chi^{X,\omega}_{c_1,d}(L,P^r)=\chi(M^X_\omega(c_1,d),\mu(L)\otimes {\mathcal P}^r)$.
By the definition $\chi^{X,M}_{c_1}(L,P)=\Lambda^{-r}\chi^{\widehat X,M}_{c_1+E}(L-E)$ the sheaf $\PP$  would encode local information  at the blown-up point.
We could view $\mathcal P$ as a $K$-theoretic analogue of the point class in Donaldson theory. 
This is our motivation for the name of $\chi^{X,\omega}_{c_1}(L,P^r)$.
For the moment we do not attempt to give a definition of this class ${\mathcal P}$.
There are already speculations about possible definitions of $K$-theoretic Donaldson invariants with powers of the point class in \cite[Sect.~1.3]{GNY}, and the introduction 
of $\chi^{X,\omega}_{c_1}(L,P^r)$ is motivated by that, but for the moment we do not 
try to make a connection to the approach in \cite{GNY}.
\end{Remark}

\section{Blowup polynomials, blowup formulas and blowdown formulas}
In \cite[Section 4.6]{GY} the blowup polynomials 
$R_n(\la,x)$, $S_n(\la,x)$ are introduced. They play a central role in our approach. In this section we will first show that they express all the 
$K$-theoretic Donaldson invariants of the blowup $\widehat X$ of a surface $X$ in terms of the $K$-theoretic Donaldson invariants of $X$.
On the other hand we will use them to show that a small part of the $K$-theoretic Donaldson invariants of the blowup $\widehat X$ determine {\it all} the 
$K$-theoretic Donaldson invariants of $X$ (and thus by the above all the $K$-theoretic Donaldson invariants of any blowup of $X$, including $\widehat X$).
Finally (as already in \cite{GY} in some  cases), in the next section, we will use the blowup polynomials to construct recursion relations for many $K$-theoretic Donaldson invariants of rational ruled surfaces, enough to apply the above-mentioned results and determine all $K$-theoretic Donaldson invariants of $\P^2$ and thus of any blowup of $\P^2$.

\subsection{Blowup Polynomials and blowup formulas}

\begin{Definition}\label{blowpol}
Define for all $n\in \Z$  rational functions  $R_n$, $S_n\in \Q(\la,x)$ by 
$R_0=R_1=1,$   
$S_1=\la,\ S_2=\la x,$
the  recursion relations
\begin{align}
\label{recur} R_{n+1}&=\frac{R_{n}^2-\la^2 S_n^2}{R_{n-1}},\quad  n\ge 1;\qquad
S_{n+1}=\frac{S_{n}^2-\la^2 R_{n}^2}{S_{n-1}},\quad  n\ge 2.
\end{align}
and $R_{-n}=R_{n}$, $S_{-n}=-S_{n}$.
We will prove  later that the $R_n$, 
$S_n$
are indeed polynomials in $\la,x$.
\end{Definition}
The definition gives 
\begin{align*}
R_1&=1,\ R_2=(1-\la^4), R_3=-\la^4 x^2 + (1-\la^4)^2,\ R_4=-\la^4x^4+(1-\la^4)^4,\\
R_5&= -\la^4x^2\big(x^4 +(2-\la^4)(1-\la^4)^2x^2 +3(1-\la^4)^3\big)+(1-\la^4)^6,\\
S_1&=\la,\ S_2=\la x,\ S_3=\la(x^2-(1-\la^4)^2),\ S_4=\la x\big((1-\la^8)x^2-2(1-\la^4)^3\big).
\end{align*}

\begin{Proposition}\label{rnpropo}(\cite[Prop.~4.7]{GY})
\begin{equation}\label{ThRn}
\frac{\widetilde \theta_4(nh)}{\WT_4(h)^{n^2}}=R_n(\Lambda,M),\qquad 
\frac{\WT_1(nh)}{\WT_4(h)^{n^2}}=S_n(\Lambda,M).
\end{equation}
\end{Proposition}

In the following \propref{blowpsi} and \corref{blowp} let $X=\P^1\times\P^1$, or let $X$ be the blowup of $\P^2$ in finitely many general points $p_1,\ldots,p_n$ with exceptional divisors $E_1,\ldots,E_n$. In case $X=\P^1\times\P^1$ let 
$F$ be the class of a fibre of the projection to one of the two factors;  otherwise let $F=H-E_i$ for some $i\in \{1,\ldots,n\}$.

\begin{Proposition}\label{blowpsi}
Let $c_1\in H^2(X,\Z)$ and let $L$ be a line bundle on $X$ with $\<c_1, L\>$ even.
Let $\widehat X$ be the blowup of $X$ in a point, and let $E$ be the exceptional divisor.
Let $\omega\in C_X\cup S_X$. Then
\begin{enumerate}
\item $\chi^{\widehat X,\omega}_{c_1}(L-(n-1)E)=\Coeff_{q^0}\big[\Psi^{\omega,F}_{X,c_1}(L,\Lambda,\tau)R_{n}(\Lambda, M)\big]$.
\item $\chi^{\widehat X,\omega}_{c_1+E}(L-(n-1)E)=\Coeff_{q^0}\big[\Psi^{\omega,F}_{X,c_1}(L,\Lambda,\tau)S_{n}(\Lambda, M)\big]$.
\end{enumerate}
\end{Proposition}

\begin{proof} In \label{highblow}[Prop.~4.34]\cite{GY} it is proven for $X=\P^1\times \P^1$ or the blowup of $\P^2$ in at most 7 points, and any $F,\omega\in C_X\cup S_X$  that 
 \begin{align*}\Psi^{\omega,F}_{\widehat X,c_1}(L-(n-1)E;\Lambda,\tau)&=\Psi^{\omega,F}_{X,c_1}(L,\Lambda,\tau)R_{n}(\Lambda, M)\\
\Psi^{\omega,F}_{\widehat X,c_1+E}(L-(n-1)E;\Lambda,\tau)&=\Psi^{\omega,F}_{X,c_1}(L,\Lambda,\tau)S_{n}(\Lambda, M),
\end{align*}
But the proof works without modification also for $X$ the blowup of $\P^2$ in finitely many points.
The result follows by \corref{strucdiff}.\end{proof}

We now see that the wall-crossing for the  $K$-theoretic Donaldson invariants $\chi^{X,\omega}_{c_1}(L,P^r)$ with point class is given by the 
 wallcrossing terms $\delta^{X}_{\xi}(L,P^r)$.

\begin{Corollary} \label{blowp}
\begin{enumerate}
\item Let $r\ge 0$ and let $L$ be a line bundle on $X$ with $\<L,c_1\>\equiv r\mod 2$.
Then $$\chi^{X,\omega}_{c_1}(L,P^r)=\Coeff_{q^0}\big[\Psi^{\omega,F}_{X,c_1}(L,\Lambda,\tau)M^r\big].$$
\item let $H_1,H_2\in C_X$, not on a wall of type $(c_1)$. Then 
\begin{align*}
\chi^{X,H_2}_{c_1}(L,P^r)-\chi^{X,H_2}_{c_1}(L,P^r)&=\sum_{\xi} \delta^X_{\xi}(L,P^r),\\
\chi^{X,H_2}_{c_1,d}(L,P^r)-\chi^{X,H_2}_{c_1,d}(L,P^r)&=\sum_{\xi} \delta^X_{\xi,d}(L,P^r),
\end{align*}
where in the first (resp.~second) sum $\xi$ runs though all classes of type $(c_1)$ (resp.~$(c_1,d)$)with $\<H_1,\xi\><0<\<H_2,\xi\>$.
\end{enumerate}
\end{Corollary}
\begin{pf} (1)
Let $\widetilde X$ be the blowup of $X$ in $r$ general points and let $\overline E=\overline E_1+\ldots+\overline E_r$ be the sum of the
exceptional divisors. Then by definition and iteration of \propref{blowpsi} we have 
$$\chi^{X,\omega}_{c_1}(L,P^r)=\frac{1}{\Lambda^r}\chi^{\widetilde X,\omega}_{c_1+\overline E}(L-\overline E)=\frac{1}{\Lambda^r}
\Coeff_{q^0}\big[\Psi^{\omega,F}_{X,c_1}(L,\Lambda,\tau)S_{2}(\Lambda, M)^r\big].$$
The claim follows because $S_{2}(\Lambda, M)=\Lambda M$.
(2) By definition $\Coeff_{q^0}\big[\overline \Delta^X_{\xi}(L,P^r)\big]=\delta^X_{\xi}(L,P^r),$ therefore by \lemref{thetawall}, and using also 
\remref{c1d},
(2) follows from (1).
\end{pf}
With this we get a general blowup formula.
\begin{Theorem}\label{nblow}
Let $X$ be $\P^2$, $\P^1\times\P^1$ or the blowup of $\P^2$ in finitely many general points. 
Let $c_1\in H^2(X,\Z)$ and let $L$ be a line bundle on $X$ and let $r\in \Z_{\ge 0}$ with $\<c_1, L\>\equiv r\mod 2$.
Let $\omega\in C_X\cup S_X$. 
Let $\widehat X$ be the blowup of $X$ in a general point with exceptional divisor $E$.
Then 
\begin{enumerate}
\item $\chi^{\widehat X,\omega}_{c_1}(L-(n-1)E,P^r)=\chi^{X,\omega}_{c_1}(L,P^r \cdot R_n(\Lambda,P))$, 
\item $\chi^{\widehat X,\omega}_{c_1+E}(L-(n-1)E,P^r)=\chi^{X,\omega}_{c_1}(L,P^r \cdot S_n(\Lambda,P))$.
\end{enumerate}
\end{Theorem}
\begin{proof}
If $X=\P^2$, then we apply \propref{blowgen} to reduce to the case that $X$ is the blowup of $\P^2$ in a point. 
Thus we can by \corref{strucdiff} and the definition of $\chi^{X,G}_{c_1}(L,P^s)$,  assume that there is an $G\in S_X$ with $\chi^{X,G}_{c_1}(L,P^s)=0$ for all $s\ge 0$.

(1) Let $\widetilde X$ be the blowup of $X$ in $r$ general points, with exceptional divisors $F_1,\ldots,F_r$ and put 
$F:=F_1+\ldots+F_r$, and let $\overline X$ the blowup of $\widetilde  X$ in a point with exceptional divisor $E$.
Then by definition
$$\chi^{\widehat X,\omega}_{c_1}(L-(n-1)E,P^r)=\chi^{\overline X,\omega}_{c_1+F}(L-F-(n-1)E).$$
We get by  \corref{strucdiff} that 
$$\chi^{\overline X,\omega}_{c_1+F}(L-F-(n-1)E)=\Coeff_{q^0}\big[\Psi^{\omega,G}_{\widetilde X,c_1+F}(L-F,\Lambda,\tau)R_{n}(\Lambda, M)\big].$$
On the other hand, by \corref{blowp} we get 
$$\Coeff_{q^0}\big[\Psi^{\omega,G}_{\widetilde X,c_1+F}(L-F,\Lambda,\tau)R_{n}(\Lambda, M)\big]=\chi^{\widetilde X,\omega}_{c_1+F}(L-F,R_{n}(\Lambda, P))=\chi^{X,\omega}_{c_1}(L,P^r\cdot R_n(\Lambda,P)).$$
The proof of (2) is similar.
\end{proof}

\subsection{Further properties of the blowup polynomials}
\begin{Proposition}\label{blowpolprop}
\begin{enumerate}
\item For all $n\in \Z$,   we have  $R_n
\in \Z[\la^4,x^2]$, $S_{2n+1}\in \la\Z[\la^4,x^2]$, $S_{2n}\in \la x\Z[\la^4,x^2]$.
\item $R_n(\lambda,-x)=R_n(\lambda,x)$ and $S_n(\lambda,-x)=(-1)^{n-1} S_n(\lambda,x)$.
\item The $R_n$, $S_n$ satisfy the symmetries
\begin{align*}
R_{2n}\Big(\frac{1}{\la},\frac{x}{\la^2}\Big)&=\frac{(-1)^n}{\la^{(2n)^2}}R_{2n}(\la,x),\quad S_{2n}\Big(\frac{1}{\la},\frac{x}{\la^2}\Big)=\frac{(-1)^{n-1}}{\la^{(2n)^2}}S_{2n}(\la,x),\\
R_{2n+1}\Big(\frac{1}{\la},\frac{x}{\la^2}\Big)&=\frac{(-1)^n}{\la^{(2n+1)^2}}
S_{2n+1}(\la,x),\quad S_{2n+1}\Big(\frac{1}{\la},\frac{x}{\la^2}\Big)=\frac{(-1)^n}{\la^{(2n+1)^2}}
R_{2n+1}(\la,x).
\end{align*}

\item For all $k,n\in \Z$, we have the relations
\begin{equation}\label{r2nRn}
\begin{split}
R_{2n}&=R_n^4-S_n^4, \quad S_{2n}=\frac{1}{\la}R_nS_n(S_{n+1}R_{n-1}-R_{n+1}S_{n-1}).
\end{split}
\end{equation}
\end{enumerate}
\end{Proposition}
\begin{proof}
We write 
$$\widetilde R_n(h):=\frac{\WT_4(nh)}{\WT_4(h)^{n^2}}=R_n(\Lambda,M),\quad \widetilde S_n(h):= \frac{\WT_1(nh)}{\WT_4(h)^{n^2}}=S_n(\Lambda,M),$$ 
where we have used \eqref{ThRn}.
It is easy to see that $\Lambda$ and $M$ are algebraically independent, i.e. there exists no polynomial $f\in \Q[\lambda,x]\setminus \{0\}$, such that $f(\Lambda,M)=0$ as 
a function on $\H\times \C$. For this, note that by $M^2=4(1+u\Lambda^2+\Lambda^4)$, the algebraic independence of $\Lambda$ and $M$ is equivalent to that of $\Lambda$ and $u$.
But this is clear, because as Laurent series in $q,y$,   $u$ is a Laurent series in $q$ starting with $-\frac{1}{4q^{-2}}$ and $\Lambda$ depends on $y$ in a nontrivial way.
As $M$ and $\Lambda$ are algebraically independent, $R_n$, $S_n$
are the unique rational functions 
satisfying \eqref{ThRn} .

Now we will show (4). 
For any $k\in \Z$ we also have
\begin{equation}\label{RSkn}
\begin{split}
\widetilde R_{kn}(h)&= \frac{\WT_4(knh)}{\WT_4(h)^{k^2 n^2}}=\frac{\WT_4(knh)}{\WT_4(nh)^{k^2}}\Big(\frac{\WT_4(nh)}{\WT_4(h)^{n^2}}\Big)^{k^2}=
\widetilde R_k(nh) \widetilde R_n(h)^{k^2},\\
\widetilde S_{kn}(h)&=\frac{\WT_1(kn z)}{\WT_4(nh)^{k^2}}=\widetilde S_k(nh)\widetilde R_n(h)^{k^2}
\end{split}
\end{equation}
Thus, using $\widetilde R_2(h)=1-\Lambda^4$, we find in particular
$$\widetilde R_{2n}(h)=\widetilde R_2(nh) \widetilde R_n(h)^{4}=
\left(1-\left(\frac{\WT_1(nh)}{\WT_4(nh)}\right)^4\right)\widetilde R_n(h)^4=
\widetilde R_n(h)^4-\widetilde S_n(h)^4;$$
i.e., using the algebraic independence of $\Lambda$ and $M$, $R_{2n}=R_n^4- S_n^4$.
In the same way we have 
 $$\widetilde S_{2n}(h)=\widetilde S_2(nh)\widetilde R_n(h)^{4}=\Lambda(nh)M(nh)\widetilde R_n(h)^{4}.$$
 By definition $\Lambda(nh)=\frac{\WT_1(nh)}{\WT_4(nh)}=\frac{\widetilde S_n(h)}{\widetilde R_n(h)}$. 
Now take the difference of the two formulas (see \cite[\S 2.1 Ex.~3]{WW})
$$\theta_1(y\pm z)\theta_4(y\mp z)\th_2\th_3=\theta_1(y)\theta_4(y)\theta_2(z)\theta_3(z)\pm
\theta_2(y)\theta_3(y)\theta_1(z)\theta_4(z)$$ with $y=nh$, $z=h$
to get 
$$\theta_1((n+1)h)\theta_4((n-1)h)\theta_2\theta_3-\theta_1((n-1)h)\theta_4((n+1)h)\theta_2\theta_3=2\theta_2(nh)\theta_3(nh)\theta_1(h)\theta_4(h).$$
This gives
\begin{align*}
M(nh)&=2\frac{\theta_2(nh)\theta_3(nh)}{\th_2\th_3\theta_4(nh)^2}=
\frac{\theta_1((n+1)h)\theta_4((n-1)h)-\theta_1((n-1)h)\theta_4((n+1)h)}{\theta_1(h)\theta_4(h)
\theta_4(nh)^2}\\&=\frac{1}{\Lambda}\frac{\widetilde S_{n+1}(h)\widetilde R_{n-1}(h)-\widetilde S_{n-1}(h)\widetilde R_{n+1}(h)}{\widetilde R_n(h)^2}
\end{align*}
Thus 
$S_{2n}=\frac{1}{\lambda}S_nR_n(S_{n+1}R_{n-1}-S_{n-1}R_{n+1}).$
This shows (4)

(1) Next we will show that 
$R_n\in \Z[\lambda^4,x]$ and $S_n \in \lambda\Z[\lambda^4,x]$ for all $n\in \Z$.
By symmetry it is enough to show this if $n$ is a nonnegative  integer. 
We know that this is true for $0\le n\le 4$. 
Now assume that $m\ge 2$, and that we know
the statement for all $0\le n\le 2m$.
Therefore $R_{m+1}\in \Z[\lambda^4,x], S_{m+1}\in \lambda\Z[\lambda^4,x]$, and  the formulas \eqref{r2nRn} give that 
$R_{2m+2}\in \Z[\lambda^4,x]$, $S_{2m+2}\in \lambda\Z[\lambda^4,x]$.
The relations \eqref{recur} say that
\begin{align*}
R_{2m+2}R_{2m}&=R_{2m+1}^2-\la^2S_{2m+1}^2,\quad
S_{2m+2}S_{2m}=S_{2m+1}^2-\la^2 R_{2m+1}^2,
\end{align*}
and thus
\begin{equation}
\label{r2n1}
\begin{split}
(1-\la^4)R_{2m+1}^2&=R_{2m+2}R_{2m}+\la^2 S_{2m+2}S_{2m},\\
(1-\la^4)S_{2m+1}^2&=S_{2m+2}S_{2m}+\la^2R_{2m+2}R_{2m}.
\end{split}
\end{equation}
Thus we get $(1-\lambda^4)R_{2m+1}^2\in \Z[\lambda^4,x]$ and
$(1-\lambda^4)S_{2m+1}^2\in \lambda^2\Z[\lambda^4,x]$. Therefore, as   $1-\la^4$ is squarefree in $\Q[\lambda,x]$, we also have 
$R_{2m+1}^2\in \Z[\lambda^4,x]$ and
$S_{2m+1}^2\in \lambda^2\Z[\lambda^4,x]$. As we already know that $R_{2m+1},\ S_{2m+1}\in \Q(\lambda,x)$, this gives 
$R_{2m+1}\in\Z[\lambda^4,x]$ and $S_{2m+1}\in \lambda\Z[\lambda^4,x]$. So $R_{2m+1}$, $R_{2m+2}\in \Z[\lambda^4,x]$ and  $S_{2m+1}$, $S_{2m+2}\in \lambda\Z[\lambda^4,x]$.
Thus  by induction on $m$, we get $R_n\in \Z[\lambda^4,x]$, $S_n\in \lambda\Z[\lambda^4,x]$.

(2) For $n=0,1,2$ we see immediately that the $R_n$ are even in $x$ and the 
$S_n$ have partity $(-1)^{n-1}$ in $x$. On the other hand the 
recursion formulas \eqref{recur} say that $R_{n+1}$ has the same parity as $R_{n-1}$ in $x$ and $S_{n+1}$ the same parity as $S_{n-1}$.
This also shows that $R_n\in \Z[\lambda^4,x^2]$, $S_{2n}\in \lambda x\Z[\lambda^4,x^2]$, $S_{2n+1}\in \lambda \Z[\lambda^4,x^2]$.

(3)
The formulas \eqref{T4trans},\eqref{T1trans},\eqref{T2trans} imply
\begin{equation}\label{Lambdatau}
\Lambda(h+\pi i \tau)=\frac{\theta_4(h)}{\theta_1(h)}=\frac{1}{\Lambda},\quad
M(h+\pi i \tau)=-2\frac{\WT_2(h)\WT_3(h)}{\WT_1(h)^2}=-\frac{M}{\Lambda^2}.
\end{equation}
Part (1) of \lemref{thetaadd}  shows
$$
\theta_4(2nh+2\pi i n\tau)=(-1)^n q^{-4n^2}(y^{2n})^{-2n}\theta_4(2nh).
$$
Thus, using \eqref{T4trans} again, we get
$$
\widetilde R_{2n}(h+\pi i \tau)=(-1)^n\frac{\WT_4(2nh)}{\WT_1(h)^{4n^2}}=(-1)^n \frac{\widetilde R_{2n}(h)}{\Lambda^{(2n)^2}}.
$$
As $\Lambda^4$ and $M$ are algebraically independent, \eqref{Lambdatau} and \eqref{R2ntau} imply that
\begin{equation}\label{R2ntau}R_{2n}\Big(\frac{1}{\lambda},\frac{x}{\lambda^2}\Big)=R_{2n}\Big(\frac{1}{\lambda},-\frac{x}{\lambda^2}\Big)=(-1)^n\frac{1}{\la^{(2n)^2}}R_{2n}(\la,x).\end{equation}
The same argument using part (2) of \lemref{thetaadd} and \eqref{T1trans} shows
$
\widetilde S_{2n}(h+\pi i \tau)=(-1)^n \frac{\widetilde S_{2n}(h)}{\Lambda^{(2n)^2}},
$
 and thus
\begin{equation}\label{S2ntau}S_{2n}\Big(\frac{1}{\lambda},\frac{x}{\lambda^2}\Big)=-S_{2n}\Big(\frac{1}{\lambda},-\frac{x}{\lambda^2}\Big)=(-1)^{n-1}\frac{1}{\la^{(2n)^2}}S_{2n}(\la,x).\end{equation}
Similarly using parts (3) and (4) of \lemref{thetaadd} we get by the same arguments
\begin{align*}
\widetilde R_{2n+1}(h+\pi i \tau)=\frac{\WT_4((2n+1)h+\pi i (2n+1)\tau)}{\WT_4(h+\pi i \tau)^{(2n+1)^2}}=
(-1)^n \frac{\WT_1((2n+1)h)}{\WT_1(h)^{(2n+1)^2}}=(-1)^n \frac{ \widetilde S_{2n+1}}{\Lambda^{(2n+1)^2}},
\end{align*}
and thus 
$$R_{2n+1}\big(\frac{1}{\la},\frac{x}{\la^2}\big)=R_{2n+1}\big(\frac{1}{\la},-\frac{x}{\la^2}\big)=\frac{(-1)^n}{\la^{(2n+1)^2}} S_{2n+1}.$$
The same argument shows
$S_{2n+1}\big(\frac{1}{\la},\frac{x}{\la^2}\big)=S_{2n+1}\big(\frac{1}{\la},-\frac{x}{\la^2}\big)=\frac{(-1)^n}{\la^{(2n+1)^2}} R_{2n+1}.$
\end{proof}

\subsection{Blowdown formulas}
Let $\widehat X$ be the blowup of a rational surface $X$ in a point. 
As mentioned at the beginning of this section, the blowup polynomials determine a blowup formula which computes the $K$-theoretic Donaldson invariants  $\widehat X$ in terms of those of $X$. We will also need a blowdown
formula which determines all the $K$-theoretic Donaldson invariants of $X$ in 
terms of a small part of those of $\widehat X$. In order to  prove the blowdown formula, we will need that,  for $n,m$ relatively prime integers, 
the polynomials $R_n$ $R_{m}$ and $S_n$, $S_{m}$  are as polynomials in $x$ in a suitable sense relatively prime.


\begin{Proposition}\label{blowdownpol} Let $n,m\in \Z$ be relatively prime. 
\begin{enumerate}
\item There exists a minimal integer $M^0_{n,m}\in \Z_{\ge 0}$ and unique polynomials 
$h^0_{n,m},\ l^0_{n,m}\in \Q[\lambda^4,x^2]$, such that $(1-\lambda^4)^{M^0_{n,m}}=h^0_{n,m} R_n+l^0_{n,m} R_m$.

\item 
 There exists a minimal integers $M^1_{n,m}\in \Z_{\ge 0}$ and unique polynomials 
$h^1_{n,m},\ l^1_{n,m}\in \Q[\lambda^4,x]$, such that $\lambda(1-\lambda^4)^{M^1_{n,m}}=h^1_{n,m}  S_n+l^1_{n,m}  S_m$.
\end{enumerate}
\end{Proposition}

\begin{proof}

For all $l\in \Z$ we write
$$\overline S_{2l}:=x\frac{S_{2l}}{\lambda}=\frac{S_2S_{2l}}{S_1^2}, \quad \overline S_{2l+1}:=\frac{S_{2l+1}}{\lambda}=\frac{S_{2l+1}}{S_1}\in \Z[\lambda^4,x^2].$$
Let $I_{n,m}=\<R_n,R_{m}\>\subset \Z[\lambda^4,x^2]$ be the ideal generated by $R_n, R_{m}\in \Z[\lambda^4,x^2]$, and let
$J_{n,m}=\<\overline S_n,\overline S_{m}\>\subset \Z[\lambda^4,x^2]$ be the ideal generated by $\overline S_n,\overline S_{m}\in \Z[\lambda^4,x^2]$.
Then the Proposition follows immediately from the following.

\begin{Claim}[1]
There are $M^0_{n,m} , M^1_{n,m}\in \Z_{\ge 0}$ with $(1-\lambda^4)^{M^0_{n,m}}\in I_{n,m}$ and  $(1-\lambda^4)^{M^1_{n,m}}\in J_{n,m}$.
\end{Claim}


Let 
\begin{align*}
V_{n,m}&:=\big\{(\alpha^4,\beta^2)\in \C^2\bigm| R_{n}(\alpha,\beta)=R_{m}(\alpha,\beta)=0 \big\}, \\
W_{n,m}&:=\big\{(\alpha^4,\beta^2)\in \C^2\bigm| \overline S_{n}(\alpha,\beta)=\overline S_{m}(\alpha,\beta)=0 \big\}.
\end{align*}
Then by the Nullstellensatz the Claim (1)  follows immediately from the following.


\begin{Claim}[2]
 $V_{n,m}, W_{n,m}\subset \{(1,0)\}$,
 \end{Claim}

\noindent{\it Proof of  Claim(2):}
The idea of the proof is as follows:
For each $(\alpha,\beta)\in \C^2$ with $(\alpha^4,\beta^2)\in \C^2\setminus \{(1,0)\}$ we want to show that 
\begin{enumerate}
\item $R_n(\alpha,\beta)$ or $R_{m}(\alpha,\beta)$ is nonzero, 
\item $\overline S_n(\alpha,\beta)$ or $\overline S_{m}(\alpha,\beta)$ is nonzero.
\end{enumerate}
Recall that we have $\widetilde R_n=R_n(\Lambda,M)$, and we put $\widehat S_n:=\overline S_n(\Lambda,M)$, so that
$\widehat S_{2n}=\frac{M \overline S_{2n}}{\Lambda}$ and $\widehat S_{2n+1}=\frac{\overline S_{2n+1}}{\Lambda}$.
We denote 
\begin{align*}
\Lambda|_S(h,\tau)&:=\Lambda(\frac{h}{\tau},-\frac{1}{\tau}), \quad M|_S(h,\tau):=M(\frac{h}{\tau},-\frac{1}{\tau}),\\
\widetilde R_m|_S(h,\tau)&:=\widetilde R_m(\frac{h}{\tau},-\frac{1}{\tau}), \quad \widehat S_m|_S(h,\tau)=\widehat S_m(\frac{h}{\tau},-\frac{1}{\tau})
\end{align*}
the application of the operator $S:(h,\tau)\mapsto (\frac{h}{\tau},-\frac{1}{\tau})$ to the Jacobi functions $\Lambda$, $M$, $\widetilde R_m$, $\widehat S_m$.
Obviously we have 
$$R_m(\Lambda|_S,M|_S)= \widetilde R_m|_S, \quad \overline S_m(\Lambda|_S,M|_S)= \widehat S_m|_S.$$
We denote $Z(f)\subset \C$ the zero set of a meromorphic function $f:\C\to\C$.  
Therefore Claim (2) will follow once we prove the following  facts:
\begin{enumerate}
\item Every $(\alpha, \beta)\in \C^2\setminus  \{(1,0)\}$ can be written as 
$(\Lambda(h,\tau)^4,M(h,\tau)^2)$ for some $h\in \C$, $\tau\in \H\cup\{\infty\}$ or as $(\Lambda^4|_S(h,\infty)$, $M^2|_S(h,\infty))$ for some $h\in \C$. 
Here we by $\Lambda(h,\infty)$, $M(h,\infty)$, $(\Lambda|_S(h,\infty)$, $M|_S(h,\infty))$, we mean the coefficient of $q^0$ of the $q$-development of 
$\Lambda$, $M$, ($\Lambda|_S$, $M|_S$) (asserting also that these developments are power series in $q$).

\item For all $\tau\in \H\cup \{\infty\}$ we have 
\begin{align*}
Z(\widetilde R_n(\bullet,\tau))\cap Z(\widetilde R_{m}(\bullet,\tau))&:= \big\{h\in \C \bigm|\widetilde R_n(h,\tau)=\widetilde R_{m}(h,\tau)=0\big\}=\emptyset,\\
Z(\widetilde R_n|_S(\bullet,\infty))\cap Z(\widetilde R_{m}|_S(\bullet,\infty))&:= \big\{h\in \C \bigm|\widetilde R_n|_S(h,\infty)=\widetilde R_{m}|_S(h,\infty)=0\big\}=\emptyset.
\end{align*}
\item For all $\tau\in \H\cup \{\infty\}$ we have
\begin{align*}Z(\widehat S_n(\bullet,\tau))\cap  Z(\widehat S_{m}(\bullet,\tau))&:= \big\{h\in \C \bigm|\widehat S_n(h,\tau)=\widehat S_{m}(h,\tau)=0\big\}=\emptyset,\\
Z(\widehat S_n|_S(\bullet,\infty)\cap Z(\widehat  S_{m}|_S(\bullet,\infty))&:= \big\{h\in \C \bigm|\widehat S_n|_S(h,\infty)=\widehat S_{m}|_S(h,\infty)=0\big\}=\emptyset.
\end{align*}
\end{enumerate}
(1) For any fixed $\tau\in \H$ the range of the elliptic function  $\Lambda=\Lambda(\tau,\bullet)$ is  $\C\cup \infty$. 
$u$ is a Hauptmodul   for $\Gamma^0(4)$, which takes the values
$-2,2,\infty$ at the cusps $0,2,\infty$ respectively. Therefore the range of $u$ as a function on $\H$ is 
$\C\setminus \{-2,2\}$.
By the equation 
 $M^2=4(1+u\Lambda^2+\Lambda^4)$, we get therefore that the range of $(\Lambda^4, M^2)$  on $\H\times \C$ contains 
 the set
 $$I_1:=\C^2\setminus \{(c^2,4(1+c)^2)\ |\ c\in \C\}.$$
 
 Now we look at $\tau=\infty$, i.e. $q=0$.
 By the $q$-developments \eqref{theta}, we see that 
 \begin{equation}
 \label{thetaq0}
 \begin{split}
 \WT_1(h,\tau)&=O(q),\quad 
  \WT_4(h,\tau)=1+O(q),
 \quad \WT_3(h,\tau)=1+O(q),\\
 \WT_2(h,\tau)&=\cosh(h/2)+O(q),\quad  \frac{\WT_1(h,\tau)}{\WT_2(h,\tau)}=-i\tanh(h/2)+O(q).
 \end{split}
 \end{equation}
 Therefore we get from the  definitions
 $$\Lambda(h,\tau)=O(q), \quad M(h,\tau)=2\frac{\WT_2(h,\tau)\WT_3(h,\tau)}{\WT_4(h,\tau)^2}=2\cosh(h/2)+O(q).$$
  As $\cosh:\C\to \C$ is surjective, we see that the range of 
 $(\Lambda^4, M^2)$  on $ \C\times \{\infty\} $ is $I_2:=\{0\}\times \C$.
 From the definitions and  \eqref{thetaq0} we obtain
 \begin{align*}
 \Lambda^4|_S(h,\tau)&=\frac{\WT_1(h,\tau)^4}{\WT_2(h,\tau)^4}=\tanh(h/2)^4+O(q),\\
 M^2|_S(h,\tau)&=4\frac{\WT_3(h,\tau)^2\WT_4(h,\tau)^2}{\WT_2(h,\tau)^4}=\frac{4}{\cosh(h/2)^4}+O(q)=
 4(1-\tanh(h/2)^2)^2+O(q).\end{align*}
 It is an easy exercise that the range of $\tanh:\C\to \C$ is $\C\setminus\{\pm 1\}$.
Thus the range of $( \Lambda^4|_S,M^2|_S)$ on $\C\times \{\infty\}$ is
$$I_3=\big \{(c^2,4(1-c)^2)\bigm| c\in \C\setminus \{1\}\big\}.$$
As $I_1\cup I_2\cup I_3=\C^2\setminus \{(1,0)\}$, (1) follows.

(2) First let $\tau$ in $\H$. It is standard that $\theta_1(h)$ and $\theta_4(h)$ are holomorphic in $h$ on $\C$ and 
$Z(\th_1(h))=2\pi i (\Z+\Z\tau)$, $Z(\th_4(h))=2\pi i (\Z+(\Z+\frac{1}{2})\tau)$. 
Thus by $\widetilde R_n=\frac{\WT_4(nz)}{\WT_4(z)^{n^2}}$, 
we see that
\begin{align*}
Z(\widetilde R_n(\bullet,\tau))&=\Big\{2\pi i\big(\frac{a}{n}+\frac{b}{2n}\tau\big)\bigm| a,b\in \Z,\ b \hbox{ odd},
(a,b)\not \equiv (0,0) \hbox{ mod }n
\Big\}
\end{align*}
Assume that  $2\pi i \big(\frac{a}{n}+\frac{b}{2n}\tau \big)=2\pi i \big(\frac{a'}{m}+\frac{b'}{2m}\tau \big)\in \Z(\widetilde R_n(\bullet,\tau))\cap Z(\widetilde R_{m}(\bullet, \tau))$.
As $n$ and $m$ are relatively prime, we see that there $a'',b''\in \Z$, such that 
$$\frac{b}{2n}=\frac{b'}{2m}=\frac{b''}{2}, \quad \frac{a}{n}=\frac{a'}{m}=a''.$$
Thus $a$ and $b$ are both divisible by $n$, and thus $2\pi i \big(\frac{a}{n}+\frac{b}{2n}\tau\big)\not\in \Z(\widetilde R_n(\bullet,\tau))$.

Now let $\tau=\infty$
Then $\widetilde R_n(h,\infty)=\frac{\WT_4(nh,\infty)}{\WT_4(h,\infty)^{n^2}}=1$. Thus $Z(R_n(\bullet,\infty))=\emptyset$.

Finally we consider  $\widetilde R_n|_S(h,\infty)$.
We have 
\begin{align*}
\widetilde R_n|_S(h,\infty)&=\frac{\WT_2(nh,\infty)}{\WT_{2}(h,\infty)^{n^2}}=
\frac{\cosh(n h/2)}{\cosh( h/2)^{n^2}},\quad
\end{align*}
This gives 
\begin{align*}
Z(\widetilde R_n|_S(\bullet,\infty))&=\Big\{\pi i \frac{b}{2n}\Bigm| b\in \Z \hbox{ odd, } n\not|b\Big\},\quad
\end{align*}
and again it is clear that $Z(\widetilde R_n|_S(\bullet,\infty))\cap Z(\widetilde R_m|_S(\bullet,\infty))=\emptyset$.

(3) We note that
\begin{align*}
\widehat S_{2l+1}&=\frac{\widetilde S_{2l+1}}{\widetilde S_1}=\frac{\th_{1}((2l+1)h)}{\th_1(h)\WT_4(h)^{4l^2+4l}},\\
\widehat S_{2l}&=\frac{\widetilde S_2\widetilde S_{2l}}{\widetilde S_1^2}=\frac{\th_{1}(2lh)\th_1(2h)}{\th_1(h)^2\WT_4(h)^{4l^2+2}}.
\end{align*}
Let $\tau\in \H$, then this gives 
\begin{align*}
Z(\widehat S_{2l+1}(\bullet, \tau))&=\big\{2\pi i (\frac{a}{2l+1}+\frac{b}{2l+1}\tau)\bigm| a,b\in \Z, \ (a,b)\not \equiv (0,0)\mod 2l+1\big\},\\
Z(\widehat S_{2l}(\bullet, \tau))&=\big\{2\pi i (\frac{a}{2l}+\frac{b}{2l}\tau)\bigm| a,b\in \Z, \ (a,b)\not \equiv (0,0)\mod 2l\big\}.
\end{align*}
Thus we see immediately that $Z(\widehat S_{n}(\bullet, \tau))\cap Z(\widehat S_{m}(\bullet, \tau))=\emptyset,$ if $n$ and $m$ are relatively prime.

Now let $\tau=\infty$. Then $$\widehat S_{2l+1}(h, \infty)=\frac{\sinh((2l+1)h/2)}{\sinh(h/2)},\quad \widehat S_{2l}(h ,\infty)=\frac{\sinh(lh)\sinh(h)}{\sinh(h/2)^2},$$
So it is easy to see that $Z(\widehat S_{n}(\bullet,\infty))\cap Z(\widehat S_{m}(\bullet, \infty))=\emptyset,$  if $n$ and $m$ are relatively prime.
Finally $$\widehat S_{2l+1}|_S(h,\tau)=\frac{\th_1((2l+1)h)}{\th_1(h)\WT_2(h)^{4l^2+4l}},\quad
\widehat S_{2l}|_S(h,\tau)=\frac{\th_1(2lh) \th_1(2h)}{\th_1(h)^2\WT_2(h)^{4m^2+2}}.$$
Thus we get $$\widehat S_{2l+1}|_S(h,\infty)=\frac{\sinh((2l+1)h/2)}{\sinh(h/2)\cosh(h/2)^{4l^2+4l}},\ 
\widehat S_{2l}|_S(h,\infty)=\frac{\sinh(lh)\sinh(h)}{\sinh(h/2)^2\cosh(h/2)^{4l^2+2}},$$
 and again it is evident that for $n$ and $m$ relatively prime
$Z(\widehat S_{n}|_S(\bullet,\infty))\cap Z(\widehat S_{m}|_S(\bullet, \infty))=\emptyset.$
\end{proof}

\begin{Corollary}\label{blowdownmn}
Let $n,m\in \Z$ be relatively prime. 
Let $X$ be $\P^2$, $\P^1\times\P^1$ or the blowup of $\P^2$ in finitely many general points.
Let $c_1\in H^2(X,\Z)$, let $L$ be a line bundle on $X$, let $r\in \Z_{\ge 0}$ with $\<c_1,L\>\equiv r\mod 2$.
Let $\widehat X$ be the blowup of $X$ in a point. Let $\omega\in S_X$
Using the notations of \propref{blowdownpol}, we have
\begin{align*}
\chi^{X,\omega}_{c_1}(L,P^r)&=\frac{1}{(1-\Lambda^4)^{M^0_{n,m}}}\big(\chi^{\widehat X,\omega}_{c_1}(L-(n-1)E,P^r\cdot h^0_{m,n}(\Lambda,P)\\
&\qquad+
\chi^{\widehat X,\omega}_{c_1}(L-(m-1)E,P^r\cdot l^0_{m,n}(\Lambda,P)\big)\\
&=\frac{1}{\Lambda(1-\Lambda^4)^{M^1_{n,m}}}\big(\chi^{\widehat X,\omega}_{c_1+E}(L-(n-1)E,P^r\cdot h^1_{n,m}(\Lambda,P)\\
&\qquad+
\chi^{\widehat X,\omega}_{c_1+E}(L-(m-1)E,P^r\cdot l^1_{n,m}(\Lambda,P)\big).
\end{align*}
\end{Corollary}
\begin{proof} 
(1) By \thmref{nblow} we have
\begin{align*}
\big(&\chi^{\widehat X,\omega}_{c_1}(L-(n-1)E,P^r\cdot h^0_{m,n}(\Lambda,P)+
\chi^{\widehat X,\omega}_{c_1}(L-(m-1)E,P^r\cdot l^0_{m,n}(\Lambda,P)\big)\\&
=
\chi^{X,\omega}_{c_1}\big(L,P^r\cdot \big(R_n(\Lambda,P)h^0_{m,n}(\Lambda,P)+R_{m}(\Lambda,P)l^0_{n,m}(\Lambda,P)\big)\big)=
(1-\Lambda^4)^{M^0_{n,m}} \chi^{X,\omega}_{c_1}(L,P^r),
\end{align*}
where in the last step we use \propref{blowdownpol}.
The proof of (2) is similar.
\end{proof}

\section{Recursion formulas for rational ruled surfaces}

\subsection{The limit of the invariant at the boundary point}
For $X=\P^1\times \P^1$ or $X=\widehat \P^2$ the blowup of $\P^2$ in a point,
we denote the line bundles on $X$ in a uniform way.
\begin{Notation}
Let $X=\P^1\times \P^1$ or $X=\widehat \P^2$. In the case $X=\P^1\times \P^1$ we denote $F$ the class of the fibre of the projection to the first factor, and by $G$ the class of the fibre of the projection to the second factor. In the case $X=\widehat \P^2$, let $H$ be the pullback of the hyperplane class on $\P^2$ and $E$ the class of the exceptional divisor. Then  $F:=H-E$ is the fibre of the ruling of $X$. We put $G:=\frac{1}{2}(H+E)$. Note that $G$ is not an integral cohomology class. In fact, while $H^2(\P^1\times\P^1,\Z)=\Z F\oplus \Z G$, we have
$$H^2(\widehat \P^2,\Z)=\Z H\oplus \Z E=\big\{aF+bG\bigm| a\in \Z,b\in 2\Z \hbox{ or } a\in \Z+\frac{1}{2}, b\in 2\Z+1\big\}.$$
On the other hand we note that both on $X=\P^1\times\P^1$ and $\widehat \P^2$ we have
$F^2=G^2=0$, $\<F,G\>=1$, and $-K_X=2F+2G$.
\end{Notation}

We want to define and study the limit of the $K$-theoretic Donaldson invariants $\chi^{X,\omega}_{c_1}(L,P^r)$ as the ample class $\omega$ tends to $F$.
For $c_1=F$ or $c_1=0$ this will be different from our previous definition of $\chi^{X,F}_{c_1}(L,P^r)$.

\begin{Definition}
Let $r\in \Z_{\ge 0}$, let $L\in \Pic(X)$ with $\<c_1,L\>+r$ even. Fix $d\in \Z$ with $d\equiv -c_1^2 \mod 4$. For $n_{d,r}>0$ sufficiently large, $n_{d,r}F+ G$ is ample on $X$, 
and 
there is no wall $\xi$ of type $(c_1,d)$ with 
$\<\xi ,(n_{d,r}dF+ G)\>>0> \<\xi, F\>$.   
For all  $\omega\in S_X\cup C_X$, we define  $\chi^{X,\omega}_{c_1,d}(L,P^r):=\Coeff_{\Lambda^d}\big[\chi^{X,\omega}_{c_1,d}(L,P^r)\big],$ and put 
\begin{align*}\\
\chi^{X,F_+}_{c_1,d}(L,P^r)&:=\chi^{X,n_{d,r}F+ G}_{c_1,d}(L,P^r),\quad
\chi^{X,F_+}_{c_1}(L):=\sum_{d\ge 0} \chi^{X,F_+}_{c_1,d}(L,P^r)\Lambda^d.
\end{align*}
\end{Definition}

Now we give a formula for $\chi^{X,F_+}_{0}(nF+mG,P^r)$ and $\chi^{X,F_+}_{F}(nF+mG,P^r)$. The result and the proof are similar to \cite[Prop.~5.3]{GY}.
The rest of this section will be mostly devoted to giving an explicit 
evaluation of this formula for $m\le 2$. 

\begin{Proposition}\label{Fplus}
Let $X=\P^1\times\P^1$ or $X=\widehat \P^2$.
\begin{enumerate}
\item 
Let $nF+mG$ be a line bundle on $X$ with $m$ even. Then
$$\chi^{X,F_+}_{F}(nF+mG,P^r)=\Coeff_{q^0}\left[\frac{1}{2\sinh((m/2+1)h)}\Lambda^2\WT_4(h)^{2(n+2)(m+2)}u'h^*M^r\right].$$

\item 
Let $nF+mG$ be a line bundle on $X$ (note that we might have $n\in \frac{1}{2}\Z$. Then
$$\chi^{X,F_+}_{0}(nF+mG,P^r)=-\Coeff_{q^0}\left[\frac{1}{2}\left(\coth((m/2+1)h)\right)\Lambda^2\WT_4(h)^{2(n+2)(m+2)}u'h^*M^r\right].$$
\end{enumerate}

\end{Proposition}

\begin{proof}
We denote $\Gamma_X=H^2(X,\Z)$ with inner product the negative of the intersection form. 
Let $c_1=0$ or $c_1=F$,  fix $d$, and let $s\in \Z_{\ge 0}$ be sufficiently large so that there is no class $\xi$ of $(c_1,d)$ with $\<\xi, F\><0<\<\xi,  (G+sF)\>$.
Write $L:=nF+mG$.
By \corref{blowp} we get 
\begin{align*}
&\chi^{X,sF+G}_{F,d}(L,P^r)=\Coeff_{\Lambda^d}\Coeff_{q^0}\left[\Psi^{G+sF,F}_{X,F}(L;\Lambda,\tau)M^r\right]\\
\quad&=\Coeff_{\Lambda^d}\Coeff_{q^0}\left[\Theta^{G+sF,F}_{\Gamma_X,F,K_X}(\frac{1}{2\pi i}(L-K_X)h,\tau)\Lambda^2\WT_4(h)^{(L-K_X)^2}u'h^*M^r\right]\\
\quad&=
\Coeff_{\Lambda^d}\Coeff_{q^0}\left[\frac{e^{-\<\frac{F}{2},(L-K_X)\>}}{1-e^{-\<F,(L-K_X)\>h}}\Lambda^2\WT_4(h)^{(L-K_X)^2}u'h^*M^r\right]+\sum_{\<F,\xi\><0<\<(G+sF),\xi\>}\delta^X_\xi(L,P^r).
\end{align*}
Here the second sum is over the classes of type $(F,d)$. By our assumption on $s$ the second sum is empty, so we get
\begin{align*}
\chi^{X,F_+}_F(L,P^r)&=\Coeff_{q^0}\left[\frac{e^{-\<\frac{F}{2},(L-K_X)\>h}}{1-e^{-\<F,(L-K_X)\>h}}\Lambda^2\WT_4(h)^{(L-K_X)^2}u'h^*M^r\right]\\&=
\Coeff_{q^0}\left[\frac{\Lambda^2\WT_4(h)^{(L-K_X)^2}u'h^*}{2\sinh(\<\frac{F}{2},(L-K_X)\>h)}M^r\right].\end{align*}
In the case $c_1=0$ the argument is very similar. By definition and \thmref{vanbound} we have
\begin{align*}
&\chi^{X,sF+G}_{0,d}(L,P^r)=\Coeff_{\Lambda^d}\Coeff_{q^0}\left[\Psi^{sF+G,F}_{X,0}(L;\Lambda,\tau)M^r\right]\\
\quad&=\Coeff_{\Lambda^d}\Coeff_{q^0}\left[\Theta^{sF+G,F}_{\Gamma_X,0,K_X}(\frac{1}{2\pi i}(L-K_X)h,\tau)\Lambda^2\WT_4(h)^{(L-K_X)^2}u'h^*M^r\right]\\
\quad&=
-\Coeff_{\Lambda^d}\Coeff_{q^0}\left[\frac{e^{-\<F,L-K_X\>h}\Lambda^2\WT_4(h)^{(L-K_X)^2}u'h^*}{1-e^{-\<F,(L-K_X)\>h}}M^r\right]+\sum_{\<F,\xi\><0<\<(G+sF),\xi\>}\delta^X_\xi(L,P^r).
\end{align*}
The second sum is again over the walls of type $(0,d)$, and thus it is $0$.
Thus we get 
\begin{align*}
\chi^{X,F_+}_{0}(L,P^r)&=-\Coeff_{q^0}\left[\frac{e^{-\<F,L-K_X\>h}\Lambda^2\WT_4(h)^{(L-K_X)^2}u'h^*}{1-e^{-\<F,(L-K_X)\>h}}M^r\right]\\
&=
-\Coeff_{q^0}\left[\frac{1}{2}\left(\coth(\<F,(L-K_X)/2\>h)-1\right)\Lambda^2\WT_4(h)^{(L-K_X)^2}u'h^*M^r\right].
\end{align*}
Note that by \remref{delb}, we get $$\Coeff_{q^0}[\Lambda^2\WT_4(h)^{(L-K_X)^2}u'h^*M^r]=\Coeff_{q^0}[(1-1)\Lambda^2\WT_4(h)^{(L-K_X)^2}u'h^*M^r]=0.$$
\end{proof}

\begin{Remark}\label{Gplus}
In the case of $\P^1\times\P^1$, we can in the same way define $\chi^{\P^1\times\P^1,G_+}_{c_1,d}(L,P^r):=\chi^{\P^1\times\P^1,G+n_dF}_{c_1,d}(L,P^r)$ for $n_d$ sufficiently large with respect to $d$, and 
$$\chi^{\P^1\times\P^1,G_+}_{c_1}(nF+mG,P^r):=\sum_{d>0} \chi^{\P^1\times\P^1,G_+}_{c_1,d}(L,P^r)\Lambda^d.$$
Then we see immediately that 
$\chi^{\P^1\times\P^1,G_+}_{F}(nF+mG,P^r)=0$, and we get by symmetry from \propref{Fplus} that 
$$\chi^{\P^1\times\P^1,G_+}_{0}(nF+mG,P^r)=-\Coeff_{q^0}\left[\frac{1}{2}\left(\coth((n/2+1)h)\right)\Lambda^2\WT_4(h)^{2(n+2)(m+2)}u'h^*M^r\right].$$
\end{Remark}

\subsection{Recursion formulas from theta constant identities}

We now   use the blowup polynomials to show recursion formulas in $n$ and $r$ for the $K$-theoretical Donaldson invariants $\chi^{X,F_+}_{0}(nF+mG,P^r)$, $\chi^{X,F_+}_{F}(nF+mG,P^r)$ for $0\le m\le 2$. 
We use the fact that the $\widetilde S_n$ vanish at  division points $a\in \frac{2}{n}\pi i \Z$, together with other vanishing results proven in \cite{GY}.
We consider expressions relating the left hand sides of the formulas of \propref{Fplus} for $\chi^{X,F_+}_{0}(nF+mG,P^r)$, $\chi^{X,F_+}_{F}(nF+mG,P^r)$ for successive values of $n$.
We will show that these
are almost holomorphic in $q$, i.e. that they have only finitely many monomials $\Lambda^d q^s$ with nonzero coefficients and $s\le 0$.
This will then give recursion formulas for  $\chi^{X,F_+}_{0}(nF+mG,P^r)$, $\chi^{X,F_+}_{F}(nF+mG,P^r)$.


We will frequently use the following
\begin{Notation}
\begin{enumerate}
\item
For a power series $f=\sum_{n\ge 0} f_n(y)q^n \in \C[y^{\pm 1}][[q]]$, and a polynomial $g\in \C[y^{\pm 1}]$ we say that $g$ {\it divides} $f$, if $g$ divides $f_n$ for all $n$.
\item 
For a Laurent series $h=\sum_{n} a_nq^n\in \C((q))$ the {\it principal part} is ${\mathcal P}[h]:=\sum_{n\le 0} a_nq^n$. Note that this contains the coefficient of $q^0$. This is because
we think of $q$ as $e^{\pi i \tau/4}$, with $\tau$ in $\H$, and then $\frac{dq}{q}=\frac{\pi i}{4} d\tau$.
For a  series $h=\sum_{n\ge 0} h_n(q)\Lambda^n\in \C((q))[[\Lambda]]$,  the {\it principal part} is   ${\mathcal P}[h]:=\sum_{n\ge 0} {\mathcal P}[h_n]\Lambda^n\in \C((q))[[\Lambda]]$.
We recall the previous notation $\Coeff_{q^0}[h]:=\sum_{n\ge 0} \Coeff_{q^0}[h_n]\Lambda^n$.
\item We write $\Q[[y^{\pm 2} q^4,q^4]]^\times$ for the set power series in $y^{\pm 2} q^4,q^4$ whose constant part is $1$.
\end{enumerate}
\end{Notation}

\begin{Remark} By \eqref{thetatilde},\eqref{MuL} we have 
${\mathcal P}[M^2]=4-q^{-2}\Lambda^2+4\Lambda^4$, and thus obviously
\begin{align*}{\mathcal P}[M^2-(1-\Lambda^4)^2]&= 3 -q^{-2}\Lambda^2+6\Lambda^4 -\Lambda^8,\\
{\mathcal P}[M^2(1+\Lambda^4)-2(1-\Lambda^4)^2]&=2-q^{-2}\Lambda^2+12\Lambda^4-q^{-2}\Lambda^6+2\Lambda^8.
\end{align*}
\end{Remark}

\begin{Lemma}\label{theth1} For all $r\in \Z_{>0}$ we have
\begin{align*}
\tag{1} &g^r_1:={\mathcal P}\Big[\frac{1}{2\sinh(h)} M^{2r}u'h^*\Lambda^2\Big]\in \Q[q^{-2}\Lambda^2,\Lambda^4]_{\le r},\\
\tag{2}&g^r_2:={\mathcal P}\Big[-\frac{1}{2}\coth(h)M^{2r}u'h^*\Lambda^2\Big]\in \Q[q^{-2}\Lambda^2,\Lambda^4]_{\le r+1},\\
\tag{3}&{\mathcal P}\Big[\frac{1}{2\sinh(3h/2)}M\big(\WT_4(h)^3(1-\Lambda^4)-1\big)u'h^*\Lambda^2\Big]=\Lambda^4,\\
\tag{4}& g^r_3:={\mathcal P}\Big[\frac{1}{2\sinh(3h/2)}M^{2r-1}(M^2\WT_4(h)^3-(1-\Lambda^4))u'h^*\Lambda^2\Big]\in \Q[q^{-2}\Lambda^2,\Lambda^4]_{\le r},\\
\tag{5}&g^r_4:={\mathcal P}\Big[-\frac{1}{2}\coth(3h/2)M^{2r-2}(M^2\WT_4(h)^3-(1-\Lambda^4)) u'h^*\Lambda^2\Big]\in \Q[q^{-2}\Lambda^2,\Lambda^4]_{\le r+1},\\
\tag{6}&g^r_5:={\mathcal P}\Big[-\frac{1}{2}\tanh(h)M^{2r-2}\big(\WT_4(h)^8(M^2-(1-\Lambda^4)^2)-1\big) u'h^*\Lambda^2\Big] \in \Q[q^{-2}\Lambda^2,\Lambda^4]_{\le r+2}.
\end{align*}
\end{Lemma}
\begin{pf}
(1) We know $\WT_4(h)\in \Q[[y^{\pm 2}q^4,q^4]]^\times$, $\WT_1(h)\in iq(y-y^{-1})\Q[[y^{\pm 2}q^4,q^4]]^\times$,
and from the product formula of \eqref{theta} we see that even $\WT_1(2h)\in iq(y^2-y^{-2})\Q[[y^{\pm 2}q^4,q^4]]^\times$.
By \defref{blowpol}, \propref{rnpropo} we have 
$\frac{\WT_1(2h)}{\WT_4(h)^4}=\Lambda M$, thus we get that 
$\Lambda^{2} M^{2}\in q^{2}(y^2-y^{-2})^{2}\Q[[y^{\pm 2}q^4,q^4]].$
As $u'\in q^{-2}\Q[[q^4]]$, we get that 
$$f(y,q):=\sum_{n\ge 0} f_n(y)q^{4n}:=\frac{1}{\sinh(h)}\Lambda^{2} M^{2}u'\in (y^2-y^{-2})\Q[[y^{\pm 2}q^4,q^4]].$$
Thus  $f_n(y)$ is a Laurent polynomial in $y^2$ of degree at most $n+1$, and we see from the definitions that is it  antisymmetric under $y\to y^{-1}$.
Therefore $f_n(y)$ can be written as a linear combination of $\sinh(lh)$ for $l=1,\ldots, n+1$.
Thus we get by \lemref{qpow} that $f_n(y)h^*\in \Q[q^{-2}\Lambda^2]_{\le n+1}\Q[[q^2\Lambda^2,q^4]]$, and thus  the principal part of 
$q^{4n} f_n(y)h^*$ vanishes unless $4n\le 2n+2$, i.e. $n\le 1$.
Therefore the principal part of $f(y,q)h^*$ is a polynomial in $q^{-2}\Lambda^2,\Lambda^2q^2$, and $q^4$ and thus (as the power of $q$ must be nonpositive) 
a polynomial in $q^{-2}\Lambda^2$ and $\Lambda^4$, and we see that its degree is at most  $1$.

By  \eqref{MuL}, we have that $M^2=4+4u\Lambda^2+4\Lambda^4$. Using that $u\in q^{-2}\Q[[q^4]]$ we get that
$M^2\in \Q[q^{-2}\Lambda^2]_{\le 1}\Q[[q^2\Lambda^2,q^4]]$. Therefore by the above 
$$M^{2r-2}f_n(y,q)h^*\in \Q[q^{-2}\Lambda^2]_{\le n+r}\Q[[q^2\Lambda^2,q^4]].$$
The same argument as above shows that the principal part of $M^{2r-2}f(y,q)h^*$ is a polynomial in $q^{-2}\Lambda^2$ and $\Lambda^4$ of degree at most $r$.

(2) In (1) we have seen that $$M^{2r-2}f_n(y,q)h^*\in \Q[q^{-2}\Lambda^2]_{\le n+r}\Q[[q^2\Lambda^2,q^4]].$$
We have $\coth(h)\Lambda^{2} M^{2r}u'h^*=\cosh(h)M^{2r-2}f(y,q)h^*$, and by
\lemref{qpow} we have that $$\cosh(h)M^{2r-2}f_n(y,q)h^*\in \Q[q^{-2}\Lambda^2]_{\le n+r+1}\Q[[q^2\Lambda^2,q^4]]$$
The same argument as in (1) shows that the principal part of $\cosh(h)M^{2r-2}f(y,q)h^*$ is a polynomial in $q^{-2}\Lambda^2$ and $\Lambda^4$ of degree at most $r+1$.

(3) By \cite{GY}, Prop.~5.10(5) and its proof, we have that $\WT_4(h)^3(1-\Lambda^4)-1\in \Q[y^{\pm 1}][[q]]$ is divisible by $y^3-y^{-3}$. 
Thus also $M(\WT_4(h)^3(1-\Lambda^4)-1)\in \Q[y^{\pm 2}][[q]]$ is divisible by $y^3-y^{-3}$. 
We note that $\Lambda\in iq(y-y^{-1})\Q[[y^{\pm 2}q^4,q^4]]$, thus $1-\Lambda^4\in \Q[y^{\pm 2}]_{\le 1}\Q[[y^{\pm 2}q^4,q^4]]$. 
We already know $\WT_4(h)\in \Q[[y^{\pm 2}q^4,q^4]]$,
$M\in (y+y^{-1})\Q[[y^{\pm 2}q^4,q^4]]$. Thus $M(\WT_4^3(1-\Lambda^4)-1)\in (y^3-y^{-3})\Q[[y^{\pm 2}q^4,q^4]]$.
Therefore, writing 
$$f:=\sum_{n\ge 0} f_n(y) q^{4n}:=\frac{1}{2\sinh(3h/2)}M(\WT_4(h)^3(1-\Lambda^4)-1),$$
$f_n(y)$ is a Laurent polynomial in $y^2$ of degree at most $n$, and we see from the definitions that it is antisymmetric under $y\to y^{-1}$.
Thus by \lemref{qpow}  we get $f_n(y)h^*\in\Q[q^{-2}\Lambda^2]_{\le n}\Q[[q^2\Lambda^2, q^4]]$. Therefore $fh^*\in \Q[[q^2\Lambda^2, q^4]]$ and $f h^*u'\Lambda^2\in 
[q^{-2}\Lambda^2,\Lambda^4]_{\le 1}\Q[[q^2\Lambda^2, q^4]]$. Computation of the first few coefficients gives ${\mathcal P}[f u'\Lambda^2]=\Lambda^4$.

(4) As $\WT_4(h)^3(1-\Lambda^4)-1\in \Q[y^{\pm 1}][[q]]$ is divisible by $y^3-y^{-3}$, the same is  true for $\Lambda M^2(\WT_4^3(1-\Lambda^4)-1)$.
On the other hand $\Lambda(M^2-(1-\Lambda^4)^2)\in  i\Q[y^{\pm 1}][[q]]$,  and by $\Lambda(M^2-(1-\Lambda^4)^2)=\widetilde S_3=\frac{\WT_1(3h)}{\WT_4(h)^9}$
(see \defref{blowpol}, \propref{rnpropo}), we see that
$\Lambda(M^2-(1-\Lambda^4)^2)$, vanishes for $3h\in 2\pi i \Z$, i.e. when $y^3=y^{-3}$. Thus it is also divisible by $y^3-y^{-3}$. Therefore also 
$$\Lambda (M^2\WT_4(h)^3(1-\Lambda^4)-(1-\Lambda^4)^2)=\Lambda M^2(\WT_4^3(1-\Lambda^4)-1)+\Lambda(M^2-(1-\Lambda^4)^2)\in i\Q[y^{\pm 1}][[q]]$$ 
is divisible $y^3-y^{-3}$. We note that $(1-\Lambda^4)=\widetilde R_2=\frac{\WT_4(2h)}{\WT_4(h)^4}\in \Q[y^{\pm1}][[q]]$ does not vanish at
any $h$ with $3h\in 2\pi i \Z$. It follows that also the power series $\Lambda (M^2\WT_4(h)^3-(1-\Lambda^4))$ is divisible by $y^3-y^{-3}$.
Finally we note that $\Lambda\in iq(y-y^{-1})\Q[[y^{\pm 2}q^4,q^4]]$, thus $1-\Lambda^4\in \Q[y^{\pm 2}]_1\Q[[y^{\pm 2}q^4,q^4]]$. 
Thus $$\Lambda (M^2\WT_4(h)^3-(1-\Lambda^4))\in iq(y-y^{-1})\Q[y^{\pm 2}]_{\le 1}\Q[[y^{\pm2}q^4,q^4]],$$ and therefore, as it is divisible by $y^3-y^{-3}$, we can write
$\frac{1}{\sinh(3h/2)}M \Lambda (M^2\WT_4(h)^3-(1-\Lambda^4))=q \sum_{n\ge 0} \overline f_n(y)q^{4n}$, where $\overline f_n(y)$ is an odd Laurent polynomial in $y$ of degree $2n+1$, symmetric under $y\to y^{-1}$. Thus by \lemref{qpow}  we get $\overline f_n(y)h^*\in \Lambda\Q[q^{-2}\Lambda^2]_{\le n}\Q[[q^2\Lambda^2, q^4]]$, and thus
$\overline f_n(y)h^*u'\Lambda\in \Q[q^{-2}\Lambda^2]_{\le n+1}\Q[[q^2\Lambda^2, q^4]]$.
It follows as before that the principal part of 
$\frac{1}{2\sinh(3h/2)}M (M^2\WT_4(h)^3-(1-\Lambda^4))u'h^*\Lambda^2$ is a polynomial in $q^{-2}\Lambda^2$ and $\Lambda^4$ of degree at most $1$.
Using the fact that $M^2\in \Q[q^{-2}\Lambda^2]_{\le 1}\Q[[q^2\Lambda^2,q^4]]$ in the same way as in the proof of (1) we see that
 the principal part  of $\frac{1}{2\sinh(3h/2)} M^{2r-1}(M^2\WT_4(h)^3-(1-\Lambda^4))u'\Lambda^2h^*$ is a polynomial of degree at most $r$ in  $q^{-2}\Lambda^2$ and $\Lambda^4$.

 (5) We see that the left hand side of (5) is obtained from the left hand side of (4) by multiplying by $\cosh(3h/2)/M$. 
 As by the above $M\in (y+y^{-1})\Q[[y^{\pm 2}q^4,q^4]]^\times$, we see that 
 $$2\cosh(3h/2)/M=(y^3+y^{-3})/M\in (y^2-1+y^{-2})\Q[[y^{\pm 2}q^4,q^4]] \subset
 \Q[q^{-2}\Lambda^2]_{\le 1}\Q[[q^2\Lambda^2,q^4]],$$ 
 where the inclusion on the right follows again by \lemref{qpow}.  Therefore (5) follows from (4).
 
 (6) We note that by \defref{blowpol} and \propref{rnpropo} we have
 $$
 s_1:=(1+\Lambda^4)M^2-2(1-\Lambda^4)^2=\frac{S_4(\Lambda,M)}{S_2(\Lambda,M)R_2(\Lambda,M)}=\frac{\WT_1(4h)}{\WT_1(2h) \WT_4(2h)\WT_4(h)^8}.$$
Again  $s_1$ is in $\Q[y^{\pm 1}][[q]]$. As $\WT_4(h)$ has no zeros on $2\pi i \R$ and $\WT_1(h)$ vanishes precisely for $h\in 2\pi i \Z$, we find that $s_1$ vanishes
if $y^4=y^{-4}$, but not $y^2=y^{-2}$. Thus the coefficient of every power of $q$ of $s_1$ is divisible by $y^2+y^{-2}$.

In \cite{GY} Proposition 5.10(6) and its proof it is shown that  
$$s_2:=\WT_4(h)^8(1-\Lambda^4)^3-(1+\Lambda^4)\in (y^2+y^{-2})\Q[y^{\pm 1}][[q]].$$
Thus also 
$$M^2\WT_4(h)^8(1-\Lambda^4)^3-2(1-\Lambda^4)^2=M^2s_2+s_1\in (y^2+y^{-2})\Q[y^{\pm 1}][[q]].$$
As $\widetilde R_2=(1-\Lambda^4)\in \Q[y^{\pm 1}][[q]]^\times$ does not vanish for $h\in i\R$, we get that 
$s_3:=M^2\WT_4(h)^8(1-\Lambda^4)-2\in(y^2+y^{-2}) \Q[y^{\pm 1}][[q]].$ 
Therefore also 
$$\frac{1}{2}(s_3+\WT_4(h)^8s_1)=M^2\WT_4(h)^8-\WT_4(h)^8(1-\Lambda^4)^2-1 \in (y^2+y^{-2})\Q[y^{\pm 1}][[q]].$$
On the other hand we know 
$M^2\in (y+y^{-1})^2 \Q[[y^{\pm 2}q^4,q^4]]$, $\WT_4(h)\in \Q[[y^{\pm 2}q^4,q^4]]$ and $(1-\Lambda^4)^2\in \Q[y^{\pm 2}]_{\le 2}\Q[[y^{\pm 2}q^4,q^4]]$.
Thus 
$$l:=\tanh(h)\big(M^2\WT_4(h)^8-\WT_4(h)^8(1-\Lambda^4)^2-1\big)\in \Q[y^{\pm 2}]_{\le 2}\Q[[y^{\pm 2}q^4,q^4]].$$
Thus we can write $l=\sum_{n\ge 0} l_n(y)q^{4n}$ where $l_n(y)$ is a Laurent polynomial in $y^2$ of degree $n+2$, symmetric under $y\to y^{-1}$. Thus by \lemref{qpow}  we get $ l_n(y)h^*\in \Q[q^{-2}\Lambda^2]_{\le n+2}\Q[[q^2\Lambda^2, q^4]]$, and thus
$l_n(y)h^*u'\Lambda^2\in \Q[q^{-2}\Lambda^2]_{\le n+3}\Q[[q^2\Lambda^2, q^4]]$.
It follows as before that the principal part of 
$\tanh(h)\big(M^2\WT_4(h)^8-\WT_4(h)^8(1-\Lambda^4)^2-1\big)h^*u'\Lambda^2$ is a polynomial in $q^{-2}\Lambda^2$ and $\Lambda^4$ of degree at most $3$.
Using again the fact that $M^2\in \Q[q^{-2}\Lambda^2]_{\le 1}\Q[[q^2\Lambda^2,q^4]]$, we see that
 the principal part  of $\tanh(h)M^{2r-2}\big(M^2\WT_4(h)^8-\WT_4(h)^8(1-\Lambda^4)^2-1\big)h^*u'\Lambda^2$ is a polynomial of degree at most $r+2$ in  $q^{-2}\Lambda^2$ and $\Lambda^4$.
\end{pf}

\begin{Remark}
The principal parts above can be easily computed by calculations with the lower order terms with the power series, using the formulas given in \secref{thetamod}.
We see for instance:
\begin{align*}
g_1^1&=q^{-2}\Lambda^2-4\Lambda^4,\quad
g_2^1=-q^{-2}\Lambda^2 +\Big(\frac{1}{2}q^{-4} - 8\Big) \Lambda^4-q^{-2}\Lambda^6-\Lambda^8,\\
g_3^1&=q^{-2}\Lambda^2 -5\Lambda^4,
\quad g_3^2=4q^{-2}\Lambda^2 - (q^{-4} + 20)\Lambda^4 
+ 9q^{-2}\Lambda^6 -23\Lambda^8,\\
g_4^1&=-\frac{1}{2}q^{-2}\Lambda^2 + (\frac{1}{2}q^{-4}- 11 )\Lambda^4-\frac{1}{2}\Lambda^8,\\
g_5^1&=(\frac{1}{2}q^{-4} - 12)\Lambda^4 + 2q^{-2}\Lambda^6 + 4\Lambda^8 -\frac{1}{2}q^{-2} \Lambda^{10} +5\Lambda^{12}.
\end{align*}
\end{Remark}

We apply \lemref{theth1} to compute the limit of the $K$-theoretic Donaldson invariants at $F$.

\begin{Proposition}\label{p11r}  For $X=\P^1\times\P^1$ or $X=\widehat \P^2$, all $n\in \Z$ we have
\begin{enumerate}
\item
$\displaystyle{
1+\chi^{X,F_+}_{F}(nF)=\frac{1}{(1-\Lambda^4)^{n+1}}}.$
\item
For all $r>0$ there is a polynomial $h^0_{r}(n,\Lambda^4)\in \Q[\Lambda^4,n\Lambda^4]_{\le r}$ with 
$\chi^{X,F_+}_{F}(nF,P^{2r})=h^0_r(n,\Lambda^4).$
\item
$\displaystyle{1+(2n+5)\Lambda^4+\chi^{X,F_+}_{0}(nF)=\frac{1}{(1-\Lambda^4)^{n+1}}}$.
\item
For all $r>0$ there is a polynomial $\overline h^0_{r}(n,\Lambda^4)\in \Q[\Lambda^4,n\Lambda^4]_{\le r+1}$ with 
$\chi^{X,F_+}_{0}(nF,P^{2r})=\overline h^0_r(n,\Lambda^4).$
\end{enumerate}
\end{Proposition}
\begin{proof}
(1) and (3) are proven in \cite[Prop.~5.14]{GY}. 

(2) Let $r>0$. By \propref{Fplus} we have 
\begin{align*}
\chi^{X,F_+}_{F}(nF,P^{2r})&=\Coeff_{q^0}\Big[\frac{1}{2\sinh(h)}\Lambda^2\WT_4(h)^{4n+8}u'h^*M^{2r}\Big]=\Coeff_{q^0}\Big[g_{1}^r \WT_4(h)^{4n+8}\Big].
\end{align*}
where the last step uses \lemref{theth1}(1) and the fact that $\WT_4(h)\in \Q[[q^2\Lambda^2,q^4]]^\times$,
and thus  $\WT_4(h)^{4n+8}\in \Q[[nq^2\Lambda^2,nq^4,q^2\Lambda^2,q^4]]^\times$.
As $g_{1}^r(q^{-2}\Lambda^2,\Lambda^4)$ is a polynomial of degree at most $r$, we see that 
$\Coeff_{q^0}\big[g_{1}^r \WT_4(h)^{4n+8}\big]$ is a polynomial of degree at most $r$ in $\Lambda^4$, $n\Lambda^4$.

(4) Let $r>0$. By \propref{Fplus} and \lemref{theth1}(2) we have 
\begin{align*}
\chi^{X,F_+}_{F}(nF,P^{2r})&=\Coeff_{q^0}\Big[-\frac{1}{2}\coth(h)\Lambda^2\WT_4(h)^{4n+8}u'h^*M^{2r}\Big]=\Coeff_{q^0}\Big[g_{2}^r\WT_4(h)^{4n+8}\Big].
\end{align*}
As $g_{2}^r$ is a polynomial of degree at most $r+1$, we see as in (1) that 
$\Coeff_{q^0}\big[g_{2}^r\WT_4(h)^{4n+8}\big]$ is a polynomial of degree at most $r+1$ in $\Lambda^4$, $n\Lambda^4$.
\end{proof}

\begin{Remark}
We list the first few polynomials $h^0_r$, $\overline h^0_r$.
\begin{align*} &h^0_1=(4n + 4)\Lambda^4, \quad h^0_2=(16n + 16)\Lambda^4 - (8n^2 + 6n -3)\Lambda^8, \\ &h^0_3=(64n + 64)\Lambda^4 + (-64n^2 + 24n + 100)\Lambda^8 + (\hbox{$\frac{32}{3}$}n^3 - 8n^2 - \hbox{$\frac{68}{3}$}n)\Lambda^{12},\\
&\overline h^0_1=-(4n + 16)\Lambda^4 + (4n^2 + 15n + 13)\Lambda^8,\\ &\overline h^0_2=-(16n + 64)\Lambda^4 + (24n^2 + 78n + 18)\Lambda^8 - (\hbox{$\frac{16}{3}$}n^3 + 20n^2 + \hbox{$\frac{50}{3}$}n -2)\Lambda^{12}.
\end{align*}
\end{Remark}

\begin{Proposition}\label{p11GM}  For $X=\P^1\times\P^1$ and $n\in \Z$, and for $X=\widehat \P^2$ and $n\in \Z+\frac{1}{2}$, and all $r\in \Z_{\ge 0}$ we have
the following.
\begin{enumerate}
\item 
$\displaystyle{\chi^{X,F_+}_{F}(nF+G,P^{2r+1})=\frac{1}{(1-\Lambda^4)^{2n+1-2r}}+h^1_r(n,\Lambda^4)},$
where $h^1_r(n,\Lambda^4)\in\Q[\Lambda^4,n\Lambda^4]_{\le r}$.
\item
$\displaystyle{
\chi^{X,F+}_{0}(nF+G,P^{2r})=\frac{1}{(1-\Lambda^4)^{2n+2-2r}}+\overline h^1_r}(n,\Lambda^4),$
where $\overline h^1_r(n,\Lambda^4)\in\Q[\Lambda^4,n\Lambda^4]_{\le r+1}$.
\end{enumerate}
\end{Proposition}
\begin{pf}
(1) First we deal with the case $r=0$. We do this by ascending and descending induction on $n$. 
Let $n=-1$. By \corref{blowp} we know that $\chi^{X,G}_{F}(-F+G,P)=0$
and  $$\chi^{X,F+}_{F}(-F+G,P)=\sum_{\xi}\delta_\xi^X(-F+G,P),$$
 where $\xi$ runs through all classes  of type $F$  with $G\xi<0<F\xi$,
i.e. through all $\xi=(2mG-(2n-1)F))$ with $n,m\in \Z_{>0}$. By \lemref{vanwall} we have $\delta_{2mG-(2n-1)F}^X(-F+G,P)=0$ unless
$|6n-3-2m|+3\ge 8nm-4m$, and we check easily that this can only happen for $n=m=1$. Then computing with the lowest order terms of the formula of \defref{wallcrossterm} gives
$\delta_{2G-F}^X(-F+G,P)=-\Lambda^4$. Thus $\chi^{X,F_+}_{F}(-F+G,P)=-\Lambda^4=(1-\Lambda^4)-1$. This shows the case $n=-1$.

Now let $n\in \frac{1}{2}\Z$ be general, then we have by \propref{Fplus}  that
\begin{align*}(1-\Lambda^4)\chi^{X,F_+}_{F}&((n+1/2)F+G,P)-\chi^{X,F_+}_{F}(nF+G,P)\\ 
&=\Coeff_{q^0}\left[\frac{1}{2\sinh(3h/2)}\WT_4(h)^{6(n+2)}\big(\WT_4(h)^3(1-\Lambda^4)-1\big)Mu'h^*\Lambda^2\right]
\\&=\Coeff_{q^0}\big[\WT_4(h)^{6(n+2)}\Lambda^4\big]=
\Lambda^4,\end{align*}
and where in the last line we have used \lemref{theth1}(3) and the fact that $\WT_4(h)\in \Q[[q^2\Lambda^2,q^4]]^\times$.
Thus
$$(1-\Lambda^4)\big(1+\chi^{X,F_+}_{F}((n+1/2)F+G,P)\big)-(1+\chi^{X,F_+}_{F}(nF+G,P))=(1-\Lambda^4)-1+\Lambda^4=0$$
 and using the result for $n=-1$, $r=0$, the result for $r=0$ follows by ascending and descending induction over $n\in \frac{1}{2}\Z$.

Let $r>0$, $n\in \frac{1}{2}\Z$. By \propref{Fplus} we have
\begin{align*}\chi^{X,F_+}_{F}\big(nF+G,&P^{2r+1}\big)-(1-\Lambda^4) \chi^{X,F_+}_{F}\big((n-1/2)F+G,P^{2r-1}\big)\\
&=\Coeff_{q^0}\Big[\frac{1}{2\sinh(3h/2)}\WT_4(h)^{6n+9}M^{2r-1}\big(M^2\WT_4(h)^3-(1-\Lambda^4)\big)u'h^*\Lambda^2\Big]
\\&= \Coeff_{q^0}\big[\WT_4(h)^{6n+9}g^r_3\big],\end{align*}
where the last line is by \lemref{theth1}(4).
As $\WT_4(h)^{6n+9}\in \Q[[nq^2\Lambda^2,nq^4,q^2\Lambda^2,q^4]]^\times$, and $g_3^r$ is a polynomial in $q^{-2}\Lambda^2$, $\Lambda^4$ of degree
$r$, we find, as in the proof of \propref{p11r},  that 
$$h'_r(n,\Lambda^4):=\chi^{X,F_+}_{F}(nF+G,P^{2r+1})-(1-\Lambda^4) \chi^{X,F_+}_{F}\big((n-1/2)F+G,P^{2r-1}\big)\in \Q[\Lambda^4,n\Lambda^4]_{\le r}.$$
Assume now by induction on $r$ that  
$$\chi^{X,F_+}_{F}\big((n-1/2)F+G,P^{2r-1}\big)=\frac{1}{1-\Lambda^4)^{2n-2r+2}}+h^1_{r-1}\big((n-1/2),\Lambda^4\big)$$
 with $h^1_{r-1}(n,\Lambda^4)\in \Q[\Lambda^4,n\Lambda^4]_{\le r-1}.$
Then 
\begin{align*}\chi^{X,F_+}_{F}(nF+G,P^{2r+1}\big)-\frac{1}{(1-\Lambda^4)^{2n-2r+1}}&=(1-\Lambda^4) h^1_{r-1}\big((n-1/2),\Lambda^4)+h'_r(n,\Lambda^4).
\end{align*}
Thus we put 
$$h^1_r(n,\Lambda^4):=(1-\Lambda^4) h^1_{r-1}\big((n-1/2),\Lambda^4)+h'_r(n,\Lambda^4).$$ As $h'_r(n,\Lambda^4)$ has degree at most $r$ in $\Lambda^4$, $n\Lambda^4$, the claim follows.

(2)  The case $r=0$  is proven in \cite[Prop~5.16]{GY}, with $\overline h^1_0(n,\Lambda^4)=-1-(3n+7)\Lambda^4.$
For $r>0$ we prove the result by induction. Let $r>0$, then
 we have by \propref{Fplus}  and \lemref{theth1}(5)
\begin{align*}\chi^{X,F_+}_{0}(nF+G,P^{2r}\big)&-(1-\Lambda^4)\chi^{X,F_+}_{0}\big((n-1/2)F+G,P^{2r-2}\big)\\&=\Coeff_{q^0}\Big[-\frac{1}{2}\coth(3h/2)\WT_4(h)^{6n+9}M^{2r-2}\big(M^2\WT_4(h)^3-(1-\Lambda^4)\big)u'h^*\Lambda^2\Big]
\\&= \Coeff_{q^0}\big[\WT_4(h)^{6n+9}g^{r}_4\big]=:l'_r(n,\Lambda^4)\in \Q[\Lambda^4,n\Lambda^4]_{\le r+1}.\end{align*}
Assume now that 
$$\chi^{X,F_+}_{0}\big((n-1/2)F+G,P^{2r-2}\big)=\frac{1}{1-\Lambda^4)^{2n-2r+3}}+\overline h^1_{r-1}(n-1/2,\Lambda^4),$$
with $\overline h^1_{r-1}(n-1/2,\Lambda^4)\in\Q[\Lambda^4,n\Lambda^4]_{\le r}$.
Then 
\begin{align*}\chi^{X,F_+}_{0}(nF+G,P^{2r})-\frac{1}{(1-\Lambda^4)^{2n-2r+2}}
&=(1-\Lambda^4) \overline h^1_{r-1}((n-1/2),\Lambda^4)+l'_r(n,\Lambda^4).
\end{align*}
The result follows by induction on $r$.
\end{pf}

\begin{Remark}
We list the $h^1_r(n,\Lambda^4),\ \overline h^1_r(n,\Lambda^4)$ for small values of $n$,
\begin{align*}
&h^1_0=-1,\quad h^1_1=-1+(6n + 5)\Lambda^4 \quad h^1_2=-1+(30n + 19)\Lambda^4 - (18n^2 + 15n - 2)\Lambda^8,\\ & h^1_3=-1+(126n + 69)\Lambda^4+ (-162n^2 + 9n + 114)\Lambda^8 + (36n^3 - 43n - 7)\Lambda^{12},\\
&\overline h^1_0=-1-(3n + 7)\Lambda^4, \quad \overline h^1_1=-1-(6n + 20)\Lambda^4 + (9n^2 +  \hbox{$\frac{69}{2}$}n + 32)\Lambda^8,\\&
\overline h^1_2=-1-(18n +78)\Lambda^4 + (54n^2 + 189n + 120)\Lambda^8 - (18n^3 + 81n^2 + 109n +40)\Lambda^{12}.
\end{align*}
\end{Remark}


\begin{Proposition}\label{p112GM}  
Let $X=\P^1\times\P^1$ or $X=\widehat \P^2$.
\begin{enumerate}
\item 
 For all $n\in \Z$
$$\chi^{X,F_+}_{F}(nF+2G)=
\frac{1}{2}\frac{(1+\Lambda^4)^n-(1-\Lambda^4)^n}{(1-\Lambda^4)^{3n+3}}.$$ 
\item  For all $n\in \Z$ and all $r>0$ we have 
$$\chi^{X,F_+}_{F}(nF+2G,P^{2r})=\frac{2^{r-1}(1+\Lambda^4)^{n-r}}{(1-\Lambda^4)^{3n+3-2r}}-h^2_r(n,\Lambda^4),$$
where $h^2_r(n,\Lambda^4)\in \Q[\Lambda^4,n\Lambda^4]_{\le 2r+2}$.
\item
$$\chi^{X,F_+}_{0}(nF+2G)=
\frac{1}{2}\frac{(1+\Lambda^4)^n+(1-\Lambda^4)^n}{(1-\Lambda^4)^{3n+3}}-1-(4n+9)\Lambda^4.$$ 
\item  For all $n\in \Z$ and all $r>0$ we have 
$$\chi^{X,F_+}_{0}(nF+2G,P^{2r})=\frac{2^{r-1}(1+\Lambda^4)^{n-r}}{(1-\Lambda^4)^{3n+3-2r}}-\overline h^2_r(n,\Lambda^4),$$
where $\overline h^2_r(n,\Lambda^4)\in \Q[\Lambda^4,n\Lambda^4]_{\le 2r+2}$.
\end{enumerate}
\end{Proposition}
\begin{pf}
(1) and (3) were proven in \cite[Prop.~5.17]{GY}.
(2) We will first show by induction on $r$ that 
\begin{equation}\label{p12req}
-\frac{1}{2}\Coeff_{q^0}\big[\tanh(h)\WT_4(h)^{8(n+2)}u'h^*\Lambda^2M^{2r}\big]=2^r\frac{(1+\Lambda^4)^{n-r}}{(1-\Lambda^4)^{3n+3-2r}}+s'_r(n,\Lambda^4).
\end{equation} For polynomials 
$s'_r(n,\Lambda^4)\in \Q[\Lambda^4,n\Lambda^4]_{\le 2r+2}$
For $r=0$ this is shown in the proof of \cite[Prop.~5.17]{GY} with $s_0'=-1-(4n+9)\Lambda^4.$

Fix $r>0$, assume  that \eqref{p12req} holds $r-1$ and for all  $n\in \Z$.
By \lemref{theth1}(6) we have
\begin{align*}
-\frac{1}{2}\Coeff_{q^0}\big[\tanh(h)&\big(M^{2r}\WT_4(h)^{8(n+2)}-(1-\Lambda^4)^2M^{2r-2}\WT_4(h)^{(8n+2)}-\WT_4(h)^{8(n+1)}M^{2r-2}\big)u'h^*\Lambda^2\big]\\
&=\Coeff_{q^0}\big[\WT_4(h)^{8(n+1)}g_5^{r}\big]=:s''_r(n,\Lambda^4).\end{align*}
Again, as $\WT_4(h)\in \Q[[\Lambda^2q^4,q^4]]^\times$, and $g_5^{r}$ has degree $r+2$ in $q^{-2}\Lambda^2,\Lambda^4$, we see that
$s''_r(n,\Lambda^4)\in \Q[\Lambda^4,n\Lambda^4]_{\le r+2}$. 
Thus we get by induction on $r$
\begin{align*}
-\frac{1}{2}\Coeff_{q^0}\big[&\tanh(h)M^{2r}\WT_4(h)^{8(n+2)}u'h^*\Lambda^2\big]= \frac{2^{r-1}(1+\Lambda^4)^{n-r+1}}{(1-\Lambda^4)^{3n+3-2r}}+(1-\Lambda^4)^2s'_{r-1}(n,\Lambda^4)
\\&+
\frac{2^{r-1}(1+\Lambda^4)^{n-r}}{(1-\Lambda^4)^{3n+2-2r}}+s'_{r-1}(n-1,\Lambda^4)+s''_r(n,\Lambda^4)=\frac{2^{r}(1+\Lambda^4)^{n-r}}{(1-\Lambda^4)^{3n+3-2r}}+s'_{r}(n,\Lambda^4)\end{align*}
with $$s'_{r}(n,\Lambda^4)=(1-\Lambda^4)^2s'_{r-1}(n,\Lambda^4)+s'_{r-1}(n-1,\Lambda^4)+s''_r(n,\Lambda^4).$$ As $s'_{r-1}\in \Q[\Lambda^4,n\Lambda^4]_{\le 2r}$, $s''_r\in \Q[\Lambda^4,n\Lambda^4]_{\le r+2},$ we get 
 $s'_{r}(n,\Lambda^4)\in \Q[\Lambda^4,n\Lambda^4]_{\le 2r+2}$.

Now we show (2): We note that 
$\frac{1}{2\sinh(2h)}=\frac{1}{4}\big(\coth(h)-\tanh(h)\big)$.
Therefore we get by \propref{Fplus}
\begin{align*}
\chi^{X,F_+}(nF+2G,P^{2r})&=\frac{1}{4}\Coeff_{q^0}\big[\big(\coth(h)-\tanh(h)\big)\WT_4(h)^{8(n+2)}M^{2r}u'h^*\Lambda^2\big]\\
=&-\frac{1}{2}\chi^{X,F_+}_0((2n+2) F,P^{2r})+\frac{1}{2}\Coeff_{q^0}\big[\big(-\frac{1}{2}\tanh(h)\big)\WT_4(h)^{8(n+2)}M^{2r}u'h^*\Lambda^2\big]\\
=&-\frac{1}{2}\overline h^0_r(2n+2,\Lambda^4)+\frac{2^{r-1}(1+\Lambda^4)^{n-r}}{(1-\Lambda^4)^{3n+3-2r}}+\frac{1}{2}s_r'(n,\Lambda^4).
\end{align*}
here in the last line we have used \propref{p11r} and \eqref{p12req}.
The claim follows with $h^2_r(n,\Lambda^4)=\frac{1}{2}\big(s_r'(n,\Lambda^4)-\overline h^0_r(2n+2, \Lambda^4)\big)\in \Q[\Lambda^4,n\Lambda^4]_{\le 2r+2}$.

Finally we show (4): 
$-\frac{1}{2}\coth(2h)=\frac{1}{4}\big(-\coth(h)-\tanh(h)\big)$ and 
 \propref{Fplus} give
\begin{align*}
\chi^{X,F_+}(nF+2G,P^{2r})&=\frac{1}{4}\Coeff_{q^0}\big[\big(-\coth(h)-\tanh(h)\big)\WT_4(h)^{8(n+2)}M^{2r}u'h^*\Lambda^2\big]\\
=&\frac{1}{2}\chi^{X,F_+}_0((2n+2)F,P^{2r})+\frac{1}{2}\Coeff_{q^0}\big[\big(-\frac{1}{2}\tanh(h)\big)\WT_4(h)^{8(n+2)}M^{2r}u'h^*\Lambda^2\big]\\
=&\frac{1}{2}\overline h^0_r(2n+2,\Lambda^4)+\frac{2^{r-1}(1+\Lambda^4)^{n-r}}{(1-\Lambda^4)^{3n+3-2r}}+\frac{1}{2}s_r'(n,\Lambda^4).
\end{align*}
The claim follows with $\overline h^2_r=\frac{1}{2}\big(s_r'(n,\Lambda^4)+\overline h^0_r(2n+2, \Lambda^4)\big)$.
\end{pf}

\begin{Remark} Again we can readily compute the first few of the $h^2_r$, $\overline h^2_r$.
\begin{align*} &h^2_1=-1,\quad h^2_2=-2 + (8n + 6)\Lambda^4,\quad  h^2_3=-4 + (48n + 24)\Lambda^4 - (32n^2 + 28n)\Lambda^8,\\ &h^2_4=-8 + (224n + 72)\Lambda^4 + (-320n^2 - 40n + 128)\Lambda^8 + (\hbox{$\frac{256}{3}$}n^3 + 32n^2 - \hbox{$\frac{184}{3}$}n - 16)\Lambda^{12},\\
&\overline h^2_1= -1  -(8n + 24)\Lambda^4 + (16n^2 + 62n + 59)\Lambda^8,\\
&\overline h^2_2=-2 -(24n +90)\Lambda^4 + (96n^2 + 348n + 270)\Lambda^8  -(\hbox{$\frac{128}{3}$}n^3 + 208n^2 + \hbox{$\frac{964}{3}$}n +154)\Lambda^{12}.
\end{align*}
It appears that  one has $h^2_r\in \Q[\Lambda^4,n\Lambda^4]_{\le r-1}$ and $\overline h^2_r\in \Q[\Lambda^4,n\Lambda^4]_{\le r}$.
\end{Remark}


\begin{Corollary}\label{P2P12GH} Let $X=\widehat \P^2$ or $X=\P^1\times \P^1$, let $\omega\in H^2(X,\R)$ be a class with $\<\omega^2\>>0$.
\begin{enumerate} 
\item For $n\in \Z_{\ge 0}$ we have
\begin{align*} 
\chi_{0}^{X,\omega}(nF)&\equiv \chi_{F}^{X,\omega}(nF)\equiv \frac{1}{(1-\Lambda^4)^{n+1}}, \quad
\chi_{0}^{X,\omega}(nF,P^{2r})\equiv \chi_{F}^{X,\omega}(nF,P^{2r})\equiv 0 \hbox{ for }r>0.
\end{align*}
\item For $n\in \Z_{\ge 0}$ if $X=\P^1\times \P^1$ and $n\in \Z_{\ge 0}+\frac{1}{2}$ if $X=\widehat \P^2$,  we have
\begin{align*} 
\chi_{0}^{X,\omega}(nF+G,P^{2r})&\equiv \frac{1}{(1-\Lambda^4)^{2n+2-2r}},\quad
\chi_{F}^{X,\omega}(nF+G,P^{2r+1})\equiv \frac{1}{(1-\Lambda^4)^{2n+1-2r}}.
\end{align*}
\item For $n\in \Z_{\ge 0}$ we have
\begin{align*} 
\chi_{0}^{X,\omega}(nF+2G)&\equiv \frac{(1+\Lambda^4)^{n}+(1-\Lambda^4)^n}{2(1-\Lambda^4)^{3n+3}},\quad
\chi_{F}^{X,\omega}(nF+2G)\equiv \frac{(1+\Lambda^4)^{n}-(1-\Lambda^4)^n}{2(1-\Lambda^4)^{3n+3}}\\
\chi_{0}^{X,\omega}(nF+2G,P^{2r})&\equiv \chi_{F}^{X,\omega}(nF+2G,P^{2r}) \equiv \frac{2^{r-1}(1+\Lambda^4)^{n-r}}{(1-\Lambda^4)^{3n+3-2r}}\hbox{ for  }1\le r \le n.
\end{align*}
\end{enumerate}
\end{Corollary}

\begin{proof}
Write $\omega=uF+vG$, with $u,v\in \R_{>0}$, write $w=v/u$.
By \propref{p11r},  \propref{p11GM}, \propref{p112GM} and  \lemref{vanwall} it is sufficient to prove that  for $L=nF+mG$, with $0\le m\le 2$, under the assumptions of the Corollary, there are only finitely many classes  $\xi$ of type $(0)$ or type $(F)$ with $\<\omega,\xi\>\ge 0>\<F,\xi\>$, such 
that $\delta_\xi^X(nF+mG,P^s)\ne 0$, i.e. such that 
$-\xi^2\le |\<\xi,(L-K_X)\>| +s+2$. 
These walls are of the form 
$\xi=aF-bG$ with $a\in \Z_{>0}$ and $b\in 2\Z_{>0}$, and $aw\ge b$, and  the condition becomes
\begin{equation}
\label{abwall}
2ab\le |b(n+2)-a(m+2)| +s+2.
\end{equation}

Let $\xi=(aF-bG)$ be such a wall with $\delta_\xi^X(nF+mG,P^s)\ne 0$.
If $a(m+2)\le b(n+2)$, then \eqref{abwall} becomes
$$2ab\le b(n+2)-a(m+2) +s+2 \le b(n+2)+s,$$ therefore $(2a-n-2)b\le s$. Therefore $2a-n-2\le \frac{s}{b}\le \frac{s}{2}$. Therefore $a$ is bounded, and by the condition $b\le aw$ also $b$ is bounded, so there 
are only finitely many possibilities for $a,b$.

Now assume $a(m+2)\ge b(n+2)$.
Then as $b\ge 2$, \eqref{abwall} gives
$$4a\le 2ab\le a(m+2)-2(n+2)+s+2,$$ i.e. $(2-m)a\le -2n+s-2$.
If $m=0,1$, then $a\le \frac{-2n+s-2}{2-m}$, thus $a$ is bounded, and by $a(m+2)\ge b(n+2)$ also $b$ is bounded.
If $m=2$, the inequality becomes $2n\le s-2$, so if $2n\ge s$ there are no walls with $\delta_\xi^X(nF+mG,P^s)\ne 0$. Thus  the claim follows.
\end{proof}

\begin{Remark}\label{nonwalls}
As we will use this later in \secref{CompP2},  we explicitly state the bounds obtained in the above proof in the case of $X=\widehat \P^2$, $\omega=H$
(i.e.  $w=2$ in the notation above). Fix $n\in \Z_{\ge 0}$, $s\in \Z_{\ge 0}$. 
Let $\xi=aF-bG$ be a class of type $(0)$ or $(F)$ with $\<\xi, \omega\>\ge 0>\<\xi,F\>$. 
\begin{enumerate}
\item If $\delta_\xi^X(nF-nE,P^s)=\delta_\xi^{X}(nF,P^s)\ne \emptyset$, then 
\begin{enumerate}
\item either $2a\le (n+2) b$ and $0<a\le\frac{n+2}{2}+\frac{s}{4}$ and $0<b\le 2a$,
\item or $0<(n+2)b\le 2a$ and $0<a\le \frac{s}{2}-n-1$.
\end{enumerate}
\item If $\delta_\xi^X(nF-(n-1)E,P^s)=\delta_\xi^{X}((n-1/2)F+G,P^s)\ne \emptyset$, then 
\begin{enumerate}
\item either $3a\le (n+\frac{3}{2}) b$ and $)<a\le \frac{n+3/2}{2}+\frac{s}{4}$ and $0<b\le 2a$,
\item or $0<(n+3/2)b\le 3a$ and $0<a\le s-2n-2$.
\end{enumerate}
\end{enumerate}
\end{Remark}

\begin{Remark}
Note that the results of \corref{P2P12GH} are compatible with \conref{ratconj}.
This is particularly remarkable for part (3) of  \corref{P2P12GH}, which can only be proven for  $r\le n$, while its correctness for $r>n$ would contradict \conref{ratconj}.

The fact that the formulas hold without restriction for $\chi^{X,F_+}_{0}(nF+2G,P^{2r})$,  $\chi^{X,F_+}_{F}(nF+2G,P^{2r})$ is not in contradiction to \conref{ratconj},
because it is only claimed for $\chi^{Y,\omega}_{c_1}(L,P^r)$ with $\omega$ an  ample class on $Y$. 
\end{Remark}

\section{Computation of the invariants of the plane}

We now want to use the results obtained so far to give an algorithm to compute the generating functions  
$\chi^{\P^2,H}_{0}(nH,P^r)$, $\chi^{\P^2,H}_{H}(nH,P^r)$ of the $K$-theoretic Donaldson invariants of the projective plane. We use this algorithm to prove that these generating functions
are always rational functions of a very special kind.
Then we will use this algorithm to explicitly compute $\chi^{\P^2,H}_{0}(nH,P^r)$, $\chi^{\P^2,H}_{H}(nH,P^r)$ for not too large values of $n$ and $r$.
First we explicitly carry out the algorithm by hand when  $r=0$ and $n\le 5$ in an elementary but tedious computation.
Finally we implemented the algorithm as a PARI program, which in principle can prove a formula for  $\chi^{\P^2,H}_{0}(nH,P^r)$, $\chi^{\P^2,H}_{H}(nH,P^r)$ for any $n$ and $r$.
The computations have been carried out for $r=0$ and $n\le 11$ and $n\le 8$ and $r\le 16$. 

\subsection{The  strategy}\label{strategy}
 \corref{blowdownmn} says in particular the following.
\begin{Remark}\label{npol}
\begin{enumerate} 
\item For all $n\in \Z_{>0}$ there exist unique polynomials $f_{n},g_n\in \Q[x,\lambda^4]$ and an integer $N_n$, such that
$f_n S_n+g_n S_{n+1}=\lambda(1-\lambda^4)^{N_n}$.
\item  For all $n\in \Z_{>0}$ there exist unique polynomials $ h_{n}, l_n\in \Q[x^2,\lambda^4]$ and an integer $M_n$, such that
 $h_n  R_n+  l_n R_{n+1}=(1-\lambda^4)^{M_n}$.
\end{enumerate}
\end{Remark}

Using these polynomials, we can determine the $K$-theoretic Donaldson invariants of $\P^2$ in terms of those of $\widehat \P^2$.
\begin{Corollary}\label{blowdownform}
For all $n,k\in \Z$, $r\in \Z_{\ge 0}$ we have 
\begin{align*}
\tag{H}\chi^{\P^2,H}_H(nH,P^r)&=\frac{1}{\Lambda(1-\Lambda^4)^{N_k}}\Big(\chi^{\widehat \P^2,H}_{F}\big((n+1-k)G+\frac{n+k-1}{2}F,P^r\cdot f_k(P,\Lambda)\Big)\\&\qquad\qquad +
\chi^{\widehat \P^2,H}_{F}\Big((n-k)G+\frac{n+k}{2}F,P^r\cdot g_k(P,\Lambda)\Big)\Big),\\
\tag{0}\chi^{\P^2,H}_0(nH,P^r)&=\frac{1}{(1-\Lambda^4)^{M_k}}\Big(\chi^{\widehat \P^2,H}_{0}\Big((n+1-k)G+\frac{n+k-1}{2}F,P^r\cdot h_k(P,\Lambda)\Big)\\&\qquad\qquad +
\chi^{\widehat \P^2,H}_{0}\Big((n-k)G+\frac{n+k}{2}F,P^r\cdot l_k(P,\Lambda)\Big)\Big).
\end{align*}
\end{Corollary}
\begin{proof}
Note that $(n-k)G+\frac{n+k}{2}F=nH-kE$. Therefore we get by 
\thmref{nblow} that
\begin{align*}
\chi^{\widehat \P^2,H}_{F}&\Big((n+1-k)G+\frac{n+k-1}{2}F,P^r\cdot f_k(P,\Lambda)\Big)
+
\chi^{\widehat \P^2,H}_{F}\Big((n-k)G+\frac{n+k}{2}F,P^r\cdot g_k(P,\Lambda)\Big)\\
&=\chi^{\P^2,H}_H\big(nH,P^r \cdot\big(f_k(P,\Lambda)  S_{k}(P,\Lambda)+g_k(P,\Lambda)   S_{k+1}(P,\Lambda)\big)\big),
\end{align*}
and the result follows by $f_k(P,\Lambda)  S_{k}(P,\Lambda)+g_k(P,\Lambda)   S_{k+1}(P,\Lambda)=\Lambda(1-\Lambda^4)^{N_k}$.
In the same way 
\begin{align*}\chi^{\widehat \P^2,H}_{0}\Big((n+1-k)G&+\frac{n+k-1}{2}F,P^r\cdot h_k(P,\Lambda)\Big)
+
\chi^{\widehat \P^2,H}_{0}\Big((n-k)G+\frac{n+k}{2}F,P^r\cdot l_k(P,\Lambda)\Big)\\
&=\chi^{\P^2,H}_0\big(nH,P^r \cdot\big(h_k(P,\Lambda)  R_{k}(P,\Lambda)+l_k(P,\Lambda) R_{k+1}(P,\Lambda)\big)\big),
\end{align*}
and $h_k(P,\Lambda)   R_{k}(P,\Lambda)+l_k(P,\Lambda)   R_{k+1}(P,\Lambda)=(1-\Lambda^4)^{M_k}$.
\end{proof}

Using \corref{P2P12GH} we can use this in two different ways to compute the $K$-theoretic Donaldson invariants of $\P^2$.
\begin{enumerate}
\item We apply parts (1) and (2) of \corref{P2P12GH} to compute the $\chi_0^{\widehat \P^2,H}(nF,P^s)$, $\chi_F^{\widehat \P^2,H}(nF,P^s)$,
$\chi_0^{\widehat \P^2,H}(G+(n-\frac{1}{2})F,P^s)$, $\chi_F^{\widehat \P^2,H}(G+(n-\frac{1}{2})F,P^s)$, and then apply  \corref{blowdownform} with $k=n$.
Parts (1) and (2) of  \corref{P2P12GH} apply for all values of $n$ and $s$, so this method can always be used.
We will apply this \secref{P2rat} to prove the rationality of the generating functions of the  $K$-theoretic Donaldson invariants of $\P^2$ and of blowups of $\P^2$, and then in  \secref{CompP2} to 
compute the $K$-theoretic Donaldson invariants of $\P^2$ using a PARI program.
\item
We apply parts (2) and (3) of \corref{P2P12GH} to compute the
$\chi_0^{\widehat \P^2,H}(G+(n-\frac{1}{2})F,P^s)$, $\chi_F^{\widehat \P^2,H}(G+(n-\frac{1}{2})F,P^s)$, $\chi_0^{\widehat \P^2,H}(2G+(n-1)F,P^s)$, $\chi_F^{\widehat \P^2,H}(2G+(n-1)F,P^s)$, and then apply  \corref{blowdownform} with $k=n-1$. This requires less computation than the first approach. However,
as part (3) of \corref{P2P12GH} holds  for $\chi_0^{\widehat \P^2,H}(2G+(n-1)F,P^s)$, $\chi_F^{\widehat \P^2,H}(2G+(n-1)F,P^s)$ only when $s\le 2n-2$, this method only allows to compute
$\chi_0^{\P_2,H}(nH,P^r)$ when 
$$r+\max\big(\deg_x(h_{n-1}(x,\lambda),\deg_x(l_{n-1}(x,\lambda))\big)\le 2n-2,
$$ 
and the same way it only allows to compute 
$\chi_H^{\P_2,H}(nH,P^r)$ when 
$$r+\max\big(\deg_x(f_{n-1}(x,\lambda),\deg_x(g_{n-1}(x,\lambda))\big)\le 2n-2.$$
 As the degree of $S_n$ and $R_n$ in $x$ grows faster than $2n$, 
this requires that $n$ and $r$ are both relatively small. 
We will use this to compute $\chi_0^{\P_2,H}(nH)$, $\chi_H^{\P_2,H}(nH)$ by hand for $n=4,5$.
\end{enumerate}

\subsection{Rationality of the generating function}\label{P2rat}
We now use the above algorithm to prove a structural result about the $K$-theoretic Donaldson invariants of $\P^2$ and the blowups of $\P^2$.
\begin{Theorem}\label{P2rat1}
\begin{enumerate}
\item For all $n\in \Z$, $r\in \Z_{\ge 0}$ with $n+r$ even, there exists an  integer $d^1_{n,r}$ and a polynomial $p^1_{n,r}\in \Q[\Lambda^4]$, such that
$$\chi^{\P^2,H}_{H}(nH,P^r)\equiv \frac{p^1_{n,r}}{\Lambda (1-\Lambda^4)^{d^1_{n,r}}}.$$
Furthermore we can choose $p^1_{n,0}\in \Z[\Lambda^4]$.
\item For all $n\in \Z$, $r\in 2\Z_{\ge 0}$ there exists an integer $d^0_{n,r}$ and a polynomial $p^0_{n,r}\in \Q[\Lambda^4]$, such that
$$\chi^{\P^2,H}_{0}(nH,P^r)\equiv \frac{p^0_{n,r}}{(1-\Lambda^4)^{d^0_{n,r}}}.$$
Furthermore we can choose $p^0_{n,0}\in \Z[\Lambda^4]$.
\end{enumerate}
\end{Theorem}

\begin{proof}
By \corref{P2P12GH} there exist for all $L=nF$, $L=(n-1/2)F+G$ with $n\in \Z$ and all $r\in \Z_{\ge 0}$ integers
$e^F_{n,r}$, $e^0_{n,r}$ and polynomials $q^F_{n,r}, q^0_{n,r}\in \Q[\Lambda^4]$, so that
$$\chi^{\widehat \P^2,H}_F(L,P^r)=\frac{q^F_{n,r}}{(1-\Lambda^4)^{e^F_{n,r}}}, \quad
\chi^{\widehat \P^2,H}_0(L,P^r)=\frac{q^0_{n,r}}{(1-\Lambda^4)^{e^0_{n,r}}}.$$
Thus part (H) of \corref{blowdownform} (with $k=n$) gives
$\chi^{\P^2,H}_{H}(nH,P^r)= \frac{p^1_{n,r}}{\Lambda (1-\Lambda^4)^{d^1_{n,r}}},$ for suitable 
$d^1_{n,r}\in \Z_{\ge 0}$, $p^1_{n,r}\in \Q[\Lambda^4]$ and similarly 
part (0) 
of \corref{blowdownform} (with $k=n$) gives
$\chi^{\P^2,H}_{0}(nH,P^r)= \frac{p^0_{n,r}}{ (1-\Lambda^4)^{d^0_{n,r}}},$ for suitable 
$d^0_{n,r}\in \Z_{\ge 0}$, $p^0_{n,r}\in \Q[\Lambda^4]$.
Finally we want to see that $p^1_{n,0}\in \Z[\Lambda^4]$, and we can chose $p^0_{n,0}$ so that it is in $\Z[\Lambda^4]$.
By definition we have  
$$\chi^{\P^2,H}_{H}(nH)=\sum_{k>0} \chi(M^{X}_{H}(H,4k-1),\mu(nH))\Lambda^d\in \Lambda^3\Z[[\Lambda^4]].$$
Writing $p^0_{n,0}=\sum_{k>0} a_k \Lambda^{4k}$ we see from the formula $\chi^{\P^2,H}_{H}(nH)=\frac{p^1_{n,0}}{\Lambda (1-\Lambda^4)^{d^1_{n,0}}}$
that $a_0=0$,  and inductively that $$a_k=\chi(M^{X}_{H}(H,4k-1))-\sum_{i=1}^{k}a_{k-i} \binom{d^{1}_{n,0}+i-1}{i}\in \Z.$$
For $k$ large enough we have that the coefficient of $\Lambda^{4k}$ of $\chi^{\P^2,H}_{0}(nH)$ is $\chi(M^{X}_{H}(0,4k),\mu(nH))$. Thus, adding a polynomial 
$h\in Q[\Lambda^4]$ to $p^0_{n,0}$, we can assume that $\frac{p^0_{n,0}}{ (1-\Lambda^4)^{d^0_{n,0}}}\in \Z[\Lambda^4]$. One concludes in the same way as for $p^1_{n,0}$.
\end{proof}
Indeed a more careful argument will show that we can choose $p^0_{n,r}, \ p^1_{n,r}\in \Z[\Lambda^4]$ for all $r$.

We now use this result and the blowup formulas to describe the generating functions of the $K$-theoretic Donaldson invariants of blowups of $\P^2$ in finitely many points in an open subset of the ample cone
as rational functions. 

\begin{Lemma}\label{blowindep}
Let $X$ be the blowup of $\P^2$ in finitely many points, $p_1,\ldots,p_n$, and denote by $E_1,\ldots,E_n$ the exceptional divisors. Fix $c_1\in H^2(X,\Z)$, and an $r\ge 0$. Let $L$ be a line bundle on $X$. Let $\omega=H-\alpha_1 E_1-\ldots -\alpha_n E_n$ with $|\alpha_i|<\frac{1}{\sqrt{n}}$, for all $i$, and $\<\omega, K_X\><0$.
Then $\chi^{X,\omega}_{c_1}(L,P^r)\equiv\chi^{X,H}_{c_1}(L,P^r)$.
\end{Lemma}
\begin{proof}
We put  $\epsilon:=\max(|\alpha_i|)_{i=1}^n$, and $\delta:=\frac{1}{n}-\epsilon^2>0$.
Let $L=dH-m_1E_1-\ldots m_nE_n$, with $d,m_1,\ldots,m_n\in \Z$, and let $r\ge 0$. 
We want to show that there are  only finitely many classes $\xi$ of type $(c_1)$ on $X$ with 
$\<H,\xi\>\ge 0 \ge \<\omega,\xi\>$ and $\delta^X_\xi(L,P^r)\ne 0$. As by \lemref{vanwall} each $\delta^X_\xi(L,P^r)$ is a polynomial in $\Lambda$, this gives  $\chi^{X,\omega}_{c_1}(L,P^r)\equiv \chi^{X,H}_{c_1}(L,P^r)$.

We write $\xi=aH-b_1E_1-\ldots -b_n E_n$, and $b:=(|b_1|+\ldots+|b_n|$); then we get $a\ge 0$ and 
$$0\ge \<\omega,\xi\>=a-\alpha_1b_1-\ldots-\alpha_nb_n\ge a-b\epsilon,$$
i.e. $a\le b\epsilon$.
Assume $\delta_\xi^X(L,P^r)\ne 0$, then by \lemref{vanwall}  $-\xi^2\le |\<\xi, (L-K_X)\>|+r+2$.
We have 
\begin{equation}\label{ineq}\xi^2=-a^2+b_1^2+\ldots+b_n^2\ge-\epsilon^2b^2+\frac{b^2}{n}=\delta b^2,
\end{equation} where we have used the easy inequality $b_1^2+\ldots+b_n^2\ge \frac{b^2}{n}$ 
and our definition  $\frac{1}{n}-\epsilon^2=\delta>0$.

On the other hand, putting $m:=\max|m_i+1|_{i=1}^n$ we get
\begin{align*}
|\<\xi, (L-K_X)\>|+r+2&=|a(d+3)-(m_1+1) b_1-\ldots-(m_n+1)b_n|+r+2\\&\le a|d+3|+|m_1+1| |b_1|+\ldots+|m_n+1| |b_n|+r+2\\&\le 
\epsilon b|d+3|+mb+r+2= (m+|d+3|\epsilon)b+r+2.
\end{align*}
Putting this together with \eqref{ineq}, and using  $\epsilon\le 1$, we get
 \begin{equation}
 \label{blowbound}
 \delta (|b_1|+\ldots+|b_n|)\le \max|m_i+1|_{i=1}^n +|d+3|+\frac{r+2}{|b_1|+\ldots+|b_n|}.
 \end{equation}
  thus $b=|b_1|+\ldots+|b_n|$ is bounded  and $a\le b\epsilon$ is bounded, and therefore there are only finitely many choices for $\xi$.
\end{proof}

The following theorem contains \thmref{rationalal} as a special case.
\begin{Theorem}\label{blowrat}
Let $X$ be the blowup of $\P^2$ in finitely many points. With the assumptions and notations of \lemref{blowindep}, there exist an integer
$d^{c_1}_{L,r}\in\Z_{\ge 0}$ and  a polynomial $p^{c_1}_{L,r}\in \Q[\Lambda^{\pm 4}]$, such that 
$$\chi^{X,\omega}_{c_1}(L,P^r)\equiv\frac{p^{c_1}_{L,r}}{\Lambda^{c_1^2}(1-\Lambda^4)^{d^{c_1}_{L,r}}}.$$
\end{Theorem}
\begin{proof}
We write $c_1=kH+l_1E_1+\ldots +l_nE_n$. By renumbering the $E_i$ we can assume that $l_i$ is odd for $1\le i\le s$ and $l_i$ is even for $s+1\le i\le n$.  
Write $L=dH-m_1E_1-\ldots -m_nE_n$, with $d,m_1,\ldots,m_n\in \Z$.

By \lemref{blowindep}, it is enough to show the claim for $\omega=H$.
By repeatedly  applying \thmref{nblow}, we get 
$$\chi^{X,H}_{c_1}(L,P^r)=\chi^{\P^2,H}_{kH}\Big(dH,P^r\cdot \Big(\prod_{i=1}^s S_{m_i+1}(P,\Lambda)\Big)\cdot \Big(\prod_{i=s+1}^n R_{m_i+1}(P,\Lambda)\Big)\Big).$$
Put $\kappa=0$ if $k$ is even, and $\kappa=1$ if $k$ is odd.
We know that  $\chi^{\P^2,H}_{kH}(dH,P^r)$ depends only on $\kappa$, and by \thmref{P2rat1} we have 
$$\chi^{\P^2,H}_{kH}(dH,P^r)=\frac{p^{\kappa}_{d,r}}{\Lambda^\kappa (1-\Lambda^4)^{d^\kappa_{d,r}}},$$
We know that $R_n(P,\Lambda)\in \Z[P,\Lambda^4]$, $S_n(P,\Lambda)\in \Lambda\Z[P,\Lambda^4]$.
Therefore  we can write $\chi^{X,H}_{c_1}(L,P^r)=\frac{ p}{\Lambda^{\kappa-s}(1-\Lambda^4)^N}$  for a suitable polynomial $p\in \Q[\Lambda^{\pm 4}]$ and a nonnegative integer $N$.
Note that $c_1^2=k^2-l_1^2-\ldots -l_n^2\equiv \kappa-s\mod 4$. Let $w:=\frac{1}{4}(c_1^2-(\kappa-s))$.  Then 
$$\chi^{X,H}_{c_1}(L,P^r)=\frac{ p}{\Lambda^{\kappa-s}(1-\Lambda^4)^N}=\frac{\Lambda^{4w}p}{\Lambda^{c_1^2}(1-\Lambda^4)^N},$$ and
the claim follows. 
\end{proof}

\subsection{Explicit computations for small $n$}\label{explicitp2}
We compute $\chi^{\P^2,H}_{0}(nH)$, $\chi^{\P^2,H}_{H}(nH)$ for small values of $n$, using the blowup formulas, using the strategy outlined in \secref{strategy}.
In the next subsection we do the same computations for larger $n$ using a computer program written in Pari.
These invariants have been computed before  (see \cite{Abe},\cite{GY}) for $1\le n \le 3$.

\begin{Proposition}
\begin{enumerate}
\item 
$\displaystyle{\chi^{\P^2,H}_{H}(4H)=\frac{\Lambda^3+6\Lambda^7+\Lambda^{15}}{(1-\Lambda^4)^{15}}.}$
\item 
$\displaystyle{\chi^{\P^2,H}_{0}(4H)=\frac{1+6\Lambda^8+\Lambda^{12}}{(1-\Lambda^4)^{15}}-1-51/2\Lambda^4.}$
\end{enumerate}
\end{Proposition}

\begin{proof}
(1) We have $$S_3(x,\lambda)=\lambda \left(x^2-(1-\lambda^4)^2\right), \quad S_4=\lambda x\big((1-\lambda^8)x^2-2(1-\lambda^4)^3\big). $$
Using division with rest as polynomials in $x$, we  write 
$\lambda(1-\lambda^4)^6$ as a linear combination of $S_3(x,\lambda)$ and $S_4(x,\lambda)$
$$\lambda(1-\lambda^4)^6=\left((1-\lambda^8)x^2-(1-\lambda^4)^4\right)S_3(x,\lambda)-x S_4(x,\lambda).$$
Thus we get by \corref{blowdownform} that
\begin{equation}\label{4HB}
\begin{split}
\chi^{\P^2,H}_{H}(4H)&=\frac{1}{\Lambda(1-\Lambda^4)^6}\big((1-\Lambda^8)\chi^{\widehat \P^2,H}_{F}(4H-2E,P^2)-(1-\Lambda^4)^4\chi^{\widehat \P^2,H}_{F}(4H-2E)\\&\qquad\qquad-
\chi^{\widehat \P^2,H}_{F}(4H-3E,P^1)\big).
\end{split}
\end{equation}
By \propref{p11GM} we have 
$$\chi^{\widehat \P^2,F_+}_{F}(4H-3E,P^1)=\frac{1}{(1-\Lambda^4)^8}-1,$$ and $\xi=H-3E$ is the only class of type $(F)$ on $\widehat \P^2$ with 
$\<\xi, H\>\ge0 >\<\xi, F\>$ with $\delta_\xi^{\widehat\P^2}((4H-3E),P^1)\ne 0$. In fact   $\delta_{H-3E}^{\widehat\P^2}((4H-3E),P^1)=\Lambda^8.$
Thus $$\chi^{\widehat \P^2,H}_{F}(4H-3E,P^1)=\frac{1}{(1-\Lambda^4)^8}-1+\Lambda^8.$$
By \propref{p112GM} we have that
$$\chi^{\widehat \P^2,F_+}_{F}(4H-2E)=\frac{3\Lambda^4+\Lambda^{12}}{(1-\Lambda^4)^{12}},\quad \chi^{\widehat \P^2,F_+}_{F}(4H-2E,P^2)=\frac{(1+\Lambda^4)^2}{(1-\Lambda^4)^{10}}-1.$$
Furthermore there is no  class of type $(F)$ on 
$\widehat \P^2$ with 
$\<\xi, H\> \ge 0 >\<\xi, F\>$ with $\delta_\xi^{\widehat\P^2}(4H-2E)\ne 0$ or $\delta_\xi^{\widehat\P^2}(4H-2E,P^2)\ne 0$.
 Thus $\chi^{\widehat \P^2,H}_{F}(4H-2E)=\chi^{\widehat \P^2,F_+}_{F}(4H-2E)$ and  $\chi^{\widehat \P^2,H}_{F}(4H-2E,P^2)=\chi^{\widehat \P^2,F_+}_{F}(4H-2E,P^2)$.
 Putting these values into \eqref{4HB} yields $\chi^{\P^2,H}_{H}(4H)=\frac{\Lambda^3+6\Lambda^7+\Lambda^{15}}{(1-\Lambda^4)^{15}}$.

(2) For  $R_3(x,\lambda)=-\lambda^4x^2+(1-\lambda^4)^2$, $R_4=-\lambda^4 x^2+(1-\lambda^4)^4,$ we get 
$$(1-\lambda^4)^5=\left(\lambda^4x^2+(1-\lambda^4)^2\right)R_3(x,\lambda)-\lambda^4R_4(x,\lambda).$$
Thus  \corref{blowdownform} gives
\begin{equation}\label{40B}
\chi^{\P^2,H}_{0}(4H)=\frac{1}{(1-\Lambda^4)^5}\left(\Lambda^4\chi^{\widehat \P^2,H}_{0}(4H-2E,P^2)+(1-\Lambda^4)^2\chi^{\widehat \P^2,H}_{0}(4H-2E)
-\Lambda^4\chi^{\widehat \P^2,H}_{0}(4H-3E)\right).
\end{equation}
By \propref{p11GM} we have 
$\chi^{\widehat \P^2,F_+}_{0}(4H-3E)=\frac{1}{(1-\Lambda^4)^9}-1-35/2\Lambda^4$. Furthermore there are no classes $\xi$ of type $(0)$ with 
$\<\xi, H\> >0 >\<\xi ,F\>$ with $\delta_\xi^{\widehat\P^2}((4H-3E))\ne 0$, and the only classes of type $(0)$ with $\<\xi, H\> =0 >\<\xi, F\>$
are $-2E$ and $-4E$ with 
$$\frac{1}{2}\delta^{\widehat \P^2}_{-2E}(4H-3E)=-2\Lambda^4 + 291\Lambda^8 - 3531\Lambda^{12} + 16215/2\Lambda^{16}, \quad
\frac{1}{2}\delta^{\widehat \P^2}_{-4E}(4H-3E)=7\Lambda^{16} - 51/2\Lambda^{20},$$ giving
$$\chi^{\widehat \P^2,H}_{0}(4H-3E)=\frac{1}{(1-\Lambda^4)^9}-1-39/2\Lambda^4 + 291\Lambda^8 - 3531\Lambda^{12} + 16229/2\Lambda^{16}- 51/2\Lambda^{20}.$$
By \propref{p112GM} we have 
$\chi^{\widehat \P^2,F_+}_{0}(4H-2E)=\frac{1+3\Lambda^8}{(1-\Lambda^4)^{12}}-1-21\Lambda^4$. Furthermore the only class  $\xi$ of type $(0)$ with 
$\<\xi, H\> \ge 0 >\<\xi, F\>$ with $\delta_\xi^{\widehat\P^2}(4H-2E)\ne 0$ is  $-2E$ with  $\frac{1}{2}\delta^{\widehat \P^2}_{-2E}(4H-2E)=-3/2\Lambda^4 + 108\Lambda^8 - 1225/2\Lambda^{12}$, giving
$$\chi^{\widehat \P^2,H}_{0}(4H-2E)=\frac{1+3\Lambda^8}{(1-\Lambda^4)^{12}}-1-45/2\Lambda^4 + 108\Lambda^8 - 1225/2\Lambda^{12}.$$
By \propref{p112GM} we have 
$\chi^{\widehat \P^2,F_+}_{0}(4H-2E,P^2)=\frac{(1+\Lambda^4)^2}{(1-\Lambda^4)^{10}}-1-48\Lambda^4+389\Lambda^8$, and the classes $\xi$ of type $(0)$ with $\<\xi, H\> \ge 0 >\<\xi, F\>$ 
with $\delta_\xi^{\widehat\P^2}(4H-2E,P^2)\ne 0$ are $-2E$ and $-4E$ with $\frac{1}{2}\delta^{\widehat\P^2}_{-2E}(4H-2E,P^2)=-6\Lambda^4 + 508\Lambda^8 - 4614\Lambda^{12} + 
8600\Lambda^{16}
$ and $\frac{1}{2}\delta^{\widehat\P^2}_{-4E}(4H-2E,P^2)=1/2\Lambda^{16}$, giving
$$\chi^{\widehat \P^2,H}_{0}(4H-2E)=\frac{(1+\Lambda^4)^2}{(1-\Lambda^4)^{10}}-1-54\Lambda^4+ 897\Lambda^8 - 4614\Lambda^{12} + 17201/2\Lambda^{16} .$$
Putting this into \eqref{40B} gives
$\chi^{\P^2,H}_{0}(4H)=\frac{ 1 + 6\Lambda^8 + \Lambda^{12}}{(1-\Lambda^4)^{15}}-1-51/2\Lambda^4$.
\end{proof}

\begin{Proposition}
$\displaystyle{\chi^{\P^2,H}_{0}(5H)=\frac{1+21\Lambda^8+20\Lambda^{12}+21\Lambda^{16}+\Lambda^{24}}{(1-\Lambda^4)^{21}}-1-33\Lambda^4.}$
\end{Proposition}
\begin{proof}
We use \propref{p112GM} to compute $\chi^{\widehat \P^2,F_+}_{0}(5H-3E,P^r)$ for $r=0,2,4$, and \propref{p11GM} to compute $\chi^{\widehat \P^2,F_+}_{0}(5H-4E,P^r)$ for $r=0,2$. 
The only classes of type $(0)$ with $\<H ,\xi\> \ge  0 >\<F,\xi\>$ and  $\delta^{\widehat \P^2}_\xi(5H-3H,P^r)\ne 0$ for $r=0,2,4$ or $\delta^{\widehat \P^2}_\xi(5H-4H,P^r)\ne 0$ for $r=0,2$ 
are $-2E$ and $-4E$. Adding their wallcrossing terms to the $\chi^{\widehat \P^2,F_+}_{0}(5H-3E,P^r)$,  $\chi^{\widehat \P^2,F_+}_{0}(5H-4E,P^r)$ we get  
\begin{align*}\chi^{\widehat \P^2,H}_{0}(5H-3E)&=\frac{1+6\Lambda^8+\Lambda^{16}}{(1-\Lambda^4)^{15}} -1 - 27\Lambda^4 + 366\Lambda^8 - 6066\Lambda^{12} + 18917\Lambda^{16} - 33\Lambda^{20},\\
\chi^{\widehat \P^2,H}_{0}(5H-3E,P^2)&=\frac{(1+\Lambda^4)^3}{(1-\Lambda^4)^{13}}-1 - 64\Lambda^4 + 2163\Lambda^8 - 32806\Lambda^{12} + 172163\Lambda^{16}\\&\qquad  - 242616\Lambda^{20} + 1007\Lambda^{24},\\
\chi^{\widehat \P^2,H}_{0}(5H-3E,P^4)&=\frac{2(1+\Lambda^4)^2}{(1-\Lambda^4)^{11}} -2 - 218\Lambda^4 + 10110\Lambda^8 - 170462\Lambda^{12} + 1121538\Lambda^{16} \\&\qquad- 2798450\Lambda^{20} + 2249462\Lambda^{24} - 18786\Lambda^{28},\\
\chi^{\widehat \P^2,H}_{0}(5H-4E)&=\frac{1}{(1-\Lambda^4)^{11}} -1 - 23\Lambda^4 + 786\Lambda^8 - 20234\Lambda^{12} + 124671\Lambda^{16} - 201885\Lambda^{20}\\&\qquad + 18372\Lambda^{24} - 21840\Lambda^{28},\\
\chi^{\widehat \P^2,H}_{0}(5H-4E,P^2)&=\frac{1}{(1-\Lambda^4)^9} -1 - 57\Lambda^4 + 3691\Lambda^8 - 95035\Lambda^{12} + 741175\Lambda^{16} - 2043587\Lambda^{20} \\&\qquad+ 1906119\Lambda^{24} - 414993\Lambda^{28} + 295880\Lambda^{32}.
\end{align*}
We compute 
$$R_5(x,\lambda)=-\lambda^4x^6+\lambda^4(1-\lambda^4)^2(2+\lambda^4)x^4-3\lambda^4(1-\lambda^4)^4x^2+(1-\lambda^4)^6.$$
Using again division with rest, we get 
\begin{align*}
(1-\lambda^4)^{11}&=\big((\lambda^4+3\lambda^8)x^4+(\lambda^4-8\lambda^8+10\lambda^{12}-3\lambda^{20})x^2 +(1+4\lambda^8-\lambda^{12})(1-\lambda^4)^4\big)R_4(x,\lambda)\\
&\qquad
-\big((\lambda^4+3\lambda^8)x^2+ (3+\lambda^4)(1-\lambda^4)^2\big)R_5(x,\lambda).
\end{align*}
Thus again we get
 $(1-\Lambda^4)^{11}\chi^{\P^2,H}_{0}(5H)$ as the result of replacing $\lambda$ by $\Lambda$ and  $x^rR_4(x,\lambda)$ by $\chi^{\widehat \P^2,H}_{0}(5H-3E,P^r)$, $x^rR_5(x,\lambda)$ 
 by $\chi^{\widehat \P^2,H}_{0}(5H-4E,P^r)$.
This gives after some computation that
$\chi^{\P^2,H}_{0}(5H)=\frac{1 + 21\Lambda^8 + 20\Lambda^{12} + 21\Lambda^{16} + \Lambda^{24}}{(1-\Lambda^4)^{21}}-1-33\Lambda^4.$
\end{proof}

\subsection{Computer computations for larger $n$}
\label{CompP2}

We outline the computations of the PARI program to compute the $\chi^{\P^2,H}_0(nH,P^r)$, 
$\chi^{\P^2,H}_H(nH,P^r)$.

We have carried out these computations for $r=0$ and $n\le 11$, and $r\le 16$ and $n\le 8$. 

To have effective bounds on the number of terms needed to compute we use the following remark.
\begin{Remark} \label{qlevel} For all $l\in \frac{1}{2}\Z$ we have 
$$
\frac{1}{\sinh(lh)},\ \coth(lh)\in q\Lambda^{-1}\C[[q^{-1}\Lambda,q^4]],\quad h,\ \exp(lh),\ \WT_4(h),\ \Lambda^2u',\ h^*,\ M\in \C[[q^{-1}\Lambda,q^4]].$$
Therefore we have the following.
\begin{enumerate}
\item 
By \propref{Fplus}, for $X=\P^1\times \P^1$ or $\widehat \P^2$, to compute $\chi^{X,F_+}_F(nF+mG,P^r)$ modulo $\Lambda^{k+1}$, it is enough to evaluate the 
formulas of \propref{Fplus} modulo $q^{k+1}$ and modulo $\Lambda^{k+1}$.
\item 
By \defref{wallcrossterm} for any rational surface $X$ and any class $\xi\in H^2(X,\Z)$ with $\xi^2<0$ and any line bundle $L\in \Pic(X)$,  to compute $\delta^X_{\xi}(L,P^r)$
modulo $\Lambda^{k+1}$ 
it is enough to evaluate the formulas of \defref{wallcrossterm} modulo $q^{k+1}$ and modulo $\Lambda^{k+1}$.
\end{enumerate}
\end{Remark}

{\bf Step 1.} As mentioned above we will use \corref{blowdownform} with $k=n$.
The polynomials $f_n,g_n$, $h_n,l_n$, and the integers $N_n$, $M_n$ of \corref{blowdownform}  are computed by the program as follows.
Apply the Euclidean algorithm in $\Q(\lambda)[x^2]$ (i.e. repeated division with rest) to $S_n$, $S_{n+1}$, to find $\overline f_n$, $\overline g_n\in \Q(\lambda)[x]$ with 
$\overline f_n S_n+\overline g_n S_{n+1}=1$. Choose the minimal $N_n\in \Z_{\ge 0}$, so that 
$$f_n:=\lambda(1-\lambda^4)^{N_n}  \overline f_n, \ 
g_n:=\lambda(1-\Lambda^4)^{N_n} \overline g_n\in \Q[x,\Lambda^4].$$
These exist by \propref{blowdownpol}.
Similarly $h_n$, $l_n$, $M_n$ are computed as follows.
Apply the Euclidean algorithm in $\Q(\lambda)[x^2]$ to $R_n$, $R_{n+1}$, to find $\overline h_n$, $\overline l_n\in \Q(\lambda)[x]$ with $\overline h_nR_n+\overline l_nR_{n+1}=1$, and then again multiply with the minimal power 
$(1-\lambda^4)^{M_n}$, to obtain.
$$h_n:=(1-\lambda^4)^{M_n} \overline h_n,\  l_n:=(1-\lambda^4)^{M_n} \overline l_n\in \Q[x^2,\lambda^4].$$

{\bf Step 2.} Use \propref{p11r} to compute $\chi^{X,F_+}_F(nF,P^{2s})$ for $2s\le \deg_x(g_n)+r$ and  $\chi^{X,F_+}_0(nF,P^{2s})$ for $2s\le\deg_x(l_n)+r$.
For $s=0$, the formula is explicitly given in \propref{p11r}. For $s>0$ we know by \propref{p11r} that   $\chi^{X,F_+}_F(nF,P^{2s})$ is a polynomial in $\Lambda^4$ of degree at most $s$ and $\chi^{X,F_+}_0(nF,P^s)$ a polynomial of at most degree $s+1$ in $\Lambda^4$, 
So, using \remref{qlevel},  the computation is done by evaluating the formula of \propref{Fplus} as a power series in $\Lambda,q$ modulo $\Lambda^{4s+1}$ and 
$q^{4s+1}$ or
$\Lambda^{4s+5}$, $q^{4s+5}$ respectively. 
As all the power series in the formula are completely explicit, this is a straightforward evaluation.  


In the same way we use \propref{p11GM} to compute 
 $\chi^{X,F_+}_F(G+(n-\frac{1}{2})F,P^{2s+1})$ for $2s+1\le \deg_x(f_n)+r$ and  $\chi^{X,F_+}_0(G+(n-\frac{1}{2})F,P^{2s})$ for $2s\le\deg_x(h_n)+r$.
 By \propref{p11GM} 
 $$\chi^{X,F_+}_F(G+(n-\frac{1}{2})F,P^{2s+1})-\frac{1}{(1-\Lambda^4)^{2n-2s}}, \quad 
\chi^{X,F_+}_0(G+(n-\frac{1}{2})F,P^{2s})-\frac{1}{(1-\Lambda^4)^{2n+1-2s}}$$  are both polynomials of degree at most $s+1$ in $\Lambda^4$, so, using also \remref{qlevel}, 
again they are computed by 
evaluating the formula of \propref{Fplus} as a power series in $\Lambda,q$ modulo $\Lambda^{4s+5}$ and $q^{4s+5}$.
Again this is a straightforward evaluation.

{\bf Step 3.}
By the proof of \corref{P2P12GH} there are finitely many classes $\xi=aF-bG$ of type $(0)$ or $F$ on $\widehat \P^2$ with $\<\xi ,H\> \ge 0>\<\xi ,F\>$ and 
$\delta_\xi^{\widehat \P^2}(nF,P^s)\ne 0$ or $\delta_\xi^{\widehat \P^2}(G+(n-\frac{1}{2})F,P^s)\ne 0$. In \remref{nonwalls} effective bounds for $a$ and $b$ are given in terms of $n$ and $s$, which leave only finitely many possibilities. For all $\xi=aF-bG$, so that $(a,b)$ satisfies these bounds, it is  first checked whether indeed the criterion 
$-\xi^2\le |\<\xi ,L-K_{\widehat \P^2})\>|+s+2$ for  the non-vanishing of  $\delta_\xi^{\widehat \P^2}(L,P^s)$ for $L=nF$ or $L=(n-1/2)F+G$ is satisfied.
If yes, $\delta_\xi^{\widehat \P^2}(L,P^s)$ is computed by evaluating the formula of \defref{wallcrossterm}. 
By \lemref{vanwall} we have that $\delta_\xi^{\widehat \P^2}(L,P^s)$ is a polynomial in $\Lambda$ of degree at most 
$$a(\xi,L,X,s):=\xi^2+2|\<\xi,L-K_{\widehat \P^2})\>|+2s+4,$$
 so to determine $\delta_\xi^{\widehat \P^2}(L,P^s)$ we only need to compute it modulo 
$\Lambda^{a(\xi,L,X,s)+1}$ and thus by \remref{qlevel} we only need to evaluate the formula of  \defref{wallcrossterm} modulo $\Lambda^{a(\xi,L,X,s)+1}$  and 
$q^{a(\xi,L,X,s)+1}$, so this is again a straightforward evaluation.

Then for $c_1=0,F$ and $L=nF$, $(n-1/2)F+G$, we compute
$$\chi^{\widehat \P^2, H}_{c_1}(L,P^s):=\chi^{\widehat \P^2, F_+}_{c_1}(L,P^s)+\frac{1}{2}\sum_{\<\xi, H\>=0>\<\xi,F\>}\delta^X_\xi(L,P^s)+
\sum_{\<\xi, H\>>0>\<\xi,F\>} \delta^X_\xi(L,P^s),$$ where
the sums are over all $\xi$ of type $(c_1)$ with $\delta^{\widehat \P^2}_\xi(L,P^s)\ne 0$.

{\bf Step 4.}
Finally apply \corref{blowdownform} to compute 
\begin{align*}
\chi^{\P^2,H}_H(nH,P^r)&=\frac{1}{\Lambda(1-\Lambda^4)^{N_k}}\Big(\chi^{\widehat \P^2,H}_{F}\big(G+(n-1/2)F,P^r\cdot f_n(P,\Lambda)\big)
+
\chi^{\widehat \P^2,H}_{F}\big(nF,P^r\cdot g_n(P,\Lambda)\big)\Big),\\
\chi^{\P^2,H}_0(nH,P^r)&=\frac{1}{(1-\Lambda^4)^{M_k}}\Big(\chi^{\widehat \P^2,H}_{0}\big(G+(n-1/2)F,P^r\cdot h_n(P,\Lambda)\big)
+
\chi^{\widehat \P^2,H}_{0}(\big(nF,P^r\cdot l_n(P,\Lambda)\big)\Big).
\end{align*}
At this point all the terms on the right hand side have already been computed.

We have carried out this computation for the following cases.
\begin{enumerate}
\item For $\chi^{\P^2,H}_H(nH,P^r)$ with $n\equiv r\mod 2$, in the cases $r\le 1$, $n\le 10$ and $r\le 16$, $n\le 8$.
\item For $\chi^{\P^2,H}_0(nH,P^r)$ with $r$ even, in the cases $r=0$, $n\le 11$ and $r\le 15$, $n\le 8$.
\end{enumerate}
For the case $r=0$ we obtain, with the notations of the introduction:
\begin{Proposition}\label{propp2}
With the notations of \thmref{mainp2} we have for $1\le n\le 11$
\begin{enumerate}
\item $\displaystyle{\chi^{\P^2,H}_0(nH)=\frac{P_n(\Lambda)}{(1-\Lambda^4)^{\binom{n+2}{2}}}-1-\frac{1}{2}(n^2+6n+11)\Lambda^4}$,
\item If $n$ is even, then  $\displaystyle{\chi^{\P^2,H}_H(nH)=\frac{Q_n(\Lambda)}{(1-\Lambda^4)^{\binom{n+2}{2}}}}$.
\end{enumerate}
\end{Proposition}

\thmref{mainp2} now follows directly from \propref{propp2} and \propref{highvan}.

We list also the results for $ \chi^{\P^2,H}_{H}(nH,P^1)$.
We put 
{\small \begin{align*}
&q_1=2-t,\ q_3=2,\ q_5=2 + 20t + 20t^2 + 20t^3 + 2t^4,\\
&q_7=2 + 80t + 770t^2 + 3080t^3 + 7580t^4 + 9744t^5 
+ 7580t^6 + 3080t^7 + 770t^8 + 80t^9 + 2t^{10},\\
&q_9=2 + 207t + 6192t^2 + 85887t^3 + 701568t^4 + 3707406t^5+ 13050156t^6 + 31611681t^7  \\&
+ 53322786t^8 + 63463686t^9 + 53322786t^{10} + 31611681t^{11} + 13050156t^{12}
+ 3707406t^{13}\\
&+ 701568t^{14}+ 85887t^{15}+ 6192t^{16} + 207t^{17}+2t^{18}
\end{align*}}

\begin{Proposition}
For $1\le n\le 9$ we have 
$\displaystyle{
\chi^H_{\P^2,dH}(H,P^1)=\frac{\Lambda^{3}q_{n}(\Lambda^4)}{(1-\Lambda^4)^{\binom{n+2}{2}-1}}}.$
\end{Proposition}

We list also in the form of tables part of the results obtained for $\chi^{\P^2,H}_{c_1}(nH,P^r)$, with $c_1=0, H$ for $r>0$.  Here to simplify the expressions we 
only write down the results up to adding a Laurent polynomial in $\Lambda$. 
We define polynomials $p_{d,r}$ by the following tables.
$$\begin{tabular}{r|c|c|c|c|c|c|c}
 &d=3&5&7\\ 
 \hline
 r=1&$2t$& $2t+20t^2+20t^3+20t^4+2t^5$&$\genfrac{}{}{0pt}{}{2t + 80^2 + 770t^3 + 3080t^4 + 7580t^5 + 9744t^6}{ + 7580t^7+ 3080t^8  + 770t^9 + 80t^{10}+ 2t^{11}}$\\
 \hline
  2&$1+t$&$1+6t + 25t^2 +25t^3+6t^4 +t^5$&$\genfrac{}{}{0pt}{}{1 + 15t + 239t^2 + 1549t^3 + 5274t^4+ 9306t^5 + 9306t^6 }{+ 5274t^7+ 1549t^8+ 239t^9 + 15t^{10} + t^{11}}$\\
\hline
 3&$1+t$&$8t + 24t^2 + 24t^3 + 8t^4$&$
\genfrac{}{}{0pt}{}{8t + 219t^2 + 1485t^3 + 5159t^4 + 9513t^5 + 9513t^6}{ + 5159t^7 + 1485t^8 + 219t^8 + 8t^{10}}$\\
 \hline
  4&2&$2+14t+32t^{2}+14t^3+2t^4$&$\genfrac{}{}{0pt}{}{2+ 44t + 546t^2 + 2936t^3 + 7676t^4 + 10360t^5 + 7676t^6 }{+ 2936t^7+ 546t^8 + 44t^9 + 2t^{10}}$\\
\hline
 5&2&$1+ 16t + 30t^2 +16t^3+t^4$&$
\genfrac{}{}{0pt}{}{32t+510t^2 + 2820t^3 + 7682t^4 + 10680t^5  }{+ 7682t^6 + 2820t^7+ 510t^8 + 32t^9}$\\
\hline
 6&$t^{-1}+1$&$5+27t+27t^2+5t^3$&$\genfrac{}{}{0pt}{}{5 + 120t + 1209t^2 + 5075t^3 + 9975t^4 + 9975t^5 + 5075t^6 }{+ 1209t^7 + 120t^8 + 5t^9}$\\
  \hline
  7&$3-t$&$4 + 28t + 28t^2 + 4t^3$&$\genfrac{}{}{0pt}{}{1 + 99t + 1134t^2 + 4954t^3 + 10196t^4+ 10196t^5 }{ + 4954t^6 + 1134t^7+ 99t^8 + t^9}$\\
 \hline
  8&&$14 + 36t + 14t^2$&$\genfrac{}{}{0pt}{}{14  + 318t + 2508t^2 + 7874t^3 + 11340t^4 + 7874t^5}{+ 2508t^6 + 318t^7 + 14t^8}$\\
  \hline
 9&& $12 + 40t + 12t^2$&$\genfrac{}{}{0pt}{}{6 + 276t + 2376t^2 + 7884t^3 + 11684t^4 }{+ 7884t^5 + 2376t^6 + 276t^7 + 6t^8}$\\
 \hline
10&&$t^{-1} + 31 + 31t + t^2$&$ \genfrac{}{}{0pt}{}{42 + 810t + 4742t^2 + 10790t^3 + 10790t^4 + 4742t^5}{ + 810t^6 + 42t^7}$\\
  \hline
 11&&$32+32t$&$\genfrac{}{}{0pt}{}{25 + 719t + 4605t^2 + 11035t^3 + 11035t^4}{ + 4605t^5 + 719t^6 + 25t^7}$\\
 \hline
 12&&$6t^{-1}+52+6t$&$ \genfrac{}{}{0pt}{}{132 + 1920t + 8028t^2 + 12608t^3 + 8028t^4 + 1920t^5}{ + 132t^6}$\\
  \hline
 13&&$-t^{-2} +8t^{-1}+50+8t-t^2$&$\genfrac{}{}{0pt}{}{90 + 1756t + 8038t^2 + 13000t^3 + 8038t^4}{ + 1756t^5 + 90t^6}$\\
 \hline
    14&&$22t^{-1}+57-21t+7t^2-t^3$&$ \genfrac{}{}{0pt}{}{t^{-1} + 407 + 4149t + 11827t^2 + 11827t^3 + 4149t^4}{ + 407t^5 + t^6}$\\
\hline
 15&&$-4t^{-2} +36t^{-1} + 36 - 4t$&$\genfrac{}{}{0pt}{}{300 + 3964t + 12120t^2 + 12120t^3 + 3964t^4}{ + 300t^5}$\\
\end{tabular}$$

$$\begin{tabular}{l | c | c | c | c    }
 &d=2&4&6\\ 
\hline
r=2&$1$&$1+3t+4t^2$&$1+10t+89t^2+272t^3+371t^4+210t^5+67t^6+4t^7$\\
4&$t^{-1}$&$2+5t+t^2$ &$2+27t + 168t^2 + 370t^3+ 318t^4 + 123t^5 + 16t^6$\\
6&&5+3t &$5+ 66t+ 287t^2 + 404t^3+ 219t^4 + 42t^5 + t^6$\\
8&&$t^{-1}+7$&$14+149t + 408t^2 + 350t^3 + 98t^4 + 5t^5$\\
10&&$4t^{-1}+5-t$&$42 + 288t + 468t^2+ 208t^3 + 18t^4$\\
12&&$9t^{-1}-1$&$t^{-1} + 116 + 462t + 388t^2 + 57t^3$\\
14&&&$8t^{-1}+280+568t+168t^2$
\end{tabular}$$

\begin{Theorem}\label{rpoly}
With the polynomials $p_{d,r}$ given above, we have
\begin{enumerate}
\item If $r$ is even, then
$\displaystyle{\chi^{\P^2,H}_{0}(dH,P^r)\equiv \frac{p_{d,r}(\Lambda^4)}{(1-\Lambda^4)^{\binom{d+2}{2}-r}}.}$
\item If $d$ and $r$ are both odd, then
$\displaystyle{\chi^{\P^2,H}_{H}(dH,P^r)\equiv \frac{ \Lambda^{-1} p_{d,r}(\Lambda^{4})}{(1-\Lambda^4)^{\binom{d+2}{2}-r}}=
\frac{ \Lambda^{d^2-2r} p_{d,r}(\Lambda^{-4})}{(1-\Lambda^4)^{\binom{d+2}{2}-r}}.}$
\item If $d$ and $r$ are both even,  then 
$\displaystyle{\chi^{\P^2,H}_{H}(dH,P^r)\equiv \frac{ \Lambda^{d^2-2r-1} p_{d,r}(\Lambda^{-4})}{(1-\Lambda^4)^{\binom{d+2}{2}-r}}.}$
\end{enumerate}
\end{Theorem}

\subsection{Invariants of blowups of the plane}

We want to apply the above results to compute $K$-theoretic Donaldson invariants of blowups of $\P^2$ in a finite number of points. 

\begin{Remark} Let $X_r$ be the blowup of $\P^2$ in $r$ general points and let $E:=E_1+\ldots+ E_r$ be the sum of the exceptional divisors. 
By definition we have for $c_1=0,H$ that
$\chi^{X_r,H}_{c_1+E}(nH-E)=\Lambda^r\chi^{\P^2,H}_{c_1}(nH,P^r)$.
By \lemref{blowindep} we have therefore $\chi^{X_r,\omega}_{c_1+E}(nH-E)\equiv\Lambda^r\chi^{\P^2,H}_{c_1}(nH,P^r)$ for all classes $\omega=H-\sum_{i=1}^r a_i E_i$ on $X_r$ with $\<\omega, K_{X_r}\><0$ and $0\le a_i< \frac{1}{\sqrt r}$ for all $i$. Therefore the  formulas of \thmref{rpoly} also give the $\chi^{X_r,\omega}_{c_1+E}(nH-E)$.
 \end{Remark}
 
By \thmref{nblow}, and using \lemref{blowindep}, we can, from the $\chi^{\P^2,H}_{0}(nH,P^r)$, $\chi^{\P^2,H}_{H}(nH,P^r)$,  readily compute the generating functions of $K$-theoretical Donaldson invariants 
$\chi^{X,\omega}_{c_1}(L)$ for any blowup $X$ of $\P^2$ in finitely many points, for any $c_1,L\in \Pic(X)$, and for any $\omega$ close to $H$, up to addition of a Laurent polynomial.
In particular we can readily apply this computation to the tables of the $\chi^{\P^2,H}_{0}(nH,P^r)$, $\chi^{\P^2,H}_{H}(nH,P^r)$ of \thmref{rpoly} above.
We will only write down the result in one simple case.
We take $X_s$ the blowup of $\P^2$ in $s$ points, and let again $E=\sum_{i=1}^s E_i$ be the sum of the exceptional divisors, let $L:=dH-2E$ and consider the cases
$c_1=0$, $c_1=E$, $c_1=H$ and $c_1=K_{X_s}$.

We define polynomials $q_{d,s}$ by the following table,
$$\begin{tabular}{l | c | c|c |c    }
 &d=3&4&5&6\\ 
 \hline
 s=1&1&$1+3t^2$&$1+15t^2+10t^3+6t^4$&$1+46t^2+104t^3+210t^4 + 105t^5 + 43t^6 + 3t^7$\\
 s=2&&$1+t^2$&$1+10t^2+4t^3+t^4$&$1+37t^2+ 70t^3 + 105t^4 + 34t^5 + 9t^6$\\
 s=3&&$1$&$1+6t^2+t^3$&$1+ 29t^2 + 44t^3 + 45t^4 + 8t^5 + t^6$\\
 s=4&&&$1+3t^2$&$1 + 22t^2 + 25t^3 + 15t^4 + t^5$\\
 s=5&&&$1+t^2$&$1+ 16t^2 + 12t^3 + 3t^4$\\
 s=6&&&$1$&$1+11t^2+4t^3$\\
 s=7&&&&$1+7t^2$\\
 \end{tabular}$$
and polynomials $r_{d,s}$  by the following table.
$$\begin{tabular}{l | c | c|c |c    }
 &d=4&6\\ 
 \hline
 s=1&$1+3t$&$1+24t+105t^2+161t^3+168t^4+43t^5+10t^6$\\
 s=2&$1+t$&$1+21t+71t^2+90t^3+63t^4+9t^5+t^6$\\
 s=3&$1$&$1+ 18t+ 45t^2+ 45t^3 + 18t^4+ t^5$\\
 s=4&&$1+15t+26t^2+19t^3+3t^4$\\
 s=5&&$1+12t +13t^2 +6t^3$\\
 s=6&&$1+9t+5t^2+t^3$\\
 s=7&&$1+6t+t^2$\\
 s=8&&$1+3t$\\
 \end{tabular}$$

\begin{Proposition} Let $X_s$ be the blowup of $\P^2$ in $r$ general points with exceptional divisors $E_1,\ldots,E_s$, and write $E=\sum_{i=1}^sE_i$.
With the $q_{d,s}$ and $r_{d,s}$ given by the above tables we get
\begin{align*}
\chi^{X,H}_{0}(dH-2E)&\equiv \frac{q_{d,s}(\Lambda^4)}{(1-\Lambda^4)^{\binom{d+2}{2}-3s}}\\
\chi^{X,H}_{K_{X_s}}(dH-2E)&\equiv \frac{\Lambda^{d^2-1-3s}q_{d,s}(\frac{1}{\Lambda^4})}{(1-\Lambda^4)^{\binom{d+2}{2}-3s}}, \quad d \hbox{ even}\\
\chi^{X,H}_{H}(dH-2E)&\equiv \frac{\Lambda^{3}r_{d,s}(\Lambda^4)}{(1-\Lambda^4)^{\binom{d+2}{2}-3s}}, \quad d \hbox{ even}\\
\chi^{X,H}_{E}(dH-2E)&\equiv \begin{cases}
\frac{\Lambda^{d^2-1-3s} q_{d,s}(\frac{1}{\Lambda^4})}{(1-\Lambda^4)^{\binom{d+2}{2}-3s}}& d \hbox{ odd}\\
\frac{\Lambda^{d^2-4-3s} r_{d,s}(\frac{1}{\Lambda^4})}{(1-\Lambda^4)^{\binom{d+2}{2}-3s}}& d \hbox{ even}
\end{cases}
\end{align*}
The same formulas also apply with $\chi^{X,H}_{c_1}(dH-2E)$ replaced by
$\chi^{X,\omega}_{c_1}(dH-2E)$, with $\omega=H-a_1E_1-\ldots a_rE_s$ and $0\le a_i\le \sqrt{s}$ for all $i$.
\end{Proposition}
\begin{proof}
Recall that $R_3=-\la^4 x^2 + (1-\la^4)^2$, $S_3=\la(x^2-(1-\la^4)^2)$. Noting that $K_{X_s}\equiv H+E\mod 2H^2(X,\Z)$, we
 we get by
 \thmref{nblow} that
\begin{align*}
\chi^{X_s,H}_0(dH-2E)&=\chi^{\P^2,H}_{0}\big(dH,\big(-\Lambda^4 P^2 + 1-\Lambda^4)^2\big)^s\big),\\ 
 \chi^{X_s,H}_H(dH-2E)&=\chi^{\P^2,H}_{H}\big(dH,\big(-\Lambda^4 P^2 + 1-\Lambda^4)^2\big)^s\big),\\
 \chi^{X_s,H}_E(dH-2E)&=\Lambda^s\chi^{\P^2,H}_{0}\big(dH,\big(P^2 - (1-\Lambda^4)^2\big)^s\big),\\ 
 \chi^{X_s,H}_{K_{X_s}}(dH-2E)&=\Lambda^s\chi^{\P^2,H}_{H}\big(dH,\big(P^2 - (1-\Lambda^4)^2\big)^s\big).
\end{align*}
Now we just put the values of the tables of \thmref{rpoly} into these formulas.
\end{proof}

\subsection{Symmetries from Cremona transforms}

\begin{Remark}
The  Conjecture \ref{ratconj} will often predict a symmetry for the polynomials $P^X_{c_1,L}(\Lambda)$.
Assume Conjecture \ref{ratconj}. Then we have the following. 
\begin{enumerate}
\item If $c_1\equiv L+K_X-c_1\mod 2H^2(X,\Z)$, then 
$P^X_{c_1,L}(\Lambda)=\Lambda^{L^2+8-K_X^2}P^X_{c_1,L}(\frac{1}{\Lambda})$.
\item More generally let $X$ be the blowup of $\P^2$ in $n$ points, with exceptional divisors $E_1,\ldots E_n$, $L=dH-a_1E_1-\ldots -a_nE_n$. If
 $\sigma$ is a permutation of $\{1,\ldots, n\}$, we write $\sigma(L):=dH-a_{\sigma(1)}E_1-\ldots -a_\sigma(n)E_n$.
 Then $\chi^{X,H}_{c_1}(L)=\chi^{X,H}_{\sigma(c_1)}(\sigma(L))$. 
  \end{enumerate}
 Thus, if there is a $\sigma$ with  $L=\sigma(L)$ and $\sigma(c_1)\equiv L+K_X-c_1\mod 2H^2(X,\Z)$, then 
 $P^X_{c_1,L}(\Lambda)=\Lambda^{L^2+8-K_X^2}P^X_{c_1,L}(\frac{1}{\Lambda})$.
 \end{Remark}
Other symmetries come from the Cremona transform of the plane, which we briefly review.
Let  $p_1,p_2,p_3$ be three general points in $\P^2$. For $i=1,2,3$ let $L_k$ the line through $p_i,p_j$ where $\{i,j,k\}=\{1,2,3\}$.
Let $X$ be the blowup of $\P^2$ in $p_1,p_2,p_3$, with exceptional divisors $E_1,E_2,E_3$, and let $\overline E_1,\overline E_2,\overline E_3$ be the strict transforms of the lines
$L_1,L_2,L_3$. The $\overline E_i$ are disjoint $(-1)$ curves which can be blown down to obtain another projective plane $\overline \P^2$. Let $H$ (resp. $\overline H$) be the pullback of 
the hyperplane class from $\P^2$ (resp. $\overline \P^2$) to $X$. Then $H^2(X,\Z)$ has two different bases $H,E_1,E_2,E_3$ and $\overline H,\overline E_1,\overline E_2,\overline E_3$, 
which are related
by the formula
$$dH-a_1E_1-a_2E_2-a_3E_3=(2d-a_1-a_2-a_3)\overline H-(d-a_2-a_3)\overline E_1-(d-a_1-a_3)\overline E_2-(d-a_1-a_2)\overline E_3.$$
Note that this description is symmetric under exchanging the role of $H,E_1,E_2,E_3$ and $\overline H, \overline E_1,\overline E_2,\overline E_3$.
Let $c_1\in H^2(X,\Z)$. If $\<c_1,K_X\>$ is even, then it is easy to see that $\overline c_1\equiv c_1\mod 2H^2(X,\Z)$, but if $\<c_1,K_X\>$ is odd, then $\overline c_1\equiv K_X-c_1\mod 2H^2(X,\Z)$.
For a class $L=dH-a_1E_1-a_2E_2-a_3E_3\in H^2(X,\Z)$ we denote $\overline L=d\overline H-a_1\overline E_1-a_2\overline E_2-a_3\overline E_3.$
Then it is clear from the definition that 
$\chi^{X,H}_{c_1}(L)=\chi^{X,\overline H}_{\overline{c}_1}(\overline L)$, and by \lemref{blowindep} we get  $\chi^{X,\overline H}_{\overline{c}_1}(\overline L)\equiv \chi^{X,H}_{\overline c_1}(\overline L).$
If $\sigma$ is a permuation of $\{1,2,3\}$ and we denote $\sigma(L):=dH-a_{\sigma_1}E_1-a_{\sigma_2}E_2-a_{\sigma_3}E_3$,
If $\sigma(L)=L$, then it is clear that $\chi^{X,H}_{c_1}(L)=\chi^{X,H}_{\sigma(c_1)}(L)$.

Now assume  $d=a_1+a_2+a_3$. Then $\overline L=L$, so that 
 $\chi^{X,H}_{c_1}(L)\equiv\chi^{X,H}_{{\overline c_1}}(L).$
 
Assume now $\<c_1,K_X\>$ is odd.
Assuming also \conref{ratconj}, the polynomials $P^X_{c_1,L} \in \Lambda^{-c_1^2}\Z[\Lambda^4]$ mentioned there satisfy 
\begin{enumerate}
\item 
$P^X_{c_1,L}(\Lambda)=P^X_{K_X-c_1,L}(\Lambda)=\Lambda^{L^2+8-K_X^2} P^X_{L+K_X-c_1}(\frac{1}{\Lambda})=\Lambda^{L^2+8-K_X^2} P^X_{L- c_1}(\frac{1}{\Lambda}).$
\item 
Thus if there is a permutation $\sigma$ of $\{1,2,3\}$ with $\sigma(L)=L$ and $c_1\equiv L-\sigma(c_1)\mod 2 H^2(X,\Z)$, or with 
$\sigma(L)=L$ and $c_1\equiv L+K_X-\sigma(c_1)\mod 2 H^2(X,\Z)$. Then we have the symmetries
$$P^X_{c_1,L}(\Lambda)=\Lambda^{L^2+8-K_X^2} P^X_{c_1,L}(\frac{1}{\Lambda})=P^X_{K_X-c_1,L}(\Lambda)=\Lambda^{L^2+8-K_X^2} P^X_{K_X-c_1,L}(\frac{1}{\Lambda}).$$
\end{enumerate}

We check these predictions in a number of cases.
 Let 
 $L_5=5H-2E_1-2E_2-E_3$, $L_6=6H-2E_1-2E_2-2E_3$, $L_7=7H-3E_1-2E_2-2E_3$.
 We find
\begin{align*}
&\chi^{X,H}_{K_X+E_2}(L_5)\equiv\chi^{X,H}_{E_2}(L_5)\equiv\frac{5\Lambda^5+6\Lambda^9+5\Lambda^{13}}{(1-\Lambda^4)^{14}},\\
&\chi^{X,H}_{H}(L_6)\equiv\chi^{X,H}_{E_1+E_2+E_3}(L_6)\equiv\frac{\Lambda^3+18\Lambda^7+45\Lambda^{11}+45\Lambda^{15}+18\Lambda^{19}+\Lambda^{23}}{(1-\Lambda^4)^{19}},\\
&\chi^{X,H}_{K_X-E_3}(L_6)\equiv\chi^{X,H}_{E_3}(L_6)\equiv\frac{8\Lambda^5+26\Lambda^9+60\Lambda^{13}+26\Lambda^{17}+8\Lambda^{21}}{(1-\Lambda^4)^{19}},\\
&\chi^{X,H}_{K_X-E_3}(L_7)\equiv\chi^{X,H}_{E_3}(L_7)\equiv\frac{11\Lambda^5+61\Lambda^9+265\Lambda^{13}+350\Lambda^{17}+265\Lambda^{21}+61\Lambda^{25}+11\Lambda^{29}}{(1-\Lambda^4)^{24}}.
\end{align*}

\section{The invariants of $\P^1\times\P^1$}
In this section we will use the results of the previous section to compute the $K$-theoretic Donaldson invariants of $\P^1\times \P^1$.

\subsection{A structural result}
First we will show analoguously to \thmref{blowrat} that all the generating functions $\chi^{\P^1\times\P^1,\omega}_{c_1}(L,P^r)$ are rational functions.

\begin{Lemma}
Let $c_1\in H^2(\P^1\times\P^1,\Z)$. Let $L$ be a line bundle on  $\P^1\times \P^1$ with $\<L,c_1\>+r$ even.
Let $\omega$ be an ample classes on $\P^1\times\P^1$. Then 
$\chi^{\P^1\times \P^1,\omega}_{c_1}(L,P^r)\equiv \chi^{\P^1\times \P^1F+G}_{c_1}(L,P^{r})$.
\end{Lemma}
\begin{proof}
We write $L=nF+mG$. By exchanging the role of $F$ and $G$ if necessary we can write $\omega=G+\alpha F$, with $1\le \alpha$.
We have to show that there are only finitely many classes $\xi$ of type $(c_1)$ with $\<\xi,(F+G)\>\le 0\le \<\xi,\omega\>$ and $\delta^X_\xi(L,P^r)\ne 0$.
Such a class is of the form $\xi=aG-bF$, with $a,b\in \Z_{>0}$ satisfying  
$$a\le b,\quad \alpha a\ge b, \quad 2ab\le |a(n+2)-b(m+2)|+r+2.$$
This gives $$2ab\le a|n+2|+b|m+2|+r+2\le b|n+m+4|+r+2.$$ Thus we get $a\le \frac{|n+m+4|}{2}+\frac{r}{2}+1$. Therefore $a$ is bounded and by $\alpha a\ge b$ also $b$ is bounded.
Therefore there are only finitely many possible classes $\xi$.
\end{proof}

We use the fact that the blowup $\widetilde \P^2$ of $\P^2$ in two different points is also the blowup of $\P^1\times \P^1$ in a point.
We can identify the classes as follows.
Let $H$ be the hyperplane class on $\P^2$ and let $E_1$, $E_2$ be the exceptional divisors of the double blowup of $\P^2$.
Let $F$, $G$ be the fibres of the two different projections of $\P^1\times \P^1$ to its factors, and let $E$ be the exceptional divisor of the blowup of $\P^1\times \P^1$.
Then on $\widetilde \P^2$ we have the identifications
\begin{align*}
F&=H-E_1, \quad G=H-E_2, \quad E=H-E_1-E_2,\\
H&=F+G-E, \quad E_1=G-E, \quad E_2=F-E.
\end{align*}

\begin{Theorem}\label{P11rat}
Let $c_1\in \{0,F,G,F+G\}$. Let $L$ be a line bundle on  $\P^1\times \P^1$ with $\<L,c_1\>$ even. Let $r\in \Z_{\ge 0}$ with $\<L,c_1\>+r$ even.
There exists a polynomial $p^{\P^1\times \P^1}_{c_1,L,r}(t)$ and an integer $N^{\P^1\times \P^1}_{c_1,L,r}$, such that for all ample classes 
$\omega$ on $\P^1\times \P^1$, we have
$$\chi^{\P^1\times \P^1,\omega}_{c_1}(L,P^r)\equiv \frac{p^{\P^1\times \P^1}_{c_1,L,r}(\Lambda^4)}{\Lambda^{c_1^2}(1-\Lambda^4)^{N^{\P^1\times \P^1}_{c_1,L,r}}}.$$
\end{Theorem}
\begin{proof}
Note that on $\widetilde \P^2$ we have
$F+G=2H-E_1-E_2$. We write $L=nF+mG$, with $n,m\in \Z$. Then on $\widetilde \P^2$ we have
$L=(n+m)H-nE_1-mE_2$. 
By \thmref{nblow} we have therefore
\begin{align*}
\chi^{\widetilde \P^2,H}_{0}(nF+mG,P^r)&=\chi^{\P^2,H}_{0}\big((n+m)H, P^r \cdot R_{n+1}(P,\Lambda)R_{m+1}(P,\Lambda)\big),\\
\chi^{\widetilde \P^2,H}_{F}(nF+mG,P^r)&=\chi^{\P^2,H}_{H}\big((n+m)H, P^r \cdot S_{n+1}(P,\Lambda)R_{m+1}(P,\Lambda)\big),\\
\chi^{\widetilde \P^2,H}_{G}(nF+mG,P^r)&=\chi^{\P^2,H}_{H}\big((n+m)H, P^r \cdot R_{n+1}(P,\Lambda)S_{m+1}(P,\Lambda)\big),\\
\chi^{\widetilde \P^2,H}_{F+G}(nF+mG,P^r)&=\chi^{\P^2,H}_{0}\big((n+m)H, P^r \cdot S_{n+1}(P,\Lambda)S_{m+1}(P,\Lambda)\big).
\end{align*}
As $R_{n}(P,\Lambda)\in \Z[P^2,\Lambda^4]$, $S_{n}(P,\Lambda)\in \Lambda\Z[P,\Lambda^4]$, we see by \thmref{P2rat1} that
for $c_1=0,F,G,F+G$ we can write
$$
\chi^{\widetilde \P^2,H}_{c_1}(nF+mG,P^r)\equiv \frac{p^{\P^1\times \P^1}_{c_1,nF+mG,r}(\Lambda^4)}{\Lambda^{c_1^2}(1-\Lambda^4)^{N^{\P^1\times \P^1}_{c_1,nF+mG,r}}},$$
with  $p^{\P^1\times \P^1}_{c_1,nF+mG,r}\in \Q[t]$ and $N^{\P^1\times \P^1}_{c_1,nF+mG,r}\in\Z_{\ge 0}.$
As on $\widetilde \P^2$  we have $F+G=2H-E_1-E_2$, we get by \thmref{blowrat} that again for $c_1=0,F,G,F+G$ we have
$\chi^{\widetilde \P^2,F+G}_{c_1}(nF+mG,P^r)\equiv\chi^{\widetilde \P^2,H}_{c_1}(nF+mG,P^r)$. Finally by the blowdown formula 
\thmref{nblow} we have 
$\chi^{\widetilde \P^2,F+G}_{c_1}(nF+mG,P^r)=\chi^{\P^1\times \P^1,F+G}_{c_1}(nF+mG,P^r).$
\end{proof}

\subsection{Computations for $L=d(F+G)$} 
We will compute $\chi^{\P^1\times \P^1,F+G}_{c_1}(d(F+G))$ for $d\le 7$  and $c_1=0$, $F$, $G$, $F+G$. Obviously by symmetry
$\chi^{\P^1\times \P^1,F+G}_{G}(d(F+G))=\chi^{\P^1\times \P^1,F+G}_{F}(d(F+G))$, and furthermore we have just seen that 
$\chi^{\P^1\times \P^1,\omega}_{c_1}(d(F+G))\equiv\chi^{\P^1\times \P^1,F+G}_{c_1}(d(F+G))$ for any ample class $\omega$ on $\P^1\times\P^1$.
We use a different strategy than in the proof of \thmref{P11rat}, which is computationally more tractable, and allows us to compute 
$\chi^{\P^1\times \P^1,F+G}_{c_1}(d(F+G))$ for $d\le 7$, using only the $\chi^{\P^2,H}_H(nH,P^r)$, for $1\le n\le 8$, $0\le r\le 16$ already computed.

{\bf Step 1.}
By $F=H-E_1$, $G=H-E_2$, $E=H-E_1-E_2$, and thus $d(F+G)-dE=dH$,  and \thmref{nblow} we have
\begin{align*}
\chi^{\widetilde \P^2,H}_E(d(F+G)-dE,P^r)&=\chi^{\widetilde \P^2,H}_{H-E_1-E_2}(dH,P^r)=\Lambda^2\chi^{\P^2,H}_{H}(dH,P^r),\\
\chi^{\widetilde \P^2,H}_E(d(F+G)-(d-1)E,P^r)&=\chi^{\widetilde \P^2,H}_{H-E_1-E_2}((d+1)H-E_1-E_2,P^{r})\\&=\Lambda^2\chi^{\P^2,H}_{H}((d+1)H,P^{r+2}),\\
\chi^{\widetilde \P^2,H}_F(d(F+G)-dE,P^r)&=\chi^{\P^2,H}_{H-E_1}(dH,P^r)=\Lambda\chi^{\P^2,H}_{H}(dH,P^r),\\
\chi^{\widetilde \P^2,H}_F(d(F+G)-(d-1)E,P^r)&=\chi^{\widetilde \P^2,H}_{H-E_1}((d+1)H-E_1-E_2,P^{r})\\&=\Lambda(1-\Lambda^4)\chi^{\P^2,H}_{H}((d+1)H,P^{r+1}),\\
\chi^{\widetilde \P^2,H}_{F+G-E}(d(F+G)-dE,P^r)&=\chi^{\P^2,H}_{H}(dH,P^r),\\
\chi^{\widetilde \P^2,H}_{F+G-E}(d(F+G)-(d-1)E,P^r)&=\chi^{\widetilde \P^2,H}_{H}((d+1)H-E_1-E_2,P^{r})\\&=(1-\Lambda^4)^2\chi^{\P^2,H}_{H}((d+1)H,P^{r}),
\end{align*}
where we have used that $S_1(P,\Lambda)=\Lambda$, $S_2(P,\Lambda)=P\Lambda$ and   $R_2(P,\Lambda)=(1-\Lambda^4)$.

The $\chi^{\P^2,H}_{H}(nH,P^{s})$ have been computed for $n\le 8$ and $s\le 16$. In the tables above they are only listed for $n\le 7$ and only up to adding a Laurent polynomial in 
$\Lambda$, so as to give them a particularly simple form, but they have been computed precisely. 
Thus in the range $d\le 7$ and $r\le 14$, the  all the invariants on the left hand side of the formulas of Step 1  have been computed. 

{\bf Step 2.} For $d\le 7$ we compute 
\begin{align*}
\tag{1}\chi^{\widetilde \P^2,F+G}_{E}(d(F+G)-dE,P^r)&=\Lambda^2\chi^{\P^2,H}_{H}(dH,P^r)+\sum_{\xi} \delta_\xi^{\widetilde \P^2}(dH,P^r),\\
\chi^{\widetilde \P^2,F+G}_{E}(d(F+G)-(d-1)E,P^r)&=\Lambda^2\chi^{\P^2,H}_{H}((d+1)H,P^{r+2})\\
&\quad+\sum_{\xi} \delta_\xi^{\widetilde \P^2}((d+1)H-E_1-E_2,P^r),\\
\tag{2}\chi^{\widetilde \P^2,F+G}_{F}(d(F+G)-dE,P^r)&=\Lambda\chi^{\P^2,H}_{H}(dH,P^{r})+\sum_{\xi} \delta_\xi^{\widetilde \P^2}(dH,P^r),\\
\chi^{\widetilde \P^2,F+G}_{F}(d(F+G)-(d-1)E,P^r)&=\Lambda(1-\Lambda^4)\chi^{\P^2,H}_{H}((d+1)H,P^{r+1})\\
&\quad+\sum_{\xi} \delta_\xi^{\widetilde \P^2}((d+1)H-E_1-E_2,P^r),\\
\tag{3}\chi^{\widetilde \P^2,F+G}_{F+G-E}(d(F+G)-dE,P^r)&=\chi^{\P^2,H}_{H}(dH,P^{r})+\sum_{\xi} \delta_\xi^{\widetilde \P^2}(dH,P^r),\\
\chi^{\widetilde \P^2,F+G}_{F+G-E}(d(F+G)-(d-1)E,P^r)&=(1-\Lambda^4)^2\chi^{\P^2,H}_{H}((d+1)H,P^{r})\\
&\quad+\sum_{\xi} \delta_\xi^{\widetilde \P^2}((d+1)H-E_1-E_2,P^r).\\
\end{align*}
Here the sums are over all classes $\xi\in H^2(\widetilde \P^2,\Z)$ with $\<H,\xi\>\le 0\le\<(2H-E_1-E_2),\xi\>$ (but at least one of the inequalities is strict) and $\delta_\xi^{\widetilde \P^2}\ne \emptyset$, and the summand
is $\delta_\xi^{\widetilde \P^2}(dH,P^r)$,  $\delta_\xi^{\widetilde \P^2}((d+1)H-E_1-E_2,P^r)$, if both inequalities are strict and
$\frac{1}{2}\delta_\xi^{\widetilde \P^2}(dH,P^r)$,  $\frac{1}{2}\delta_\xi^{\widetilde \P^2}((d+1)H-E_1-E_2,P^r)$ if one of them is an equality.
(Note that we can exclude the $\xi$ with $\<\xi,H\>=0=\<\xi,2H-E_1-E_2\>$, because with $\xi$ also $-\xi$ will fulfil this property and 
$\delta_\xi^{\widetilde \P^2}(L,P^r)=-\delta_{-\xi}^{\widetilde \P^2}(L,P^r)$.)

In (1) these are classes of type $(E=H-E_1-E_2)$, in (2) of type $(F=H-E_1)$ and in (3) of type $(F+G-E=H)$.
By \lemref{blowindep} there are finitely many such classes.
In fact in the notations of the proof of \lemref{blowindep} we have 
$$n=2, \ \epsilon=\frac{1}{2}, \ \delta=\frac{1}{4}, |m_1+1|=|m_2+1|=1.$$
 Thus  if $\xi=aH-b_1E_1-b_2E_2$ is such a class, then we get by \eqref{blowbound} that 
$|b_1|+|b_2|\le 4(|d+3| +r+4)$ and $0<a\le \frac{1}{2}(|b_1|+|b_2|)$.
For all $\xi$ satisfying these bounds it is first checked whether indeed the criterion of \lemref{vanwall}(2) for the non-vanishing of the wallcrossing term $\delta_\xi^{\widetilde \P^2}(dH,P^r)$,  $\delta_\xi^{\widetilde \P^2}((d+1)H-E_1-E_2,P^r)$ is fulfilled, if yes we compute the wallcrossing term, we again use that by \lemref{vanwall}
$\delta_{\xi,d}^X(L,P^r)=0$ unless $d\le a_{\xi,L,X}:=\xi^2+2|\<\xi,L-K_X\>|+2r+4$.
Thus, also using \remref{qlevel}  it is enough evaluate the formula of \defref{wallcrossterm} modulo $q^{a_{\xi,L,X}}$ and $\Lambda^{a_{\xi,L,X}}$.
This is again a finite evaluation.

{\bf Step 3.}
By \thmref{nblow} we have
\begin{align*}
\chi^{\widetilde \P^2,F+G}_{E}(d(F+G)-dE,P^r)&=\chi^{ \P^1\times \P^1,F+G}_{0}(d(F+G),S_{d+1}(P,\Lambda) P^r),\\
\chi^{\widetilde \P^2,F+G}_{E}(d(F+G)-(d-1)E,P^r)&=\chi^{\P^1\times \P^1,F+G}_{0}(d(F+G),S_{d}(P,\Lambda) P^r),\\
\chi^{\widetilde \P^2,F+G}_{F}(d(F+G)-dE,P^r)&=\chi^{\P^1\times \P^1,F+G}_{F}(d(F+G),R_{d+1}(P,\Lambda)P^r),\\
\chi^{\widetilde \P^2,F+G}_{F}(d(F+G)-(d-1)E,P^r)&=\chi^{\P^1\times \P^1,F+G}_{F}(d(F+G),R_{d}(P,\Lambda)P^r),\\
\chi^{\widetilde \P^2,F+G}_{F+G-E}(d(F+G)-dE,P^r)&=\chi^{\P^1\times \P^1,F+G}_{F+G}(d(F+G),S_{d+1}(P,\Lambda)P^r),\\
\chi^{\widetilde \P^2,F+G}_{F+G-E}(d(F+G)-(d-1)E,P^r)&=\chi^{\P^1\times \P^1,F+G}_{F+G}(d(F+G),S_{d}(P,\Lambda)P^r).
\end{align*}
By \remref{npol} there exist polynomials $f_d\in \Q[x,\Lambda^4]$, $g_d\in \Q[x,\Lambda^4]$ with 
$f_d S_d(x,\lambda)+g_d S_{d+1}(x,\lambda)=\lambda(1-\lambda^4)^{N_d}$, 
and 
$h_d\in \Q[x,\Lambda^4]$, $l_d\in \Q[x,\Lambda^4]$ with 
$h_d R_d(x,\lambda)+l_d R_{d+1}(x,\lambda)=(1-\lambda^4)^{M_d}$.
For $d\le 7$ we see that $f_d$, $h_d$ are  polynomials in $x$ of degree at most $14$,  and $g_d$, $l_d$ are polynomials in $x$ of degree at most $11$.

Thus we get
\begin{align*}
\chi^{ \P^1\times \P^1,F+G}_{0}(d(F+G))&=\frac{1}{\Lambda(1-\Lambda^4)^{M_d}}\Big(\chi^{\widetilde \P^2,F+G}_{E}(d(F+G)-(d-1)E,f_d(P,\Lambda))\\
&\qquad+\chi^{\widetilde \P^2,F+G}_{E}(d(F+G)-dE,g_d(P,\Lambda))\Big),\\
\chi^{ \P^1\times \P^1,F+G}_{F}(d(F+G))&=\frac{1}{(1-\Lambda^4)^{N_d}}\Big(\chi^{\widetilde \P^2,F+G}_{F}(d(F+G)-(d-1)E,h_d(P,\Lambda))\\
&\qquad+\chi^{\widetilde \P^2,F+G}_{F}(d(F+G)-dE,l_d(P,\Lambda))\Big),\\
\chi^{ \P^1\times \P^1,F+G}_{F+G}(d(F+G))&=\frac{1}{\Lambda(1-\Lambda^4)^{M_d}}\Big(\chi^{\widetilde \P^2,F+G}_{F+G-E}(d(F+G)-(d-1)E,f_d(P,\Lambda))\\
&\qquad+\chi^{\widetilde \P^2,F+G}_{F+G-E}(d(F+G)-dE,g_d(P,\Lambda))\Big).
\end{align*}

All these computations are carried out with a Pari program. Finally we arrive at the following result.

\begin{Theorem}
With the notation of \thmref{P11gen}.
\begin{enumerate}
\item  $\displaystyle{\chi^{\P^1\times\P^1,F+G}_{0}(dF+dG)=\frac{p_d(\Lambda^4)}{(1-\Lambda^4)^{(d+1)^2}}-1-(d^2+4d+5)\Lambda^4}$ 
for $1\le d\le 7$.
\item $\displaystyle{\chi^{\P^1\times\P^1,F+G}_{F+G}(dF+dG)=\frac{q_d(\Lambda^4)}{(1-\Lambda^4)^{(d+1)^2}}-\Lambda^2}$ for $1\le d\le 7$.
\item $\displaystyle{\chi^{\P^1\times\P^1,F+G}_{F}(dF+dG)=\frac{r_d(\Lambda^4)}{(1-\Lambda^4)^{(d+1)^2}}}$ for $d=2,4,6$.
\end{enumerate}
\end{Theorem}
\thmref{P11gen} follows from this and \propref{highvan}.

\end{document}